\documentclass[reqno]{amsart}

\usepackage{general,strings,txfonts,ifthen}
\usepackage[switch*,pagewise]{lineno}
 
\usepackage[bookmarks, colorlinks=true, pdfstartview=FitV, linkcolor=blue, citecolor=blue, urlcolor=blue]{hyperref}

\newcommand{\bd}{\mathrm{bd}\,} 
\newcommand{\mydot}{\,\cdot\,}

\newcommand{\linenopax}{\par\vspace{-2ex}}  
\renewcommand{\ii}{\mathbbm{i}}  
\newcommand{\pole}{\ensuremath{\gw}\xspace}   
\newcommand{\simm}{\gY}   
\newcommand{\simtset}{\{\simt_1,\ldots,\simt_N\}}
\newcommand{\id}{\ensuremath{\mathbb{I}}\xspace}     
\newcommand{\scale}{\ensuremath{\ell}\xspace}     
\newcommand{\inradius}{\ensuremath{\gr}\xspace}     
\newcommand{\abscissa}{\ensuremath{D}\xspace}  

\newcommand{\integral}{{\scalebox{0.6}{\ensuremath{\int{}}}}}     
\newcommand{\Words}{\sW}     
\renewcommand{\languidOrder}{\ensuremath{\gg}\xspace} 

\newcommand{\ghull}{\ensuremath{K}\xspace}
\newcommand{\he}{\rm head}
\newcommand{\ta}{\rm tail}

\newcommand{\gengen}{\ensuremath{U}\xspace} 
\newcommand{\genstr}{\ensuremath{\hat \ell}\xspace} 

\renewcommand{\gzT}[1][{}]{\ensuremath{\scalebox{1.1}{\gz}_{{\scalebox{0.7}{$\negsp[3]$\tiling}}#1}}\xspace} 
\renewcommand{\gzL}{\ensuremath{\scalebox{1.1}{\gz}_{\scalebox{0.7}{$\negsp[5]$\sL}}}\xspace} 
  \renewcommand{\head}{\negsp[4], \scalebox{0.6}{\rm{head}}}
  \renewcommand{\tail}{\negsp[4], \scalebox{0.6}{\rm{tail}}}

\newcommand{\clean}{false}  
\newcommand{\version}[2]{\ifthenelse{\equal{\clean}{true}}{#1}{{\footnotesize #2}}}

  \renewcommand\marginpar[1]{}
    
\numberwithin{equation}{section}
\numberwithin{theorem}{section}
\numberwithin{figure}{section}

\begin{document}
  \linenumbers

  \title[Pointwise tube formulas for fractal sprays and self-similar tilings]
    {Pointwise tube formulas for fractal sprays and \\ self-similar tilings with arbitrary generators}

  \author{Michel L. Lapidus}
  \address{University of California, Department of Mathematics, Riverside, CA 92521-0135 USA}
  \email{\href{mailto:lapidus@math.ucr.edu}{lapidus@math.ucr.edu}}

  \author{Erin P. J. Pearse}
  \address{University of Oklahoma, Department of Mathematics, Norman, OK 73019-0315 USA}
  \email{\href{mailto:ep@ou.edu}{ep@ou.edu}}

  \author{Steffen Winter}
  \address{Karlsruhe Institute of Technology, Department of Mathematics, 76128 Karlsruhe, Germany}
  \email{\href{mailto:steffen.winter@kit.edu}{steffen.winter@kit.edu}}

  \begin{abstract}
    In a previous paper by the first two authors, a tube formula for fractal sprays was obtained which also applies to a certain class of self-similar fractals. The proof of this formula uses distributional techniques and requires fairly strong conditions on the geometry of the tiling (specifically, the inner tube formula for each generator of the fractal spray is required to be polynomial). Now we extend and strengthen the tube formula by removing the conditions on the geometry of the generators, and also by giving a proof which holds pointwise, rather than distributionally.
    Hence, our results for fractal sprays extend to higher dimensions the pointwise tube formula for (1-dimensional) fractal strings obtained earlier by Lapidus and van Frankenhuijsen.  
    
    Our pointwise tube formulas are expressed as a sum of the residues of the ``tubular zeta function'' of the fractal spray in $\mathbb{R}^d$. This sum ranges over the complex dimensions of the spray, that is, over the poles of the geometric zeta function of the underlying fractal string and the integers $0,1,\dots,d$. The resulting ``fractal tube formulas'' are applied to the important special case of self-similar tilings, but are also illustrated in other geometrically natural situations. Our tube formulas may also be seen as fractal analogues of the classical Steiner formula. 
  \end{abstract}

  \date{\textbf{\today}}
  \keywords{Complex dimensions, tube formula, Steiner formula, scaling and integer dimensions, zeta functions, scaling and tubular zeta functions, inradius, self-similar tiling, curvatures, generating function, fractal tube formula, fractal string, fractal spray, lattice and nonlattice, Minkowski measurability and content, fractal curvatures.}

  \subjclass[2010]{
    Primary: 11M41, 28A12, 28A75, 28A80, 52A39, 52C07,
    Secondary: 11M36, 28A78, 28D20, 42A16, 42A75, 52A20, 52A38
    }
    
  \thanks{The work of MLL was partially supported by the US National Science Foundation under the research grant DMS-0707524. The work of EPJP was partially supported by the University of Iowa Department of Mathematics NSF VIGRE grant DMS-0602242. The work of SW was supported by a DAAD Postdoctoral grant and a travel grant from Cornell University.}

\maketitle

\setcounter{tocdepth}{1} {\small \tableofcontents}

\allowdisplaybreaks


    
\section{Introduction}
\label{sec:intro}

Our main results are tube formulas for fractal sprays in $\bR^d$, the higher-dimensional analogues of (geometric) fractal strings in \bR. In \cite{FGCD}, a fractal string is defined to be a bounded
open subset of the real line \bR; see also \cite{{Lap:FD},{Lap:UAB},{Lap:Dundee},LaMa,LaPo1,LaPo2,La-vF:CDandDiophantine,La-vF:Fractality,KTF}. Here, we emphasize the
interpretation of a fractal string as a sequence of positive
numbers, rather than as a collection of open intervals in the
geometric sense. 

\begin{defn}[Fractal string]\label{def:fractal_string}
  A \emph{fractal string} $\sL = \{\scale_j\}_{j=1}^\iy$ is a
  nonincreasing sequence of positive real numbers $\scale_j > 0$
  satisfying $\lim_{j \to \iy} \scale_j = 0$. 
\end{defn}
  
In particular, we do not assume that the sum
of the lengths of a fractal string is finite. Hence, \sL might not have a geometric realization as a bounded open subset of \bR, in the sense of \cite[Def. 1.2]{FGCD}.

\begin{defn}[Fractal spray]\label{def:fractal-spray}\label{def:generator}
  A \emph{fractal spray} \tiling defined on a bounded open set
  $\gengen \ci \bRd$ via the fractal string $\sL = \{\scale_j\}_{j=1}^\iy$ is a
  collection of disjoint bounded open sets \smash{$\{\gengen^j\}_{j=1}^\iy$} in
  \bRd such that, for each set $\gengen^j$ with $j \geq 1$, there exists a similarity
  transformation $\simm_j$ of \bRd with scaling ratio $\scale_j$ and satisfying
  $\gengen^j = \simm_j(\gengen)$. The spray \tiling is said to be
  \emph{scaled} by the fractal string \sL, and the connected components of the set \gengen are called the \emph{generators} of the fractal spray. The generators are denoted by $\gen_q$, where $q$ ranges over some finite or countable index set.
  When there is only one generator, we denote it by \gen instead of $\gen_1$.
\end{defn}

Hence, a fractal spray on the generator \gen is just a collection of disjoint
scaled copies of \gen such that the scaling ratios form a
fractal string (in the sense of Definition~\ref{def:fractal_string}), just as in \cite{LaPo2} and \cite[Section~1.4]{FGCD}.
Note that since \gengen is bounded and open, each generator 
is a bounded open and connected subset of \bRd, and hence there can be at most  countably many generators. We always assume in the sequel that \tiling has \emph{finitely many} generators $\{\gen_q\}_{q=1}^Q$, which allows us to study only the case of a \emph{single} generator \gen (see the explanation at the start of Section~\ref{sec:pw-tube-formula}, and the discussion just following \eqref{eqn:compatibility-condition}).
 
Note that fractal strings in the geometric sense may be viewed as fractal sprays in \bR generated by a bounded open interval $\gen$; indeed, they are disjoint unions of a sequence of bounded open intervals. Therefore, geometric fractal strings are included in the setting of fractal sprays. 
An important subclass of fractal sprays is formed by self-similar tilings, which appear naturally in connection with self-similar sets and are higher-dimensional generalizations of the (geometric) self-similar strings studied in \cite{La-vF:CDandDiophantine,La-vF:Fractality,FGCD}; see Section~\ref{sec:Self-similar-tilings}.

In the classical literature, the \emph{\ge-parallel set} (or \emph{\ge-neighborhood}) of a bounded set $A \ci \bRd$ is the set of points within (Euclidean) distance \ge of $A$ (see \eqref{eqn:ext-nbd-of-A}), and a \emph{tube formula} for $A$ is an explicit expression for the volume of the \ge-parallel set of $A$, viewed as a function of \ge; see Section~\ref{sec:motivation}. In this paper, we make use of the following ``inner'' analogues of these notions: 

For $\ge>0$, the
  \emph{inner \ge-parallel set} (or \emph{inner \ge-neighborhood}) %
  of a bounded open set $A \ci \bRd$ is the set
  \linenopax
  \begin{equation}\label{eqn:def:inner-parallel-set}
    A_{-\ge}:= \{ x \in A \suth \dist(x,A^c) \leq \ge\},
  \end{equation}
and an \emph{(inner) tube formula} for $A$ is an expression giving
 the volume  $V(A,\ge) := \gl_d(A_{-\ge})$ (i.e., the $d$-dimensional Lebesgue measure) of the set $A_{-\ge}$ as a function of $\ge\in(0,\infty)$. 
  
 Similarly, by a tube formula for a fractal spray $\tiling=\{\gen^j\}_{j=1}^\iy$, we will simply understand an expression $V(\tiling,\ge)$ for the volume of the inner $\ge$-parallel set $T_{-\ge}$ of the union set $T := \bigcup_{j=1}^\iy \gen^j$ of the components $\gen^j$ as a function of $\ge$; that is,  
  \linenopax
  \begin{equation}\label{eqn:def:spray-tube-formula}
    V(\tiling,\ge)
    := \gl_d(T_{-\ge})
    = \sum_{j=1}^\iy \gl_d((\gen^j)_{-\ge})
    = \sum_{j=1}^\iy V(\gen^j,\ge).
  \end{equation}

Our main results in this paper are tube formulas for fractal sprays in $\bR^d$, which are given in Theorem~\ref{thm:pointwise-tube-formula} and Corollary~\ref{thm:fractal-spray-tube-formula-monophase}, and later specialized to the class of self-similar tilings in Theorems~\ref{thm:ptwise-result-self-similar-case} and \ref{thm:Pointwise-tube-formula-with-error-term}, along with their respective corollaries, the `fractal tube formulas' obtained in Corollaries \ref{thm:ptwise-result-self-similar-case-simplified}--\ref{thm:ptwise-result-self-similar-case-monophase} and Corollary~\ref{thm:Fractal-tube-formula-with-error-term}.

These tube formulas express the volume $V(\sT,\ge)$ of the inner $\ge$-parallel sets $T_{-\ge}$ of the given fractal spray $\sT$ in $\bR^d$ with $d\geq 1$, as a (typically infinite) sum over the set $\sD_\sT\ci \bC$ of \emph{complex dimensions} of $\sT$. These complex dimensions are defined as the poles of the \emph{tubular zeta function} $\gzT=\gzT(\ge,s)$ associated with the spray $\sT$ (see Definition~\ref{def:volume-zeta-fn} and Proposition~\ref{prop:poles}), and each summand is equal to the residue%
  \footnote{We denote by $\res{f(s)}$ the residue of a meromorphic function $f$ at an isolated singularity \gw. Recall that this is the unique value \ga such that $(f(s)-\ga)/(s-\gw)$ has an analytic antiderivative in a punctured disk $\{s \suth 0 < |s-\gw| < \gd\}$; equivalently, the residue is the coefficient $a_{-1}$ in the Laurent expansion of $f$.}
  of $\gzT$ at the corresponding complex dimension.  
  Roughly speaking, we show that for all sufficiently small $\ge>0$,
\linenopax
  \begin{equation} \label{eqn:pointwise-tube-formula-prelim}
    V(\tiling,\ge)
    = \sum_{\pole \in \sD_\sT} \res[s=\pole]{\gzT(\ge,s)} \,, 
  \end{equation}
where $\gzT$ is suitably defined in terms of the scaling zeta function $\gzL$ of the underlying fractal string $\sL$ and the geometry of the generator $G$ of the spray. Here, $\gzL(s)$ is the meromorphic continuation of the Dirichlet series $\sum_{j=1}^\infty \ell_j^s$, initially defined for $\Re(s)$ sufficiently large; see Definition~\ref{def:scaling-zeta-function}. Moreover, the set $\sD_\sT=\sD_\sL\cup\{0,1,\ldots,d\}$ of complex dimensions of $\sT$ consists of the \emph{integer dimensions} $0,1, \ldots, d$ and the \emph{scaling dimensions}, which are the complex dimensions of the associated fractal string $\sL$ (defined as the poles of the scaling zeta function $\gzL$); see Definitions~\ref{def:complex-dimn} and \ref{def:scaling-zeta-function}. 

Depending on the assumptions, our pointwise tube formulas are either exact (as in \eqref{eqn:pointwise-tube-formula-prelim} just above) or else contain an error term, which can be estimated explicitly as $\ge$ tends to zero. In the latter case, the aforementioned sum of residues ranges only over the `visible' complex dimensions of $\sT$; i.e., those complex dimensions lying in a \emph{window} $W$, the region to the right of a (suitably chosen) vertical
contour S called the \emph{screen}; see Section~\ref{sec:zeta-fns} for the precise definitions.

The fractal tube formulas obtained in this paper extend previous results 
in two ways. First, we extend the scope of the tube formulas of \cite{TFCD} to fractal sprays whose generators may be arbitrary bounded open sets. Second, we give a pointwise version of the tube formula obtained in a distributional sense in \cite{TFCD}. This generalizes and clarifies the results previously obtained in \cite{TFCD,TFSST,GeometryOfSST,SST,FGCD,KTF}. 

Furthermore, formula \eqref{eqn:pointwise-tube-formula-prelim} (given precisely in Theorem~\ref{thm:pointwise-tube-formula}) directly extends \cite[Thm.~8.7]{FGCD} (the pointwise tube formula for fractal strings) to higher dimensions \emph{and implies as a corollary that all self-similar tilings have a fractal tube formula}; see Section~4.1 and Section~5 for more details. 
The tube formulas with error term also allow us to obtain information concerning the Minkowski measurability (and Minkowski content, when it exists) of fractal sprays and self-similar tilings under certain conditions; this is taken up in \cite{MMFS}.

The tube formulas in this paper may also be interpreted as fractal analogues of the Steiner formula and its generalizations (see Section~\ref{sec:motivation} below, in particular \eqref{eq:Steiner-formula-ext}). 
Steiner-type formulas express the volume of the $\ge$-parallel sets of a given set $A$ as a polynomial in $\ge$ with coefficients that just depend on $A$. Under the additional assumption that the complex dimensions (i.e., the poles of $\gzT$) are simple, our tube formula can be written in the following way (see Corollary~\ref{thm:ptwise-result-self-similar-case-simplified}):
\linenopax
  \begin{align} \label{eqn:self-similar-pointwise-tube-formula-simplified-intro}
    V(\tiling,\ge)
    &= \sum_{\pole\in \sD_{\sL}} c_\pole \ge^{d-\pole} 
     + \sum_{k=0}^{d} \left(c_k + e_k(\ge)\right) \ge^{d-k},
  \end{align}
  where the coefficients $c_\pole$, $c_k$ and $e_k(\ge)$ are given (in \eqref{eqn:self-similar-pointwise-tube-formula-coefficients-cw} -- \eqref{eqn:self-similar-pointwise-tube-formula-coefficients-ek}) in terms of the residues of $\gzL(s)$ and the geometry of the generator \gen. If, in addition, the tube formula of the generator \gen is a polynomial, then the coefficients $e_k(\ge)$ disappear (see Corollary~\ref{thm:ptwise-result-self-similar-case-monophase}) and the remaining coefficients are completely independent of $\ge$, just as the coefficients in the Steiner formula. (Compare the ``fractal power series'' in formula \eqref{eqn:ptwise-result-self-similar-case-monophase} to the polynomial in \eqref{eq:Steiner-formula-ext}.) 

This paper is part of the program of the present authors to develop a fractal notion of curvature in terms of complex dimensions, and to relate it to other notions of curvature, especially as developed in \cite{Steffen:thesis, Steffen:EulerChar}. 

We note that related questions are also being concurrently studied by other researchers \cite{Demir2}. Recently, some tube formulas extending aspects of \cite{TFSST,TFCD} have been obtained for tilings associated to graph-directed iterated function systems in \cite{Demir}.

For our purposes, the precise embedding of \tiling into \bRd is not  important and the mapping $\simm_j$ associated to $\gen^j$ is not emphasized. Due to the disjointness of the sets $\gen^j$ in Definition~\ref{def:fractal-spray}, the tube formulas require only those properties of fractal sprays which depend either on the geometry of the generator \gen or on the scaling ratios $\scale_j$.

In particular, for the generator $\gen$, we will require that the inner parallel volume of \gen admit a \emph{Steiner-like} formula (Definition~\ref{def:Steiner-like}); that is, it can be represented as a `polynomial' in $\ge$ where the coefficients are allowed to depend on $\ge$. The Steiner-like condition should not be viewed as a restriction on the class of allowed generators $\gen$ but as a choice of the representation of its inner parallel volume. In particular, Steiner-like representations are not unique.
For the fractal string \sL, we assume that it is \emph{languid} or \emph{strongly languid} (Definition~\ref{def:languid} or Definition~\ref{def:strongly-languid}), which is similar to the assumptions made in previous tube formula results. In the case of self-similar tilings, these languidity assumptions are always satisfied. We describe these conditions in detail in the following sections.

\begin{remark}\label{rem:renormalize}
  Without loss of generality, and in contrast to \cite{TFSST, TFCD, SST, GeometryOfSST}, we make a normalization assumption on the fractal string, for the remainder of the paper:
  \[\scale_1=1.\]  
  This assumption imposes no restrictions on the class of fractal sprays, but will simplify the exposition greatly. It amounts to choosing the largest connected set in the spray as the generator (or one of them, if there is not a unique largest set). 
  Also, instead of thinking of the numbers $\scale_j$ as distances (as in \cite{FGCD}, where the terms in a fractal string represent usually lengths of subintervals of \bR), we think of them as scales or scaling ratios. Thus, $\scale_1$ is the scaling factor of the identity mapping $\id:\bRd \to \bRd$, in accord with the original definition of fractal sprays given in \cite{LaPo2} (see also \cite{FGCD}) and the interpretation in terms of self-similar tilings discussed in Section~\ref{sec:Self-similar-tilings} and in \cite{TFCD,GeometryOfSST,SST}.
\end{remark}

\subsection{Tube formulas and classical geometry}
\label{sec:motivation}

To motivate our theorems, we give a brief description of tube formulas in geometry.  
Such formulas have myriad applications in convex, integral and differential geometry and have roots in the results of Steiner \cite{Steiner} (when $A$ is convex) and Weyl \cite{Weyl} (when $A$ is a smooth submanifold). For connections to convex and integral geometry, see \cite{KlainRota, Schneider}, and for connections to differential geometry, see \cite{BergerGostiaux,Gray}. 
For a bounded set $A\ci\bRd$ and $\ge\geq 0$, we denote the \emph{\ge-neighborhood} (or \emph{\ge-parallel set}) of $A$ by
\linenopax
\begin{align} \label{eqn:ext-nbd-of-A}
  A_\ge := \{x \in \bRd \less A \suth \dist(x,A) \leq \ge\}.
\end{align}
Sometimes $A_\ge$ is referred to as a ``collar'' in the literature. Note that some authors include the set $A$ in $A_\ge$, but we have instead excluded $A$ from $A_\ge$. In particular, $A_\ge$ is not a neighborhood of $A$ in the topological sense. 

The \emph{Steiner formula} is a foundational result of convex geometry which states that the tube formula of any compact convex subset of \bRd is a polynomial in \ge. 
\begin{theorem}[Steiner formula]
  \label{thm:Steiners-formula-basic}
  If $K \ci \bRd$ is convex and compact, then the $d$-dimensional volume of $K_\ge$ is given by
  \linenopax
  \begin{align}\label{eq:Steiner-formula-ext}
    \gl_d(K_\ge) = \sum_{k=0}^{d-1} \ge^{d-k} \ga_{d-k} V_k(K),
  \end{align}
  where the coefficients $V_k(K)$ depend only on the set $K$, 
  and $\ga_{j}$ is the volume of the unit ball in $\bR^j$. 
\end{theorem}

Note that formula \eqref{eq:Steiner-formula-ext} can be extended to 
  \linenopax
  \begin{align}\label{eq:Steiner-formula}
    \gl_d(K \cup K_\ge) = \sum_{k=0}^{d} \ge^{d-k} \ga_{d-k} V_k(K),
  \end{align}
  where $V_d(K):=\gl_d(G)$ is the volume of $K$. 
The coefficients $V_0(K),\ldots, V_{d}(K)$ are called \emph{intrinsic volumes} or \emph{Minkowski functionals} of $K$. 
They form a system of important geometric invariants which is, in a sense, complete. 
Some of them have a simple direct interpretation. In particular, $V_d$ is the volume of $K$ and $V_{d-1}$ is half the surface area of its boundary (provided $K$ has interior points); furthermore, $V_1$ is, up to a normalization constant, the \emph{mean width} of $K$, while $V_0$ is its \emph{Euler characteristic}. For nonempty convex sets $K$, $V_0$ is always equal to $1$. See \cite[Section~4.2]{Schneider} for further details.

When the set $K$ is sufficiently regular (i.e., when its boundary is a $C^2$ surface), the coefficients $V_k(K)$ can be given in terms of curvature tensors, and the Steiner formula coincides with the tube formula obtained by Weyl in \cite{Weyl} for smooth submanifolds of $\bR^d$ without boundary. In \cite{Federer}, Federer unified both approaches and extended these results to \emph{sets of positive reach} through the introduction of curvature measures $C_0(K,\cdot),\dots,C_d(K,\cdot)$ and a localization of the Steiner formula. A set $K \ci \bRd$ is said to have \emph{positive reach} iff there is some $\gd>0$ such that any point $x \in \bRd$ with $\dist(x,K)<\gd$ has a unique metric projection $p_K(x)$ to $K$; i.e., there is a unique point $p_K(x)$ in $K$ minimizing $dist(x,K)$. The supremum of all such numbers $\gd$ is called the \emph{reach} of $K$. 

%
The intrinsic volumes $V_k(K)$ turn out to be the total masses of the curvature measures: $V_k(K)=C_k(K,\bR^d)$ for $k=0,\ldots,d$.
Here, the volume measure $C_d(K,\cdot):=\lambda_d(K\cap\cdot)$ is added for completeness.

Federer's curvature measures and associated tube formulas have since been extended in various directions; see, for example, \cite{Schneider:CM, Zaehle86, Zaehle87, Fu85, RaZa1, RaZa2} along with the book \cite{Schneider} and the references therein. 
Recently (and most generally), so-called \emph{support measures} have been introduced in \cite{HuLaWe} (based on results in \cite{Stacho}) for arbitrary closed subsets of $\bR^d$, which were also shown to admit a Steiner-type formula.  


The total curvatures and the curvature measures above are defined as the coefficients of some tube formula. It is precisely this approach that we hope to emulate in a forthcoming paper, making use of the tube formulas obtained in the present paper. We believe that (for a suitable choice of the Steiner-like representation for the generators)  the coefficients appearing in our tube formulas may also be understood in terms of curvature, in a suitable sense, and that a localization of the results in this paper may lead to a notion of \emph{complex curvature measures} (or possibly, distributions) for fractal sets. We hope to explore such a possibility in a future work; see also Section~\ref{rem:search-for-fractal-curvatures}.

\subsection{Outline}
In Section~\ref{sec:Steiner-like}, we discuss the geometric hypotheses placed upon the generator(s) of the fractal spray. In Section~\ref{sec:zeta-fns}, we define a zeta function associated to a fractal spray \tiling, which will allow us to derive a pointwise tube formula for \tiling in Section~\ref{sec:pw-tube-formula}. In Section~\ref{sec:Self-similar-tilings}, we obtain the tube formula associated with a self-similar tiling (an important special case of a fractal spray). 
Several examples are discussed in Section~\ref{sec:Examples}. In Section~\ref{sec:the-proof}, we give the detailed proof of the main theorem (Theorem~\ref{thm:pointwise-tube-formula}), the pointwise tube formula for fractal sprays, as well as of Corollary~\ref{thm:ptwise-result-self-similar-case-simplified}, the fractal tube formula for self-similar tilings. Finally, in Section~\ref{sec:conclusion}, we discuss the relation with previously obtained tube formulas and give some possible directions for future research.


\section{Steiner-like formulas for generators.}
\label{sec:Steiner-like}

In this paper, we consider the interior of a set instead of its exterior, as discussed in Section~\ref{sec:motivation}. However, our primary requirement of a generator is that it has a similar (inner) tube formula; 
see Definition~\ref{def:Steiner-like} below and also Section~\ref{sec:motivation} for motivation of the nomenclature.

For a nonempty and bounded open set $\gen \ci\bR^d$, let $\genir=\gr(\gen)$ denote its \emph{inradius}; that is, the radius of the largest open ball
  contained in \gen. It is clear that $\genir$ is always positive and finite. In case \gen is the generator of a fractal spray $\sT$, we have  
  \linenopax
  \begin{align}\label{eqn:gr(gen)=g}
    \gr(\gen^j) 
    = \gr(\simm_j \gen) 
    = \gr(\scale_j \gen) 
    = \scale_j \gr(\gen) 
    = \scale_j \genir
  \end{align}
  for the inradii of the components $\gen^j$ of $\sT$.

It will be useful to write the inner parallel volume $V(\gen,\ge)$ of the set $\gen \ci \bRd$ as a ``polynomial-like'' expansion in \ge of degree at most $d$. More precisely, we have the following definition. 
\begin{defn}
  \label{def:Steiner-like}
  An \emph{(inner) Steiner-like formula} (or a \emph{Steiner-like representation} of the tube formula) for a nonempty and bounded open set $\gen \ci \bRd$ with inradius $\genir = \gr(\gen)$ is an expression for the volume of the inner \ge-parallel sets of \gen of the form
  \linenopax
  \begin{align}\label{eqn:def-prelim-Steiner-like-formula}
    V(\gen,\ge) = \sum_{k=0}^{d} \crv_k(\gen,\ge) \ge^{d-k},
    \qq\text{ for } 0 < \ge \leq \genir, 
  \end{align}
  where for each $k=0,1,\dots,d$, the coefficient function $\crv_k(\gen,\cdot)$ is a real-valued function on $(0,\genir]$ that is bounded on $[\ge_0, \genir]$ for every given $\ge_0 \in (0,\genir]$. 
\end{defn}

\begin{remark}[The choice of the coefficient functions $\crv_k(\gen,\ge)$]
  \label{rem:nonuniqueness-of-Steinerlike}
  Note that a representation of the form \eqref{eqn:def-prelim-Steiner-like-formula} always exists. For example, one can always take a trivial representation with $\crv_d(\gen, \ge) = V(\gen,\ge)$ and $\crv_0(\gen, \ge) = \dots = \crv_{d-1}(\gen, \ge) = 0$ on $(0,\genir]$. Another, slightly less trivial, representation is given by letting $\crv_k(\gen, \ge) = \frac1{d+1} V(\gen, \ge) \ge^{k-d}$ for $k=0,\dots,d$.
   For brevity, we may use the term ``Steiner-like generator/set'' to indicate that a fixed Steiner-like representation for the tube formula is intended, and write ``tube formula'' for  ``inner tube formula''.
  
  We have in mind nontrivial representations of the volume function, in which the coefficients allow interpretations in terms of curvature. Clearly, not every representation can have such an interpretation, and so some uniqueness condition will be needed to characterize the correct one for this purpose. However, this is not our aim here (this issue shall be addressed in a forthcoming paper by the same authors). For the main results of this paper, the tube formulas for fractal sprays (and tilings), we make no assumptions on the generators. In fact, our theorems provide many tube formulas for the same spray --- one for each choice of a Steiner-like representation for the generators. Our formulas should be seen as a general tool to transfer a given representation of the volume function of a generator into a tube formula for the generated sprays. We do not yet know what the canonical choice of the representation for the generator is, but our approach seems flexible enough to contain it. It seems that a reasonable strategy would be to choose the coefficients as ``constant as possible''. It is likely that some integrals of the support measures of \cite{HuLaWe} could provide the coefficients of some canonical representation.
\end{remark}

\begin{remark}[Monophase and pluriphase generators]
  \label{rem:monophase}\label{rem:monophase-and-pluriphase}
  As noted above, the coefficients in the expansion \eqref{eqn:def-prelim-Steiner-like-formula} are clearly not unique. However, if a set \gen has a Steiner-like representation with constant coefficients 
\linenopax
\begin{align*}
  \crv_k(\gen,\ge) = \crv_k(\gen)
  \qq\text{for all }
  \ge \in (0,\genir] \text{ and } k=0,1,\dots,d, 
\end{align*}
then such an expansion is unique, and the set is called \emph{monophase}.  More precisely, a bounded open set $\gen \ci \bRd$ is monophase iff its inner tube formula may be written in the form
  \linenopax
  \begin{align}\label{eqn:monophase}
    V(\gen,\ge) = 
      \sum_{k=0}^{d} \crv_k(\gen) \ge^{d-k}, 
      \qq 0 < \ge \leq \genir.
  \end{align}
  For monophase sets, we always choose this canonical Steiner-like representation. In this case, one has $\gk_d(G)=0$, since otherwise $\lim_{\ge\to 0}\lambda_d(G_{-\ge})\neq 0$. The monophase case has been treated in \cite{TFCD}, at least from the distributional perspective. A variety of natural and classical examples of self-similar tilings in \bRd have monophase generators; see \cite[Section~9]{TFCD}. Furthermore, all geometric (or ordinary) fractal strings (i.e., 1-dimensional fractal sprays) also have monophase generators, since \gen is always an interval; see \cite[Section~8.2]{TFCD}. In general, however, it is rather restrictive to assume that the generator is monophase because many sets (including generators of self-similar tilings) do not have a polynomial expansion with constant coefficients; see Section~\ref{sec:Examples}. 
  
  For monophase sets, the inner tube formula is a polynomial for $\ge \in (0,\genir]$ 
  and this is the reason for the nomenclature. More generally, as in \cite{TFCD,TFSST}, we say a bounded open set $\gen \ci \bRd$ is \emph{pluriphase} iff it has a Steiner-like tube formula with coefficient functions $\crv_k(\gen,\cdot)$  that are piecewise constant with respect to a finite partition of $[0,\genir]$.
  In short, the inner tube formula is piecewise polynomial, with finitely many pieces. (Such a representation is unique if it is assumed that one takes the partition to have as few components as possible.) We use the term \emph{general Steiner-like} (or \emph{Steiner-like with variable coefficients}) to emphasize the distinction from the special cases of monophase and pluriphase sets. It was conjectured in \cite{TFCD,TFSST} that all convex polyhedra are pluriphase. 
\end{remark}

\begin{remark}\label{rem:Steinerlike-in-TFCD}
  The above definition of ``Steiner-like'' is slightly more general than in \cite[Definition~5.1]{TFCD}, where it was introduced: in particular, the local integrability of each coefficient function $\crv_k(\gen,\mydot)$ and the limit condition for $\lim_{\ge\to0^+} \crv_k(\gen,\ge)$ have both been removed.  
  The content of Proposition~\ref{cor:V(Gj,ge)} was taken as a hypothesis in \cite{TFCD}, but is now seen to follow from the assumptions in Definition~\ref{def:Steiner-like}.
\end{remark}


No assumption is made on the uniqueness of the coefficients
$\crv_k(\gen,\ge)$ in Definition~\ref{def:Steiner-like} (as discussed in Remark~\ref{rem:nonuniqueness-of-Steinerlike}), but any
choice of coefficients for \gen satisfying
\eqref{eqn:def-prelim-Steiner-like-formula} gives rise to some coefficients $\crv_k(\gen ^j,\ge)$ for each set $\gen^j = \simm_j(\gen)$ by defining 
   \linenopax
   \begin{equation}\label{eqn:scaling}
        \crv_k(\gen^j,\ge) := \scale_j^k \crv_k(\gen,\scale_j^{-1}\ge),
        \qq \text{for } \; 0 < \ge \leq \gr(\gen ^j) = \scale_j \genir,
      \end{equation}
as is seen in the following proposition. Here and henceforth, $\scale_j^k$ denotes the \kth power of $\scale_j$.

\begin{prop}\label{cor:V(Gj,ge)}
  Let $\sT=\{\gen ^j\}$ be a fractal spray on a generator \gen with a given Steiner-like representation as in \eqref{eqn:def-prelim-Steiner-like-formula}. 
  Then the inner tube formula of each set $\gen^j$ has a Steiner-like representation in terms of the same coefficients:
  \linenopax
  \begin{equation}\label{eqn:V(Gj,ge)}
    V(\gen^j,\ge) = \sum_{k=0}^d \scale_j^k \crv_k(\gen,\scale_j^{-1} \ge) \ge^{d-k}, 
    \qq \text{for } 0<\ge \le \gr(\gen^j) = \scale_j \genir.
  \end{equation}
\end{prop}
\begin{proof}
   The motion invariance and homogeneity of Lebesgue measure implies that for each $j \geq 1$, $\lambda_d(\gen^j_{-\ge})=\lambda_d(\Psi_j (\gen_{-\ge/\scale_j}))=\scale_j^d \lambda_d(\gen_{-\ge/\scale_j})$, where $\Psi_j$ is the similarity transformation of \bRd described in Definition~\ref{def:fractal-spray}. Whence, by \eqref{eqn:def-prelim-Steiner-like-formula} and \eqref{eqn:scaling},
  \linenopax
  \begin{align*}
    V(\gen^j,\ge)
    = \scale_j^d V(\gen,\scale_j^{-1}\ge)
    = \scale_j^d \sum_{k=0}^{d} \crv_k(\gen, \scale_j^{-1}\ge) (\scale_j^{-1}\ge)^{d-k}
    = \sum_{k=0}^{d} \crv_k(\gen^j,\ge) \ge^{d-k}
  \end{align*}
  for $0 < \ge \leq \gr(\gen^j) = \scale_j\genir$, and the Steiner-like representation \eqref{eqn:V(Gj,ge)} follows. Note that the coefficients $\crv_k(\gen^j,\cdot)$ of $\gen^j$ clearly inherit the boundedness properties from the coefficients $\crv_k(\gen,\cdot)$ of $\gen$. 
\end{proof}

In the sequel, we will always work with the coefficient functions of the sets $\gen^j$ chosen according to \eqref{eqn:scaling}.  Proposition~\ref{cor:V(Gj,ge)} ensures this choice is always possible.

Up to this point, the coefficient functions $\crv_k(\gen,\cdot)$ in a Steiner-like formula for $\gen$ have been defined only for $0 < \ge \leq \genir = \gr(\gen)$. For $k=0,1,\dots,d$, 
we define $\crv_k(\gen) := \crv_k(\gen, \genir)$ and then extend $\crv_k(\gen,\ge)$ to $\ge \in (\genir,\iy)$ as constant functions with this value: 
\linenopax
\begin{equation}\label{eqn:coeff-ext}
  \crv_k(\gen,\ge) := \crv_k(\gen)
  \qq\text{for } \ge \geq \genir.
\end{equation}
Note that \eqref{eqn:def-prelim-Steiner-like-formula} need not hold for $\ge > \genir$ and so we have the freedom of the choice \eqref{eqn:coeff-ext}. We emphasize that this choice is vitally important for the tube formulas in Theorem~\ref{thm:pointwise-tube-formula} and its corollaries below to be correct.

%
%
%

Note that, for $\ge=g$, equation (\ref{eqn:def-prelim-Steiner-like-formula}) implies that 
the $d$-dimensional Lebesgue measure of \gen satisfies
\begin{equation} \label{eqn:leb-coeff-rel}
  \gl_d(\gen)
  = V(\gen, \genir)
  = \sum_{k=0}^{d} \crv_k(\gen, \genir) \genir^{d-k}
  = \sum_{k=0}^{d} \crv_k(\gen) \genir^{d-k}.
\end{equation}


\section{Zeta functions and complex dimensions}
\label{sec:zeta-fns}

We will require certain mild hypotheses on the fractal string
\sL which gives the scaling of the spray \tiling. These
conditions are phrased as growth conditions on a zeta function
associated with \sL, within a suitable window, as defined just below.

\begin{defn}\label{def:scaling-zeta-function}
  For a fractal string $\sL = \{\scale_j\}_{j=1}^\iy$, the \emph{scaling zeta function} is given by
  \begin{equation}\label{eqn:scaling-zeta}
    \gzL(s) = \sum_{j=1}^\iy \scale_j^s
  \end{equation}
  for $s \in \bC$ with $\Re(s) > \abscissa$, where \abscissa is the abscissa of convergence of this series. (Compare with \cite[Def.\ 1.8]{FGCD}, where $\gzL$ is called the \emph{geometric zeta function} of $\sL$.) 
  Recall that $\abscissa := \inf\{\ga \in \bR \suth \sum_{j=1}^\iy \scale_j^\ga < \iy\}$ and that \gzL is holomorphic (i.e., analytic) for $\Re s > \abscissa$. Henceforth, if $W \ci \bC$ contains $\{\Re s > \abscissa\}$ and \gzL has a meromorphic continuation (necessarily unique) to a connected open neighborhood of $W$, we abuse notation and continue to denote by \gzL its meromorphic extension. Under these assumptions, each pole $\gw \in W$ of \gzL is called a \emph{visible complex dimension} of \sL and
  the set of visible complex dimensions is written as
  \begin{equation}\label{comp-dim}
    \sD_{\sL}(W)=\{\pole \in W \suth \pole \mbox{ is a pole of } \gzL\}.
  \end{equation}
\end{defn}
Moreover, in the special case when \gzL has a meromorphic continuation to all of \bC, we may choose $W = \bC$ and then simply write $\sD_{\sL}:=\sD_{\sL}(\bC)$ and refer to $\sD_{\sL}$ as \emph{the complex dimensions of \sL}.

In practice, $W$ will be a window (the part of $\bC$ to the right of a screen $S$) as in Definition \ref{def:screen}, just below. 
The following three definitions are excerpted from \cite[Section~5.3]{FGCD}.

\begin{defn}
  \label{def:screen}
  Let $S:\bR \to (-\iy, \abscissa]$ be a bounded Lipschitz continuous function. Then the \emph{screen}
  is $S = \{S(t) + \ii t \suth t \in \bR\}$, the graph of a function with
  the axes interchanged. Here and henceforth, we denote the imaginary unit by $\ii := \sqrt{-1}$. We let
  \linenopax
  \begin{align}
    \label{eqn:infS}
    \inf S &:= \inf\nolimits_{t \in \bR} S(t) = \inf\{\Re s \suth s \in S\}, \text{ and} \\
    \label{eqn:supS}
    \sup S &:= \sup\nolimits_{t \in \bR} S(t) = \sup\{\Re s \suth s \in S\}.
  \end{align}
  The screen is thus a vertical contour in \bC. The region to the
  right of the screen is the set $W$, called the \emph{window}:
  \linenopax
  \begin{align}
    \label{eqn:window}
    W &:= \{z \in \bC \suth \Re z \geq S(\Im z)\}.
  \end{align}
  For a given string \sL, we always choose $S$ to avoid $\sD_{\sL}$ and such that \gzL can be meromorphic{ally} continued to an open neighborhood of $W$. We also assume that $\sup S \leq \abscissa$, that is, $S(t) \leq \abscissa$ for every $t \in \bR$.
\end{defn}

\begin{defn}
  \label{def:languid}
  The fractal string \sL is said to be \emph{languid} if its
  associated zeta function \gzL satisfies certain horizontal and
  vertical growth conditions. Specifically, let
  $\{T_n\}_{n \in \bZ}$ be a sequence in \bR such that $T_{-n} < 0 < T_n$
  for $n \geq 1$, and
  \linenopax
  \begin{align}
    \label{eqn:Tn seq conds}
    \lim_{n \to \iy} T_n = \iy,
    \lim_{n \to \iy} T_{-n} = -\iy, \text{ and }
    \lim_{n \to \iy} \frac{T_n}{|T_{-n}|} = 1.
  \end{align}
  For \sL to be languid, there must exist constants $\languidOrder \in \bR$ and $c>0$, and a sequence $\{T_n\}$ as described in \eqref{eqn:Tn seq conds}, such that: \\ 
  \vstr 
  For all $n \in \bZ$ and all $\ga \geq S(T_n)$,
  \linenopax
  \begin{align}
    \label{eqn:L1} \tag*{\textbf{L1}}
    |\gzL(\ga + \ii T_n)| \leq c \cdot \left(|T_n| + 1\right)^\languidOrder,
  \end{align}
  and for all $t \in \bR$, $|t| \geq 1$,
  \linenopax
  \begin{align}
    \label{eqn:L2} \tag*{\textbf{L2}}
    |\gzL(S(t) + \ii t)| \leq c \cdot |t|^\languidOrder.
  \end{align}
  In this case, \sL is said to be \emph{languid of order \languidOrder}.
\end{defn}

\begin{defn}
  \label{def:strongly-languid}
  The fractal string \sL is said to be \emph{strongly languid} of order \languidOrder and with constant $A$ iff it satisfies \textbf{\ref{eqn:L1}} and the following condition \textbf{\ref{eqn:L2'}}, which is clearly stronger than \textbf{\ref{eqn:L2}}:\\
  \vstr 
  There exists a sequence of screens $S_m$ for $m \geq 1$,
  $t \in \bR$, with $\sup S_m \to -\iy$ as
  $m \to \iy$, and with a uniform Lipschitz bound, for which there exist constants $\languidOrder \in \bR$ and $A,c>0$ such that 
  \linenopax
  \begin{align}
    \label{eqn:L2'}  \tag*{\textbf{L2$'$}}
    |\gzL(S_m(t) + \ii t)| \leq c \cdot A^{|S_m(t)|}(|t|+1)^\languidOrder,
  \end{align}
  for all $t \in \bR$ and $m \geq 1$.
\end{defn}

By saying ``\gzL is languid'', we mean just that \sL is languid.
In the rest of this paper, \tiling is assumed to be a fractal spray with a generator $\gen \ci \bRd$, scaled by the fractal string $\sL = \{\scale_j\}_{j=1}^\iy$ with $\ell_1=1$ as in Remark~\ref{rem:renormalize}. The tubular zeta function first appeared in \cite{TFCD}, but we need to modify the definition in order to extend it to the case when the generators are not monophase; thus, the following definition is new.

\begin{defn}\label{def:volume-zeta-fn}
  The \emph{tubular zeta function} \gzT of the fractal spray \tiling is defined by
  \begin{equation}\label{eqn:def:volume-zeta}
    \gzT(\ge,s)
    = \ge^{d-s} \sum_{j=1}^\iy \scale_j^s \left(
      \sum_{k=0}^{d}\frac{\genir^{s-k}}{s-k} \crv_k(\gen,\scale_j^{-1}\ge) - \frac{\genir^{s-d}}{s-d}\gl_d(G)\right)
  \end{equation}
  for every $\ge\in(0,\iy)$ and for each $s \in \bC$ such that the sum converges absolutely. As in Definition~\ref{def:scaling-zeta-function}, we will henceforth abuse notation and use $\gzT(\ge,s)$ to mean a meromorphic extension of the function defined by the formula \eqref{eqn:def:volume-zeta}, as convenient.
\end{defn}

Note that by \eqref{eqn:leb-coeff-rel}, \eqref{eqn:scaling-zeta}, and Proposition~\ref{cor:V(Gj,ge)}, for $\ge \geq \genir$, one has
\linenopax
\begin{align}\label{eqn:gzT-for-big-epsilon}
  \gzT(\ge,s)
    &= \ge^{d-s} \gzL(s) \left(
      \sum_{k=0}^{d}\frac{\genir^{s-k} \crv_k(\gen)}{s-k}- \frac{\genir^{s-d} \gl_d(G)}{s-d}\right) 
    = \ge^{d-s} \gzL(s) \sum_{k=0}^{d-1}\frac{\genir^{s-k} \crv_k(\gen) (d-k)}{(s-k)(d-s)}.
\end{align}
Since $\gzL(s)$ is a Dirichlet series, it has an \emph{abscissa of convergence}: there is a unique number $\abscissa \in [-\iy,\iy]$ such that $\gzL(s)$ converges absolutely for $s$ with $\Re s > \abscissa$ and diverges for $s$ with $\Re s < \abscissa$. The abscissa of convergence is thus analogous to the radius of convergence of a power series.
  Note that in all reasonable situations we have $0 \leq \abscissa \leq d$. Indeed, $\abscissa \geq 0$ follows immediately from the non-finiteness of the fractal string (assumed in Definition~\ref{def:fractal_string}), and $\abscissa \leq d$ follows if one requires that the generated fractal spray $\sT$ have finite total volume. (Note that for fractal sprays with infinite total volume, the question for a (global) tube formula does not make sense.)
For the scaling zeta functions of the self-similar tilings discussed in Section~\ref{sec:Self-similar-tilings}, one has $0 < \abscissa < d$, and 
\abscissa coincides with the Minkowski dimension, Hausdorff dimension, and similarity dimension of the associated self-similar set; for a precise statement, please see \cite[Section~4.3]{TFCD} and Section~\ref{sec:Self-similar-tilings} below (especially Remark~\ref{rem:incarnations-of-D}).

Note that the tubular zeta function \gzT may be viewed as a generating function for the geometry of the fractal spray \tiling.

Proposition~\ref{prop:poles} clarifies the relation between the scaling zeta function \gzL and the tubular zeta function \gzT of a fractal spray \tiling. It also motivates and justifies the definition of complex dimensions of fractal sprays. The intended application of Proposition~\ref{prop:poles} is with \gW as a suitable open  neighborhood of a window $W$ for the scaling zeta function \gzL of \tiling, as in Definition~\ref{def:screen}. (Proposition~\ref{prop:poles} is extended significantly in Theorem~\ref{thm:Meromorphic-continuation-of-gzT} and Lemma~\ref{thm:residues-of-gzT}.)

\begin{prop} \label{prop:poles}
  If \abscissa is the abscissa of convergence of $\gzL$, then for all $\ge > 0$, the series in \eqref{eqn:def:volume-zeta} defining $\gzT(\ge,s)$ converges absolutely for any fixed $s \in \bC \less \{0,1,\dots,d\}$ with $\Re s > \abscissa$. More generally, suppose $\gzL$ is meromorphic in a connected open set \gW containing $\{\Re s > \abscissa\}$. Then for all $\ge > 0$, the function $\gzT(\ge, \mydot)$ is meromorphic in \gW and each pole $\gw \in \gW$ of $\gzT(\ge, \mydot)$ is a pole of \gzL or belongs to the set $\{0,1,\ldots,d\}$.
\end{prop}
\begin{proof}
  Fix $\ge>0$. 
  Upon expanding \eqref{eqn:def:volume-zeta} of Definition~\ref{def:volume-zeta-fn} and interchanging the sums, the tubular zeta function becomes
  \linenopax
  \begin{align}\label{eqn:zeta-expanded}
    \gzT(\ge,s)
    &= \sum_{k=0}^{d}\frac{ \ge^{d-s} \genir^{s-k} }{s-k}\sum_{j=1}^\iy \scale_j^s\crv_k(\gen,\scale_j^{-1}\ge) - \frac{\ge^{d-s} \genir^{s-d}}{s-d} \gzL(s) \gl_d(G).
  \end{align}
  It is clear that the second term on the right-hand side of \eqref{eqn:zeta-expanded}  is convergent for $s$ as in the hypotheses, so it remains to check that the first term is similarly convergent for each $k$.
  
  Since \ge is fixed and $\scale_j \searrow 0$, define $J=J(\ge)$ to be the index of the last scale greater than \ge:
  \linenopax
  \begin{align}\label{eqn:J(eps)}
    J(\ge) := \max\{j \geq 1 \suth \scale_j^{-1} \ge < \genir \} \vee 0. 
  \end{align}
At the end of \eqref{eqn:J(eps)}, ``$\vee 0$'' indicates that $J(\ge) = 0$ for $\ge \geq \scale_1 \genir$.%
  \footnote{By convention, $\sup \es = -\iy$. Thus the maximum $\vee 0$ is included so that $J(\ge) = 0$ when $\ge \geq \scale_1 \genir$.}
Now $\crv_k(\gen,\scale_j^{-1}\ge) = \crv_k(\gen)$ for all $j > J$, by \eqref{eqn:coeff-ext}, and so
  \linenopax
  \begin{align}\label{eqn:protosplit}
    \sum_{j=1}^\iy \scale_j^s \crv_k(\gen,\scale_j^{-1}\ge)
    = \sum_{j=1}^{J} \scale_j^s \crv_k(\gen,\scale_j^{-1}\ge)
     + \crv_k(\gen) \sum_{j=J+1}^{\iy} \scale_j^s.
  \end{align}
  Observe that the first sum on the right-hand side of \eqref{eqn:protosplit} is entire, as a finite sum of the entire functions $c_x x^s$, and that the second sum on the right-hand side of \eqref{eqn:protosplit} converges absolutely exactly where $\gzL$ does; that is, for $\Re s > \abscissa$.

  Justification of the claims of meromorphicity are obtained by parallel reasoning; the decomposition \eqref{eqn:protosplit} shows that, except possibly for $\{0,1,\dots,d\}$, the tubular zeta function is meromorphic precisely where the scaling zeta function is.
\end{proof}

\begin{defn}[Complex dimensions]\label{def:complex-dimn}
  The set 
\[ 
\sD_{\sT} = \sD_{\sL} \cup \{0,1,\dots,d\}
\]
of (potential) poles of \gzT is called the set of \emph{complex dimensions} of $\sT$. Let $W \ci \bC$ be a window for \gzL as in Definition~\ref{def:screen}, so that \gzL is meromorphic in an open neighborhood of $W$. 
  (Proposition~\ref{prop:poles} thus implies that for each fixed $\ge > 0$, the function $\zeta_\sT(\ge,\mydot)$ is also meromorphic in an open neighborhood of $W$.) 
  Then $\sD_\tiling(W) := \sD_\tiling \cap W$ is called the set of \emph{visible complex dimensions} of $\sL$ in $W$, in parallel with \eqref{comp-dim}. We refer to $\sD_\sL$ as the \emph{scaling} complex dimensions and $\{0,1,\dots,d\}$ as the \emph{integer} complex dimensions of \tiling. 
\end{defn}




\section{Pointwise tube formulas for fractal sprays}
\label{sec:pw-tube-formula}

Now we are ready to state one of our main results, a pointwise tube formula for a fractal spray \tiling, which, for $\ge > 0$, describes the inner parallel volume $V(\tiling,\ge)$ as a sum of the residues of its tubular zeta function $\gzT(\ge,s)$. For fractal sprays with more than one generator, one can consider each generator independently, and the tube formula of the whole spray is then given by the sum of the expressions derived for the sprays on each single generator. Thus, there is no loss of generality in considering only the case of a single generator in Theorem~\ref{thm:pointwise-tube-formula}.

\begin{theorem}[Pointwise tube formula for fractal sprays]
  \label{thm:pointwise-tube-formula}
  Let \tiling be a fractal spray given by the fractal string $\sL = \{\scale_j\}_{j=1}^\iy$ and the generator $\gen \ci \bRd$. Fix a Steiner-like representation for $G$, as in \eqref{eqn:def-prelim-Steiner-like-formula}, and assume that the abscissa of convergence \abscissa of the scaling zeta function \gzL of \tiling is strictly smaller than $d$.
  
  \emph{Tube formula with error term}.
  If \gzL is languid of order $\languidOrder < 1$ for some screen $S$ for which $S(0) < 0$ (so that $W$ contains the integers $\{0,1,\ldots,d\}$), then for all $\ge > 0$,
  \linenopax
  \begin{equation} \label{eqn:pointwise-tube-formula}
    V(\tiling,\ge)
    = \sum_{\pole\in \sD_{\tiling}(W)} \res[s=\pole]{\gzT(\ge,s)} + \gl_d(\gen)\gzL(d) + \R(\ge),
  \end{equation}
  where the error term $\R$ (given explicitly in Remark~\ref{rem:error-term} below) is estimated by $\R(\ge) = O(\ge^{d-\sup S})$ as $\ge \to
  0^+$.

  \emph{Tube formula without error term}.  
  If \gzL is strongly languid of order $\languidOrder < 2$ and with constant $A
  > 0$, then the choice $W = \bC$ for the window is possible in
  \eqref{eqn:pointwise-tube-formula}, implying that the error
  term $\R(\ge)$ vanishes identically for all $0 < \ge < \min\{\genir, A^{-1} \genir\}$. 
\end{theorem}

Theorem~\ref{thm:pointwise-tube-formula} and its corollaries are consistent with earlier results in \cite{FGCD} and \cite{TFCD} and generalize them in several respects; see Section~\ref{sec:consistency}. 
We will give the rather lengthy proof of Theorem~\ref{thm:pointwise-tube-formula}  
in Section~\ref{sec:the-proof-sprays}. For a description of one of the main new ideas and techniques, 
we refer the reader to Remark~\ref{rem:obs1}. 

Note that the tube formula without error term is an \emph{exact} pointwise formula. For this result, one must assume that the sequence of screens $\{S_m\}_{m=1}^\iy$ of Definition~\ref{def:strongly-languid} satisfies $S_m(0) < 0$ for each $m$. However, there is no loss of generality because $\sup S_m \to -\iy$.

The following result is really a corollary of the proof of the first part of Theorem~\ref{thm:pointwise-tube-formula}. For this reason, its short proof is provided in Section~\ref{sec:proof-of-cor-monophase} at the very end of Section~\ref{sec:the-proof-sprays}.

\begin{cor}[The monophase case]
  \label{thm:fractal-spray-tube-formula-monophase}
  If, in addition to the hypotheses of the first part of Theorem~\ref{thm:pointwise-tube-formula}, we assume that \gen is monophase, then the tube formula with error term remains valid (with the same error estimate), without the restriction that $S(0)<0$, provided this hypothesis is replaced by the much weaker condition that the screen $S$ avoids the integers $0,1,\ldots,d$. Hence, in particular, it still holds for a screen $S$ that is arbitrarily close to the vertical line $\Re s = \abscissa$.
  Moreover, the error term $\R$ is given by \eqref{eqn:pointwise-error2} (or \eqref{eqn:pointwise-error}).
\end{cor}

\begin{remark}\label{rem:error-term}
The error term $\R$ in formula \eqref{eqn:pointwise-tube-formula} in Theorem~\ref{thm:pointwise-tube-formula} is explicitly given by
  \linenopax
  \begin{equation} \label{eqn:pointwise-error}
    \R(\ge)
    = \frac{1}{2\gp\ii} \int_S \frac{\ge^{d-s}\gzL(s)}{d-s} \left(\sum_{k=0}^{d-1} \frac{\genir^{s-k}}{s-k} (d-k)\crv_k(\gen)\right)
      \,ds.
  \end{equation}
  The integrand in formula \eqref{eqn:pointwise-error} will be called the \emph{tail zeta function} of $\sT$ and denoted by $\gzT[\tail](\ge,s)$ in Section~\ref{sec:the-proof}; see, in particular, Section~\ref{sec:splitting} and equation \eqref{eqn:volume-zeta-split2}.  The function $\gzT[\tail](\ge,s)$ is one part of the head-tail splitting of the tubular zeta function $\gzT(\ge,s)$ employed in the proof of Theorem~\ref{thm:pointwise-tube-formula}.
 In the situation of Corollary~\ref{thm:fractal-spray-tube-formula-monophase}, one has  $\gzT[\tail]=\gzT$ and thus the error term $\R$ is equivalently given by 
 \begin{equation} \label{eqn:pointwise-error2}
    \R(\ge)
    = \frac{1}{2\gp\ii} \int_S \gzT(\ge,s)
      \,ds.
  \end{equation}
  See Section~\ref{sec:consistency} below for a discussion of the consistency of the error term with earlier results. 
\end{remark}

\begin{remark}\label{rem:0-in-W}
  For investigating delicate questions concerning the Minkowski measurability of fractal sprays and self-similar tilings (see, for example, \cite[Corollary~8.5]{TFCD}), it is important to be able to drop the assumption that $S(0)<0$, as in Corollary~\ref{thm:fractal-spray-tube-formula-monophase}. However, this generalization poses technical challenges for the case of more general (i.e. non-monophase) generators. In  the monophase case, in contrast, our tube formulas enable us to derive results on the Minkowski measurability of fractal sprays. For example, for a self-similar tiling $\sT$ (as discussed in Section~\ref{sec:Self-similar-tilings} below), let us denote
  \linenopax
  \begin{align}\label{eqn:Md(G)}
    \gG_s(\gen)
    := \sum_{k=0}^d \frac{\genir^{d-s}}{s-k} (d-k) \crv_k(\gen).
  \end{align}  
Assume that $\gG_\abscissa(\gen)\neq 0$ and (in the lattice case) that $\gG_{\abscissa+\ii m \per}(\gen) \neq 0$ for some $m \in \bZ\less\{0\}$.%
\footnote{See Remark~\ref{rem:Minkowski-warmup} for a discussion of terminology and notation in the lattice and nonlattice cases.} 
If $d-1 < \abscissa < d$, and the self-similar tiling \tiling has a single monophase generator, then one can apply the methods of proof (and the conclusions) of  Theorems~8.23, 8.30, and 8.36 (along with Theorems~2.17, 3.6, 3.25 and 5.17) of \cite{FGCD} to see that \tiling is Minkowski measurable if and only if \gzL is nonlattice. In addition, \tiling has Minkowski dimension \abscissa. In the nonlattice case, \tiling has (positive and finite) Minkowski content
  \linenopax
  \begin{align}\label{eqn:Minko-contento}
    \res[\abscissa]{\gzL(s)} \frac{\gG_\abscissa(\gen)}{d-\abscissa},
  \end{align}
  with residues $\res[\abscissa]{\gzL(s)}$ given by \eqref{eqn:nonlattice-residues}. In the lattice case, \tiling has \emph{average} Minkowski content given by \eqref{eqn:Minko-contento}, with residues $\res[\abscissa]{\gzL(s)}$ as in \eqref{eqn:lattice-residues}.

With definitions suitably adapted from \cite[Chapter~8]{FGCD}, an entirely analogous statement can be made about a self-similar set \attr. More precisely, if the tiling is also assumed to satisfy the compatibility condition \eqref{eqn:compatibility-condition}, then analogous results extend to the associated self-similar fractal \attr. This strengthens and specifies the results of \cite[Corollary~8.5]{TFCD}; see also \cite[Remark~10.6]{TFCD}. Further discussion of this issue is lengthy and beyond the scope of the present paper. For the interested reader, details on the monophase case can be found in \cite{MMappendix}; see also Section~\ref{rem:0-in-W-redux}. For more general generators, these results are under further development in \cite{MMFS}.
\end{remark}

\begin{remark}\label{rem:languidOrder}
  In light of Definition~\ref{def:languid} and Definition~\ref{def:strongly-languid}, note that if \gzL is strongly languid of order \languidOrder, then it is also strongly languid (and hence languid) of any higher order, but not necessarily of any lower order. Consequently, the assumptions of the second part of Theorem~\ref{thm:pointwise-tube-formula} (the strongly languid case) do not imply those of the first part (the languid case). Compare to \cite[Remark~8.8]{FGCD}.
\end{remark}

\begin{remark}\label{rem:residues-are-subtractive}
  Define $T := \bigcup_{j=1}^\iy \gen^j$ and let $T_{-\ge}$ be as defined in \eqref{eqn:def:inner-parallel-set}.
  Note that homogeneity of Lebesgue measure gives
  \linenopax
  \begin{align*}
    \gl_d(\gen)\gzL(d)
    = \sum_{j=1}^\iy \scale_j^d \gl_d(\gen)
    = \sum_{j=1}^\iy \gl_d(\gen^j),
  \end{align*}
  and hence the tube formula \eqref{eqn:pointwise-tube-formula} expresses the fact that the measure of the complement of $T_{-\ge}$ in $T$ is given by
  \linenopax
  \begin{equation} \label{eqn:pointwise-tube-formula-recalled}
    \gl_d(T\setminus T_{-\ge})
    = - \sum_{\pole\in \sD_{\tiling}(W)} \res[s=\pole]{\gzT(\ge,s)} \q\big(-\R(\ge)\big).   
  \end{equation}
\end{remark}


\section{Pointwise tube formulas for self-similar tilings}
\label{sec:Self-similar-tilings}

In \cite{SST,GeometryOfSST,TFCD}, the focus is on self-similar tilings. Such an object is a fractal spray associated to an iterated function system $\simtset$, $N \geq 2$, where each $\simt_n$ is a contractive similarity mapping of \bRd with scaling ratio $r_n \in (0,1)$. For $A \ci \bRd$, we write $\simt(A) := \bigcup_{n=1}^N \simt_n(A)$. The \emph{self-similar set} $\attr$ (generated by the self-similar system $\{\simt_1,\ldots,\simt_N\}$) is the unique (compact and nonempty) solution of the fixed-point equation $\attr = \simt(\attr)$\,; cf.~\cite{Hut}. The fractal \attr is also called the \emph{attractor} of the self-similar system $\{\simt_1,\ldots,\simt_N\}$.
%
To proceed with the construction of a self-similar tiling, the system must satisfy the \emph{open set condition} and a \emph{nontriviality condition}: 

  A self-similar system $\simtset$ (or its attractor \attr) satisfies the \emph{open set condition} (OSC) if and only if there is a nonempty open set $O \ci \bRd$ such that
  \linenopax
  \begin{align}
    \simt_n(O) &\ci O, \q n=1, 2,\dots, N \label{eqn:def:OSC-containment} \\
    \simt_n(O) &\cap \simt_m(O) = \emptyset \text{ for } n \neq m.
      \label{eqn:def:OSC-disjoint}
  \end{align}
  In this case, $O$ is called a \emph{feasible open set} for $\simtset$ (or \attr), cf.~\cite{Hut, Falconer, BandtNguyenRao}.

  A self-similar set \attr satisfying OSC is said to be \emph{nontrivial}, if there exists a feasible open set $O$ such that
  \begin{equation}\label{eqn:nontriv}
    O\not \ci \simt(\cj{O})\,,
  \end{equation}
  where $\cj{O}$ denotes the closure of $O$; otherwise, \attr is called \emph{trivial}. This condition is needed to ensure that the set
  $O\setminus\simt(\cj{O})$ in Definition~\ref{def:self-similar-tiling} is nonempty.  It turns out that nontriviality is independent of the particular choice of the set $O$. It is shown in \cite{GeometryOfSST} that \attr is trivial if and only if it has interior points, which amounts to the following characterization of nontriviality:
\begin{prop}[{\cite[Corollary~5.4]{GeometryOfSST}}]
  \label{cor:OSC-dimension-d-implies-trivial}
  Let $\attr \ci \bRd$ be a self-similar set satisfying OSC. Then \attr is nontrivial if and only if \attr has Minkowski dimension (or equivalently, Hausdorff dimension) strictly less than $d$. 
\end{prop}
All the self-similar sets \attr considered in this paper will be assumed to be nontrivial, and the discussion of a self-similar tiling \tiling implicitly assumes that the corresponding attractor \attr is nontrivial (and satisfies OSC).

Denote the set of all finite \emph{words} formed by the alphabet $\{1,\dots,N\}$ by
\linenopax
\begin{equation}\label{eqn:def:words}
  \Words := \bigcup_{k=0}^\iy \{1,\dots,N\}^k\,.
\end{equation}
For any word $w=w_1 w_2\dots w_n \in \Words$, let $r_w := r_{w_1}\cdot\ldots\cdot r_{w_n}$ and $\simt_w := \simt_{w_1} \circ \dots \circ \simt_{w_n}$. In particular, if $w \in \Words $ is the \emph{empty word}, then $r_w=1$ and $\simt_w=\mathrm{Id}$.

\begin{defn}\label{def:self-similar-tiling}(Self-similar tiling) 
  Let $O$ be a feasible open set for $\simtset$. Denote the connected components of the open set $O \setminus \simt(\bar O)$ by $G_q, q \in Q$. Then the \emph{self-similar tiling} \tiling associated with the self-similar system $\{\simt_1,\ldots,\simt_N\}$ and $O$ is the set
  \linenopax
  \begin{equation}\label{eqn:def:self-similar-tiling}
    \tiling(O) := \{ \simt_w(G_q) \suth w \in \Words, q \in Q\}.
  \end{equation}
\end{defn}  
We order the words $w^{(1)}, w^{(2)}, \ldots$ of \Words in such a way that the sequence $\{\scale_j\}_{j = 1}^\iy$ given by $\scale_j := r_{w^{(j)}}$, $j=1,2,\ldots$, is nonincreasing.
  It is clear that a self-similar tiling is thus a collection of fractal sprays, each with fractal string $\sL = \{\scale_j\}_{j=1}^\iy$ and a generator $\gen_q$, $q \in Q$. In this context, the mapping $\simm_j$ appearing in Definition~\ref{def:fractal-spray} corresponds to $\simt_{w^{(j)}}$. 
  
The terminology ``self-similar tiling'' comes from the fact (proved in \cite[Theorem~5.7]{GeometryOfSST}) that $\tiling(O)$ is an \emph{open tiling} of $O$ in the following sense: The \emph{tiles} $\simt_w(G_g)$ in $\tiling(O)$ are pairwise disjoint open sets and the closure of their union is the closure of $O$, that is,
\linenopax 
\[\cj{O}= \cj{\bigcup\nolimits_{q \in Q} \bigcup\nolimits_{w \in \Words} \simt_w(G_q)}\,.\]

\begin{remark}[Tube formulas for self-similar \emph{sets}]
  \label{rem:compatibility-theorem}
  In \cite[Theorem~6.2]{GeometryOfSST}, precise conditions are given for when the tube formula of a self-similar tiling can be used to obtain the tube formula for the corresponding self-similar set, the attractor \attr; recall from \eqref{eqn:ext-nbd-of-A} that for a bounded set $A \ci \bRd$, we define $A_\ge := \{x \in \bRd \less A \suth \dist(x,A) \leq \ge\}$. 
  
  Let $\attr\ci\bR^d$ be a self-similar set satisfying OSC with some feasible open set $O$ and $\dim_M \attr < d$ (i.e., \attr is nontrivial). Let $\tiling(O)$ be the associated tiling of $O$, and let $K:=\cj{O}$ and $T:=\bigcup_{w\in \Words, q\in Q} \simt_w(G_q)$. Then \cite[Theorem~6.2]{GeometryOfSST} states that one has a disjoint decomposition
  \linenopax
  \begin{align}\label{eqn:disjoint-decomposition}     
    \attr_\ge = T_{-\ge} \cup K_\ge,
    \qq\text{for all }\; \ge \geq 0,
  \end{align}
  if and only if the following compatibility condition is satisfied:
  \linenopax
  \begin{align}\label{eqn:compatibility-condition}     
    \bd K \ci F.
  \end{align}
  In this case, the tube formula for the self-similar set \attr can be obtained simply by adding to $V(\tiling,\ge)$ the (outer) tube formula $\gl_d(K_\ge)$ as in \eqref{eq:Steiner-formula-ext} (although note that in the present context, $K$ need not be convex). For example, the Sierpinski gasket and the Sierpinski carpet tilings (see Figures~\ref{fig:sierpinski-gasket} and \ref{fig:sierpinski-carpet}) satisfy the compatibility condition \eqref{eqn:compatibility-condition}, whereas the Koch curve and the pentagasket tilings do not (see Figures~\ref{fig:koch-tiling} and \ref{fig:pentagasket}). Condition \eqref{eqn:compatibility-condition} will not be assumed in the remainder of the paper.
\end{remark}  

From now on, let $\tiling = \tiling(O)$ be a self-similar tiling associated with the self-similar system $\{\simt_n\}_{n=1}^N$ and the generator \gen. We refer to the fractal \attr as the self-similar set associated to \tiling. 
  For the same reasons as described in the first paragraph of Section~\ref{sec:pw-tube-formula}, we lose no generality by stating all results for self-similar tilings with one generator, which we will denote by $G$ in the sequel. (For natural examples of a self-similar tiling with multiple generators, see the pentagasket depicted in Figure~\ref{fig:pentagasket} of the examples section, and also Example~\ref{exm:U-shaped}, which is depicted in Figure~\ref{fig:u-shaped}.)
  Without loss of generality, we may also assume that the scaling ratios $\{r_n\}_{n=1}^N$ of $\{\simt_n\}_{n=1}^N$ are indexed in descending order, so that
  \linenopax
  \begin{align}\label{eqn:scaling-ratios-ordered}
    0 < r_N \leq \dots \leq r_2 \leq r_1 < 1.
  \end{align}

It follows from \cite[Theorem~2.9]{FGCD} (see also \cite[Theorem~4.7]{TFCD}) that \gzL has a meromorphic extension to all of \bC given by 
  \linenopax
  \begin{align}\label{eqn:gzL-extended-to-C}
    \gzL(s) = \frac1{1 - \sum_{n=1}^N r_n^s},
    \qq s \in \bC,
  \end{align}
and hence that the set $\sD_\sL$ of scaling complex dimensions of \tiling consists precisely of the roots $s\in\bC$ of the equation
  \linenopax
  \begin{align}\label{eqn:Moran}
    \sum_{n=1}^N r_n^s = 1.
  \end{align}
It is known from \cite[Theorem~3.6]{FGCD} that the set $\sD_\sL$ lies in a bounded vertical strip: there exists a real number $\abscissa_l \in (-\iy, \abscissa)$ such that
  \linenopax
  \begin{align}\label{eqn:pole-strip}
     \abscissa_l \leq \Re s \leq \abscissa, \qq \text{ for all }\; s \in \sD_\sL.  
  \end{align}
  For the remainder of this paper, we let
  \linenopax
  \begin{align}\label{eqn:dimns-of-C}
    \DT = \DT(\bC) = \sD_\sL \cup \{0,1,\dots,d\}.
  \end{align}

\begin{remark}[Various incarnations of \abscissa]
  \label{rem:incarnations-of-D}
  Recall that \abscissa denotes the abscissa of convergence of \gzL. It follows from \cite[Theorem~3.6]{FGCD} that $\abscissa = \abscissa_\sL$ is a simple pole of \gzL and that $D$ is the only pole of \gzL (i.e., the only scaling complex dimension of \tiling) which lies on the positive real axis. Furthermore, it coincides with the unique real solution of \eqref{eqn:Moran}, often called the \emph{similarity dimension} of \attr and denoted by \gd. Since \attr satisfies OSC, $D$ also coincides with the Minkowski and Hausdorff dimension of \attr, denoted by $D_\attr$ and $H_\attr$, respectively. (For this last statement, see \cite{Hut}, as described in \cite[Theorem~9.3]{Falconer}.) 
  Moreover, it is clear that $\abscissa>0$ since $N \geq 2$, and that $\abscissa \leq d$; in fact, Proposition~\ref{cor:OSC-dimension-d-implies-trivial} implies $\abscissa < d$. In summary, we have
  \linenopax
  \begin{align}\label{eqn:D-dimensions}
    0 < \abscissa < d
    \qq\text{and}\qq
    \abscissa = \gd = \abscissa_\attr = H_\attr.
  \end{align}
\end{remark}

The following result is an immediate consequence of \cite[Theorem~3.6]{FGCD}, which provides the structure of the complex dimensions of self-similar fractal strings (even for the case when \abscissa may be larger than 1). 

\begin{prop}[Lattice/nonlattice dichotomy, see {\cite[Section~4.3]{TFCD}}]
  \label{thm:lattice-nonlattice-dichotomy}
  \hfill \\
  \hstr[3] \emph{Lattice case}. When the logarithms of the scaling ratios $r_n$ are each an integer multiple of some common positive real number, the scaling complex dimensions lie periodically on finitely many vertical lines, including the line $\Re s = \abscissa$. In this case, there are infinitely many complex dimensions with real part \abscissa.
  \hfill \\
  \hstr[3] \emph{Nonlattice case}. Otherwise, the scaling complex dimensions are quasiperiodically distributed (as described in \cite[Chapter~3]{FGCD}) and $s = \abscissa$ is the only complex dimension with real part \abscissa. However, there exists an infinite sequence of simple scaling complex dimensions approaching the line $\Re s = \abscissa$ from the left.
  In this case (cf.~\cite[Section~3.7.1]{FGCD}), the set $\{\Re s \suth s \in \D\}$ appears to be dense in finitely many compact subintervals of $[\abscissa_l,\abscissa]$, where $D_l$ is as in \eqref{eqn:pole-strip}.
\end{prop}

\begin{remark}\label{rem:Minkowski-warmup}
  It follows from \cite[Theorem~3.6]{FGCD} that in the lattice case (i.e., when $r_n = r^{k_n}$, $n=1,\ldots,N$, for some $0<r<1$ and positive integers $\{k_n\}_{n=1}^N$), the scaling complex dimensions have the same multiplicity and a Laurent expansion with the same principal part on each vertical line along which they appear. In particular, since $\abscissa$ is simple (see Remark~\ref{rem:incarnations-of-D}), all the scaling complex dimensions $\{\abscissa+\ii m\per\}_{m \in \bZ}$ (where $\per = 2\gp/\log r^{-1}$) along the vertical line $\Re s = \abscissa$ are simple and have residue equal to 
  \linenopax
  \begin{align}\label{eqn:lattice-residues}
    \res[\abscissa]{\gzL(s)}
    = \frac{1}{\log r^{-1} \sum_{n=1}^N k_n r^{k_n \abscissa}}.
  \end{align}
  In the nonlattice case, $\abscissa$ is simple with residue
  \linenopax
  \begin{align}\label{eqn:nonlattice-residues}
    \res[\abscissa]{\gzL(s)}
    = \frac{1}{\sum_{n=1}^N r_n^{\abscissa} \log r_n^{-1}}.
  \end{align}
  Note that \eqref{eqn:nonlattice-residues} is also valid in the lattice case. Proposition~\ref{thm:lattice-nonlattice-dichotomy} and the contents of this remark are used when applying Theorem~\ref{thm:ptwise-result-self-similar-case} and Corollary~\ref{thm:ptwise-result-self-similar-case-simplified} to the examples in Section~\ref{sec:Examples}.
\end{remark}

\subsection{Exact pointwise tube formulas}
\label{sec:Exact-pointwise-tube-formulas}

The following result is a consequence of the strongly languid case of Theorem~\ref{thm:pointwise-tube-formula} when applied to self-similar tilings.

\begin{theorem}[Exact pointwise tube formula for self-similar tilings]
  \label{thm:ptwise-result-self-similar-case}
  Assume that \tiling is a self-similar tiling with generator $\gen \ci \bRd$, and that a Steiner-like representation has been chosen for \gen as in \eqref{eqn:def-prelim-Steiner-like-formula}. 
  Then for all $\ge \in (0, \genir)$,
  \linenopax
  \begin{equation} \label{eqn:self-similar-pointwise-tube-formula}
    V(\tiling,\ge)
    = \sum_{\pole \in \sD_{\tiling}} \res[s=\pole]{\gzT(\ge,s)} + \gl_d(\gen)\gzL(d).
  \end{equation}
  \begin{proof}
    The open set condition and nontriviality condition ensure that $\abscissa < d$. It remains to show that \gzL is strongly languid of some order $\languidOrder<2$ with constant $A = r_N$. Indeed, in view of \eqref{eqn:scaling-ratios-ordered} and \eqref{eqn:gzL-extended-to-C}, \gzL is strongly languid of order $\languidOrder = 0 < 2$ with constants $A = r_N$ and $C=1>0$, as in  Definition~\ref{def:strongly-languid}: 
    \linenopax
    \begin{align}\label{eqn:languid-of-order-0}
      |\gzL(s) |
      = \left| \frac1{1-\sum_{n=1}^N r_n^s}\right|
      \leq \left(r_N^{-1}\right)^{-|\Re s|},
      \qq \text{ as } \Re s \to -\iy.
    \end{align}
    Clearly, it follows from \eqref{eqn:pole-strip} that the sequence of screens $\{S_m\}_{m=1}^\iy$ in Definition~\ref{def:strongly-languid} may be chosen to be a sequence of vertical lines lying strictly to the left of $\min\{D_\ell,0\}$ and tending to $-\iy$. In particular, this ensures $\sup S_m < 0$ for all $m=1,2,\dots$. 
    Applying the second part of Theorem~\ref{thm:pointwise-tube-formula} with $A = r_N$, we deduce that the tube formula for \tiling has no error term and is given by \eqref{eqn:self-similar-pointwise-tube-formula} for all $0 < \ge < \min\{\genir , r_N^{-1}\genir\} = \genir$. (Note that since $r_N < 1$, we have $r_N^{-1} \genir > \genir$.)
  \end{proof}
\end{theorem}

\begin{remark}\label{rem:exact-formula-generalized}
  Theorem~\ref{thm:ptwise-result-self-similar-case} generalizes to higher dimensions the pointwise tube formula for self-similar strings (i.e. 1-dimensional self-similar tilings) obtained in \cite[Section~8.4]{FGCD}. The formula \eqref{eqn:self-similar-pointwise-tube-formula} holds pointwise, as opposed to the corresponding result in \cite[Theorem~8.3]{TFCD}, which was shown to hold only distributionally, also generalizes the tube formula for self-similar tilings (obtained in \cite[Theorem~8.3]{TFCD}) to generators which may not be monophase (or even pluriphase; see Remark~\ref{rem:monophase-and-pluriphase}). 
\end{remark}

Concerning the proof of Theorem~\ref{thm:ptwise-result-self-similar-case}, see also the discussion in \cite[Section~6.4]{FGCD} regarding the self-similar string \smash{$\sL = \{\scale_j\}_{j=1}^\iy$} (this is a \emph{generalized} self-similar fractal string, in the sense of \cite[Chapter~3]{FGCD}). The following more explicit form of Theorem~\ref{thm:ptwise-result-self-similar-case} is used to compute examples in Section~\ref{sec:Examples}.

\begin{cor}[Fractal tube formula]
  \label{thm:ptwise-result-self-similar-case-simplified}
  Assume, in addition to the hypotheses of Theorem~\ref{thm:ptwise-result-self-similar-case}, that the poles of the tubular zeta function \gzT are simple (which implies that $\sD_\sL$ and $\{0,1,\dots,d\}$ are disjoint). Then for all $0 < \ge < \genir $, we have the following exact tube formula:
  \linenopax
  \begin{align} \label{eqn:self-similar-pointwise-tube-formula-simplified}
    V(\tiling,\ge)
    &= \sum_{\pole\in \sD_{\sL}} c_\pole \ge^{d-\pole} 
     + \sum_{k=0}^{d} \left(c_k + e_k(\ge)\right) \ge^{d-k},
  \end{align}
  where
  \linenopax
  \begin{align} 
    c_\pole &:= \frac{\res{\gzL(s)}}{d-\pole}
    \sum_{k=0}^{d-1} \frac{\genir^{\pole-k}(d-k)}{\pole-k} \crv_k(\gen),
      && \text{for }\, \pole \in \sD_\sL, 
      \label{eqn:self-similar-pointwise-tube-formula-coefficients-cw}\\
    c_k &:= \crv_k(\gen)\gzL(k),
      && \text{for }\, k \in \{0,1,\dots,d\},
      \label{eqn:self-similar-pointwise-tube-formula-coefficients-ck} \\
    e_k(\ge) &:= \sum_{j=1}^{J(\ge)} \scale_j^k 
      \left(\crv_k(\gen, \scale_j^{-1}\ge) - \crv_k(\gen)\right),
      && \text{for }\, k \in \{0,1,\dots,d\}, 
      \label{eqn:self-similar-pointwise-tube-formula-coefficients-ek}
  \end{align}
  and $J(\ge) := \max\{j \geq 1 \suth \scale_j^{-1} \ge < \genir\} \vee 0$ as in \eqref{eqn:J(eps)}.
  Alternatively, one has
  \linenopax
  \begin{align} \label{eqn:self-similar-pointwise-tube-formula-condensed}
    V(\tiling,\ge)
    &= \sum_{\pole\in \sD_{\sL}} c_\pole \ge^{d-\pole} 
     + \sum_{k=0}^{d} c_k(\ge)\ge^{d-k},
  \end{align}
  where $c_\pole$ is as in \eqref{eqn:self-similar-pointwise-tube-formula-coefficients-cw} and $c_k(\ge) := c_k + e_k(\ge)$ with $c_k$ and $e_k(\ge)$ as in \eqref{eqn:self-similar-pointwise-tube-formula-coefficients-ck}--\eqref{eqn:self-similar-pointwise-tube-formula-coefficients-ek}, for $k \in \{0,1,\dots,d\}$.
\end{cor}

The proof of Corollary~\ref{thm:ptwise-result-self-similar-case-simplified} is postponed to Section~\ref{sec:proof-of-fractal-tube-formula}, as it is technical and depends on the terminology and technique developed in the first part of Section~\ref{sec:the-proof} (the proof of the tube formula for fractal sprays, Theorem~\ref{thm:pointwise-tube-formula}).
Corollary~\ref{thm:ptwise-result-self-similar-case-simplified} also allows us to recover the pointwise version of \cite[Corollary~8.7]{TFCD}, where the generator \gen was assumed to be monophase.

\begin{cor}[Fractal tube formula, monophase case]
  \label{thm:ptwise-result-self-similar-case-monophase}
  In addition to the hypotheses of Corollary~\ref{thm:ptwise-result-self-similar-case-simplified}, assume that \gen is monophase. Then, for all $0 < \ge < \genir$, we have the pointwise tube formula for self-similar tilings:
  \linenopax
  \begin{align}\label{eqn:ptwise-result-self-similar-case-monophase}
    V(\tiling,\ge) 
    = \sum_{\pole \in \sD_\sL} c_\pole \ge^{d-\pole} + \sum_{k=0}^d c_k \ge^{d-k}
    = \sum_{\pole \in \sD_\sT} c_\pole \ge^{d-\pole},
  \end{align}
  where $c_\pole$ (for $\pole \in \sD_\sL$) and $c_k$ (for $k=0,1,\dots,d$) are as in \eqref{eqn:self-similar-pointwise-tube-formula-coefficients-cw} and \eqref{eqn:self-similar-pointwise-tube-formula-coefficients-ck}, respectively.
  \begin{proof}
    When \gen is monophase, each function $\crv_k(\gen,\mydot)$ is constant and equal to $\crv_k(\gen)$, and hence $e_k(\ge) = 0$ for each $\ge>0$ and $k=0,1,\dots,d$, where $e_k(\ge)$ is as in \eqref{eqn:self-similar-pointwise-tube-formula-coefficients-ek}. Consequently, Corollary~\ref{thm:ptwise-result-self-similar-case-monophase} follows immediately from Corollary~\ref{thm:ptwise-result-self-similar-case-simplified}. 
  \end{proof}
\end{cor}

\begin{remark}\label{rem:general-spray}
  For an arbitrary fractal spray \tiling satisfying the hypotheses of the strongly languid case of Theorem~\ref{thm:pointwise-tube-formula}, 
  it follows from Theorem~\ref{thm:pointwise-tube-formula} (instead of Theorem~\ref{thm:ptwise-result-self-similar-case}),  that \eqref{eqn:self-similar-pointwise-tube-formula} holds pointwise for all $0 < \ge < \min\{\genir, A^{-1}\genir\}$.
  If, in addition, all the complex dimensions of \tiling are simple, then one can deduce from Lemma~\ref{thm:residues-of-gzT} (as in the proof of Corollary~\ref{thm:ptwise-result-self-similar-case-simplified}) that \eqref{eqn:self-similar-pointwise-tube-formula-simplified} holds; see also Remark~\ref{rem:residues-of-gzT} in this regard. Moreover, if \gen is assumed to be monophase, then \eqref{eqn:self-similar-pointwise-tube-formula-simplified} takes the simpler form \eqref{eqn:ptwise-result-self-similar-case-monophase}. A parallel remark holds (under the assumptions of the languid case of Theorem~\ref{thm:pointwise-tube-formula}) for the tube formulas with error term considered in Section~\ref{sec:Pointwise-tube-formulas-with-error-term}.
\end{remark}

\subsection{Pointwise tube formulas with error term}
\label{sec:Pointwise-tube-formulas-with-error-term}

\begin{theorem}[Pointwise tube formula with error term for self-similar tilings]
  \label{thm:Pointwise-tube-formula-with-error-term}
  Assume that \tiling is a self-similar tiling with generator \gen, and that a Steiner-like representation for \gen has been chosen. Let $S$ be a screen such that $S(0)<0$ (so that all integer dimensions are visible) and let $W$ be the associated window. Then for all $\ge > 0$,
  \linenopax
  \begin{align}\label{eqn:Pointwise-tube-formula-with-error-term}
    V(\tiling,\ge)
    = \sum_{\pole \in \DT(W)} \res{\gzT(\ge,s)} + \gl_d(\gen) \gzL(d) + \R(\ge),
  \end{align}
  where the error term $\R(\ge)$ is given explicitly as in \eqref{eqn:pointwise-error} and satisfies $\R(\ge) = O(\ge^{d-\sup S})$, as $\ge \to 0^+$.
  
  Moreover, if \gen is monophase, then this same conclusion holds without the assumption that $S(0)<0$, as long as $S$ avoids the set $\{0,1,\dots,d\}$.\footnote{In particular, this allows for a screen $S$ which lies arbitrarily close to the vertical line $\Re s = \abscissa$.} In addition, $\R(\ge)$ is equivalently given by \eqref{eqn:pointwise-error2} in this case.
  \begin{proof}
    This follows immediately from the first part of Theorem~\ref{thm:pointwise-tube-formula}, since the proof of Theorem~\ref{thm:ptwise-result-self-similar-case} implies \gzL is languid of order $\languidOrder = 0 < 1$ along any screen $S$.
    When \gen is monophase, the latter claim follows from Corollary~\ref{thm:fractal-spray-tube-formula-monophase}. Finally, it follows from the second part of Remark~\ref{rem:error-term} that in the monophase case, $\R(\ge)$ is equivalently given by \eqref{eqn:pointwise-error} or \eqref{eqn:pointwise-error2}.
  \end{proof}
\end{theorem}

The following result is the exact counterpart of Corollary~\ref{thm:ptwise-result-self-similar-case-simplified} (or of Corollary~\ref{thm:ptwise-result-self-similar-case-monophase} when \gen is monophase).

\begin{cor}[Fractal tube formula with error term]
  \label{thm:Fractal-tube-formula-with-error-term}
  Assume, in addition to the hypotheses of Theorem~\ref{thm:Pointwise-tube-formula-with-error-term}, that the visible poles of the tubular zeta function are simple (which implies that $\DL(W)$ and $\{0,1,\dots,d\}$ are disjoint). Then for all $\ge > 0$,
  \linenopax
  \begin{align}\label{eqn:Fractal-tube-formula-with-error-term}
    V(\tiling,\ge)
    = \sum_{\pole \in \DL(W)} \negsp[10] c_\pole \ge^{d-\pole} 
    + \sum_{k \in \{0,1,\dots,d\} \cap W} \negsp[20] (c_k + e_k(\ge)) \ge^{d-k} 
    + \R(\ge),
  \end{align}
  where the error term $\R(\ge)$ is as in \eqref{eqn:pointwise-error} and $c_\pole$, $c_k$, $e_k$ are as in \eqref{eqn:self-similar-pointwise-tube-formula-coefficients-cw}--\eqref{eqn:self-similar-pointwise-tube-formula-coefficients-ek}.
  
  Moreover, if \gen is assumed to be monophase, then \eqref{eqn:Fractal-tube-formula-with-error-term} holds for any screen which avoids the set $\{0,1,\dots,d\}$, and the formula takes the simpler form
  \linenopax
  \begin{align}\label{eqn:Fractal-tube-formula-with-error-term-simplified}
    V(\tiling,\ge)
    = \sum_{\pole \in \DL(W)} \negsp[10] c_\pole \ge^{d-\pole} 
    + \sum_{k \in \{0,1,\dots,d\} \cap W} \negsp[20] c_k \ge^{d-k} 
    + \R(\ge)\,,
  \end{align}
  with the error term $\R(\ge)$ as in \eqref{eqn:pointwise-error2}.
  \begin{proof}
    This follows from Theorem~\ref{thm:Pointwise-tube-formula-with-error-term} by the same methods as in Corollary~\ref{thm:ptwise-result-self-similar-case-simplified} (or Corollary~\ref{thm:ptwise-result-self-similar-case-monophase}, when \gen is monophase).
  \end{proof}
\end{cor}

\begin{remark}\label{rem:S(0)<0-yet-again}
  The significance of the assumption $S(0)<0$, and more importantly, the need for being able to omit it, is discussed in Remark~\ref{rem:0-in-W} and Section~\ref{rem:0-in-W-redux}. See also \cite{MMFS}.
\end{remark}


\section{Examples}
\label{sec:Examples} 

Firstly, it should be noted that Theorem~\ref{thm:ptwise-result-self-similar-case} implies that all tube formula results for the examples of self-similar tilings of \cite{TFCD,SST,GeometryOfSST} are now known to hold pointwise. This includes the Koch tiling (Figure~\ref{fig:koch-tiling} and \cite[Fig.~2 \& 3]{SST}), the Sierpinski gasket tiling (Figure~\ref{fig:sierpinski-gasket} and \cite[Fig.~6]{SST}), the Sierpinski carpet tiling (Figure~\ref{fig:sierpinski-carpet} and \cite[Fig.~7]{SST}), the pentagasket tiling (Figure~\ref{fig:pentagasket} and \cite[Fig.~5]{TFCD}), the Menger tiling (Figure~\ref{fig:menger-sponge} and \cite[Fig.~8]{SST}), and the three U-shaped examples from \cite[Fig.~3]{GeometryOfSST} (see Figure~\ref{fig:u-shaped} for one of them). The tube formulas of the first three of these examples can be found in \cite[Section~9]{TFCD}. 

In Figures~\ref{fig:koch-tiling}--\ref{fig:cantor_carpet_tiling} as well as in Figure~\ref{fig:u-shaped}, the following sets are shown from left to right.  The set $O$ is the initial open set of the tiling construction. (In all examples except the U-shaped one in Figure~\ref{fig:u-shaped}, $O$ is the interior of the convex hull of the underlying self-similar set.) The second set shows the generator(s) of the tiling (or, more precisely, the set $O\setminus\simt(\cj{O})$), while the subsequent ones give the first iterates of the generator(s) under the set mapping \simt. The right-most set always shows the union of all tiles of the generated tiling $\sT(O)$.

Of the self-similar tilings mentioned just above, only the Cantor carpet tiling and the U-shaped tiling will be studied in more detail below. Apart from illustrating how the tube formulas are applied in general, these two examples exhibit some important new features of the results obtained in this paper. Indeed, the Cantor carpet tiling (discussed in Section~ \ref{exm:Cantor Carpet Tiling}) has a generator which is not monophase (and not even pluriphase), a situation not covered by previous results. Furthermore, the U-shaped example (discussed in Section~\ref{exm:U-shaped}) has a generator which is itself fractal, in the sense that it has arbitrary small features and exhibits some kind of self-similarity.
Finally, the binary trees discussed in Section~\ref{exm:binary-tree} and the Apollonian packings discussed in Section~\ref{exm:Apollonian-packings} are natural examples of fractal sprays which are not self-similar tilings.

\begin{figure}
  \centering
  \scalebox{0.7}{\includegraphics{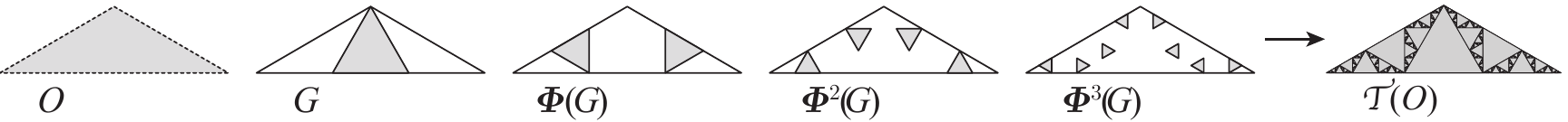}}
  \caption{The Koch curve tiling.}
  \label{fig:koch-tiling}
\end{figure}

\begin{figure}
  \centering
  \scalebox{0.80}{\includegraphics{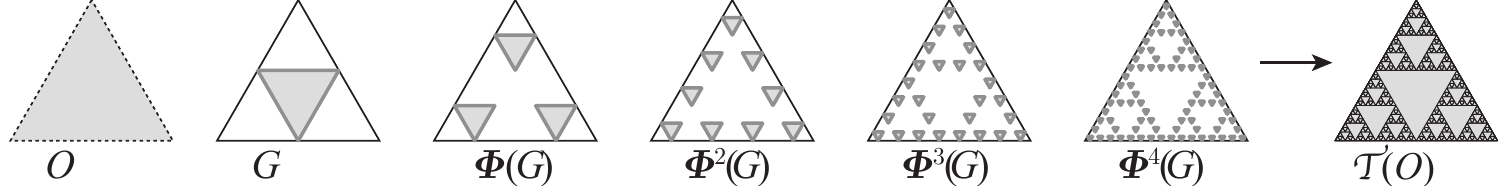}}
  \caption{The Sierpinski gasket tiling.}
  \label{fig:sierpinski-gasket}
\end{figure}

\begin{figure}
  \centering
  \scalebox{0.82}{\includegraphics{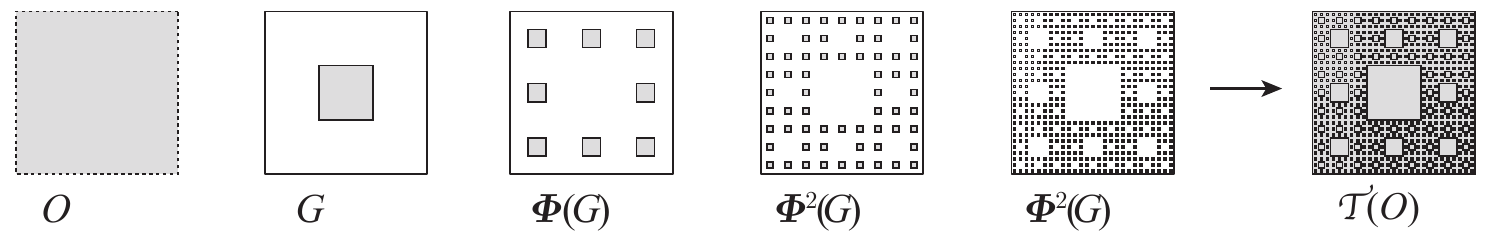}}
  \caption{The Sierpinski carpet tiling.}
  \label{fig:sierpinski-carpet}
\end{figure}

\begin{figure}
  \centering
  \scalebox{0.80}{\includegraphics{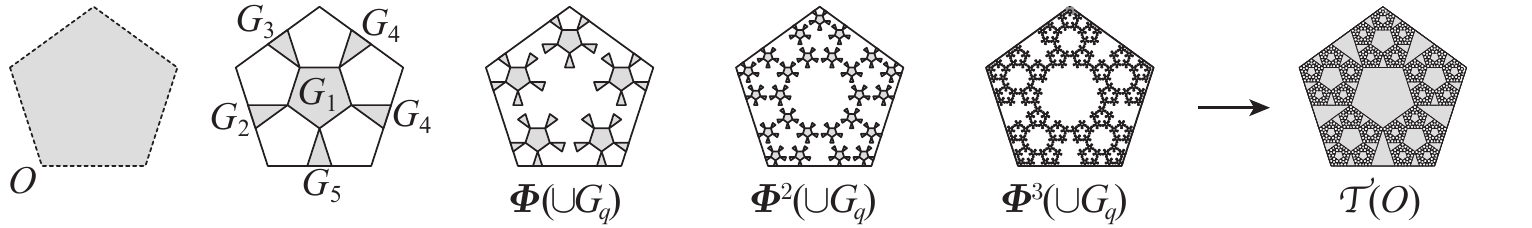}}
  \caption{The pentagasket tiling has multiple generators: one equilateral pentagon and five isoceles triangles.}
  \label{fig:pentagasket}
\end{figure}

\begin{figure}
  \centering
  \scalebox{0.8}{\includegraphics{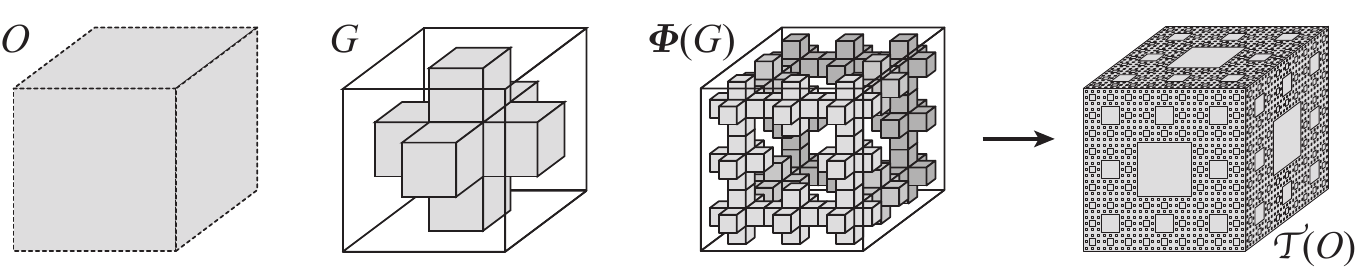}}
  \caption{The Menger sponge tiling has a Steiner-like generator which is neither convex nor pluriphase; see the computations for the Cantor carpet in Section~\ref{exm:Cantor Carpet Tiling}, for which the Menger sponge is a 3-dimensional analogue.}
  \label{fig:menger-sponge}
\end{figure}

\subsection{The Cantor carpet tiling} 
\label{exm:Cantor Carpet Tiling}

We consider the self-similar tiling associated to the Cartesian product $C \times C\ci\bR^2$ of the ternary Cantor set $C$ with itself; see Figure~\ref{fig:cantor_carpet_tiling}. By abuse of notation, we denote the associated self-similar tiling by $\sC^2$. 
  The fractal $C \times C$ is constructed via the self-similar system defined by the four maps
\[\simt_\j(x) = \tfrac13 x + \tfrac23 p_\j, \qq \j=1,\dots,4,\]
with common scaling ratio $r=\frac13$, and points $p_\j$ being the vertices of a square, as seen in Figure~\ref{fig:cantor_carpet_tiling}. Consequently, the corresponding string $\sL_{\sC^2} = \{\scale_j\}_{j=1}^\iy$ has scales 
  \linenopax
  \begin{align}\label{eqn:cantor_carpet-scales}
    \scale_j = 3^{-[\log_4 3j]},
    \qq j=1,2,\dots,
  \end{align}
where $[x]$ is the integer part of $x$.

\begin{figure}
  \centering
  \scalebox{0.90}{\includegraphics{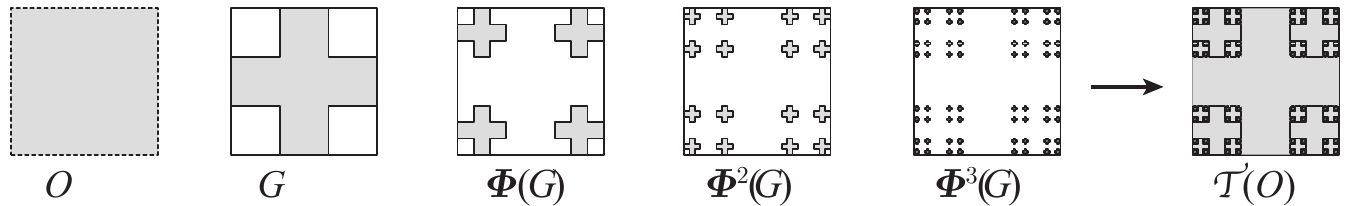}}
  \caption{The Cantor carpet tiling $\sC^2$.}
  \label{fig:cantor_carpet_tiling}
\end{figure}
\begin{figure}
  \centering
  \scalebox{0.70}{\includegraphics{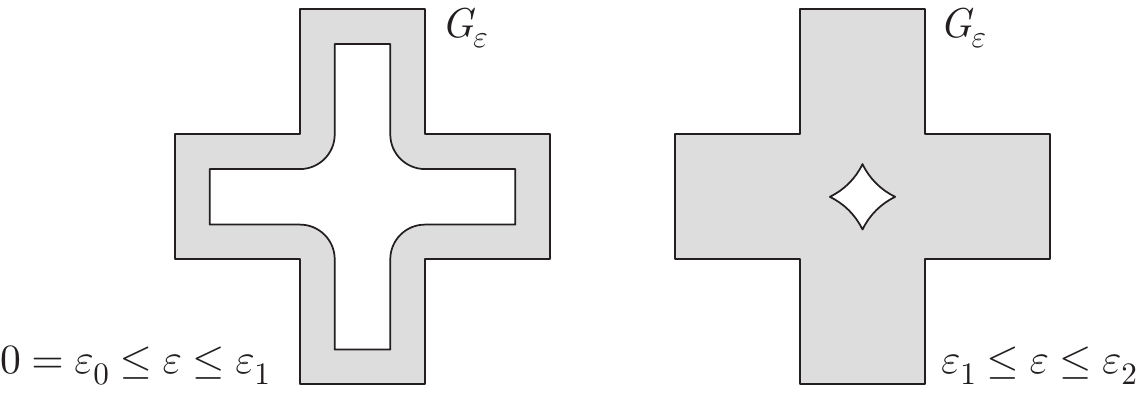}}
  \caption{The generator of the tiling $\sC^2$ is not pluriphase.}
  \label{fig:cantor_carpet_tiling_generator}
\end{figure}

The Cantor carpet tiling $\sC^2$ is discussed here because it has a generator \gen which is not monophase (and not even pluriphase), as seen in Figure~\ref{fig:cantor_carpet_tiling_generator} and formula \eqref{eqn:gen-tube-formula-for-Cantor-Carpet}. The inradius of the generator is $\genir = \gr(\gen) = \ell/(3\sqrt2)$, where $\ell$ is the side length of the initial square (we set $\ell=1$ in the sequel), and the relevant partition of the \ge-interval $(0,\genir]$ is
\linenopax
  \begin{align*}
  \{\ge_0=0, \ge_1=\tfrac{\genir}{\sqrt2}=\tfrac16, \ge_2=\genir=\tfrac{\sqrt2}6\}.
\end{align*}

The tube formula for the generator of this tiling is given by the following Steiner-like (but clearly not pluriphase) representation:
\linenopax
  \begin{align}\label{eqn:gen-tube-formula-for-Cantor-Carpet}
  V(\gen,\ge) =
  \begin{cases}
    (\gp-8)\ge^2 + 12\sqrt2\genir\ge, &0 < \ge \leq \tfrac{\genir}{\sqrt2},\\
    \gp \ge^2 - 4\arccos\left(\frac{\genir}{\ge\sqrt2}\right)\ge^2 + 2\genir\sqrt{2\ge^2-\genir^2} + 8\genir^2, &\tfrac{\genir}{\sqrt2} < \ge \leq \genir,\\
    10\genir^2, &\ge \geq \genir.
  \end{cases}
\end{align}
Here, for $\ge_1 < \ge \leq \ge_2$, the constant term $8\genir^2 = \frac49$ in \eqref{eqn:gen-tube-formula-for-Cantor-Carpet} gives the area of the four ``protrusions'' of \gen which are completely contained in $\gen_{-\ge}$. By \eqref{eqn:gen-tube-formula-for-Cantor-Carpet}, we can take the coefficient functions $\crv_k(\gen,\ge)$ to be
\linenopax
\begin{align}
  \crv_0(\gen,\ge) &=
  \begin{cases}
    \gp-8,  &0 < \ge \leq \tfrac{\genir}{\sqrt2},\\
    \gp - 4\arccos\left(\frac{\genir}{\ge\sqrt2}\right), &\tfrac{\genir}{\sqrt2} < \ge \leq \genir,
  \end{cases} \notag \\
  \crv_1(\gen,\ge) &=
  \begin{cases}
    12\sqrt2\genir,  &0 < \ge \leq \tfrac{\genir}{\sqrt2},\\
    \frac{2\genir}{\ge} \sqrt{2\ge^2 - \genir^2}, &\tfrac{\genir}{\sqrt2} < \ge \leq \genir,
  \end{cases} \label{eqn:cantor-carpet-crv1(gen)} \\
  \crv_2(\gen,\ge) &=
  \begin{cases}
    0,  &0 < \ge \leq \tfrac{\genir}{\sqrt2},\\
    8\genir^2, &\tfrac{\genir}{\sqrt2} < \ge \leq \genir,
  \end{cases} \notag
\end{align}

Since $\genir = \sqrt2/6$ and $\crv_k(\gen,\genir) = \crv_k(\gen)$ for $k=0,1,2$, it follows that 
\linenopax
  \begin{align}\label{eqn:cantor-carpet-kappa(G)s}
     \crv_0(\gen) = 0, \q
     \crv_1(\gen) = 2\genir = \frac{\sqrt2}3, \q
     \text{ and }\q
     \crv_2(\gen) = 8\genir^2 = \frac49.
  \end{align}
Note that according to \eqref{eqn:cantor-carpet-crv1(gen)}, each function $\crv_k(\gen,\ge)$ has a discontinuity at $\genir/\sqrt2$ but is analytic on each of the two intervals of the partition. Hence, it is piecewise analytic on $(0,\genir]$ in the sense of Section~\ref{rem:Piecewise-analytic-SLreps}.

From \eqref{eqn:cantor_carpet-scales}, the scale $\tfrac1{3^k}$ appears with multiplicity $4^{k}$, for $k=0,1,2,\dots$, so the scaling zeta function is
\linenopax
  \begin{align}
  \gzL(s) = \frac1{1 - 4 \cdot 3^{-s}},
  \qq s \in \bC.
\end{align}
It follows that the scaling complex dimensions are simple, and given by
\linenopax
\begin{align}
  \sD_\sL = \{\abscissa + \ii n\per \suth n \in \bZ\}
  \qq \text{with } \abscissa = \log_3 4, \; \per=\tfrac{2\gp}{\log 3},
\end{align}
and the corresponding residues are
\linenopax
\begin{align}\label{eqn:cantor-carpet-residues}
  \res[s=\abscissa + \ii n\per]{\gzL(s)}
  = \frac1{\log 3},
  \qq \text{for all } n \in \bZ.
\end{align}

Finally, we have the disjoint union $\DT = \sD_\sL \cup \{0,1,2\}$. 
All that remains is the substitution of the above quantities into the formula given in Corollary~\ref{thm:ptwise-result-self-similar-case-simplified}. We obtain 
%
  \linenopax
  \begin{align}
    V(\tiling,\ge)
    &= \frac1{\log 3} \sum_{n \in \bZ} \sum_{k=0}^{1} 
    \frac{\genir^{D-k+\ii n\per}(2-k)\crv_k(\gen)}{(D-k+\ii n\per)(2-D-\ii n\per)} 
    \ge^{2-D-\ii n\per} \notag \\
    &\hstr[4] + \sum_{k=0}^{2} \left(\frac{\crv_k(\gen)}{1-4\cdot3^{-k}} + \sum_{j=1}^{J(\ge)} 3^{-k[\log_4 3j]} \left(\crv_k(\gen, 3^{[\log_4 3j]} \ge) - \crv_k(\gen)\right)\right) \ge^{2-k},
      \label{eqn:cantor-carpet-tube-formula-simplified}
  \end{align}
  where $J(\ge) := \max\{j \geq 1 \suth \scale_j^{-1} \ge < \genir\} \vee 0$ as in \eqref{eqn:J(eps)}, and $[x]$ is the integer part of $x$. 
The computations for higher-dimensional analogues (like the Menger sponge) are extremely similar. In each case, the only complication is to obtain the tube formula for the generator. Observe that \tiling is a lattice tiling in the sense of Proposition~\ref{thm:lattice-nonlattice-dichotomy}.

\subsection{U-shaped modification of the Sierpinski carpet} 
\label{exm:U-shaped} 
  The U-shaped fractal of Figure~\ref{fig:u-shaped} is a modification of the Sierpinski carpet obtained by removing one contraction mapping from the self-similar system, and composing some of the remaining mappings with rotations of $\pm \gp/2$.
  The generator $\gen = \gen_1$ of $U$ from Figure~\ref{fig:u-shaped} provides an example of why it is useful to remove the requirement that $\lim_{\ge\to0^+} \crv_k(\gen,\ge)$ exists from Definition~\ref{def:Steiner-like} (Steiner-like)%
  \footnote{This assumption was part of the definition of \emph{Steiner-like} in \cite{TFCD} but was removed in the present paper.}%
  ; see Figure~\ref{fig:oscilliphase}. 

\begin{figure}
  \centering
  \scalebox{0.7}{\includegraphics{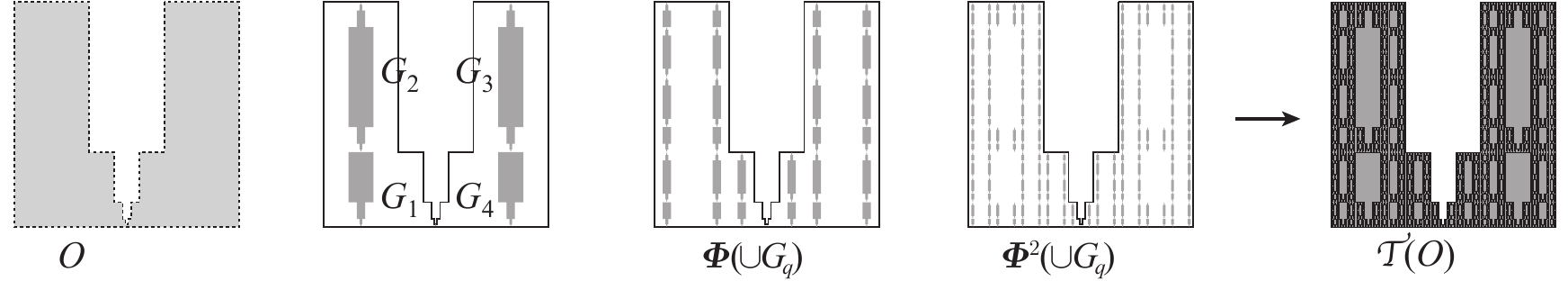}}
  \caption{The U-shaped example $U$ discussed in Section~\ref{exm:U-shaped}.}
  \label{fig:u-shaped}
\end{figure}

\begin{figure}
  \centering
  \scalebox{0.7}{\includegraphics{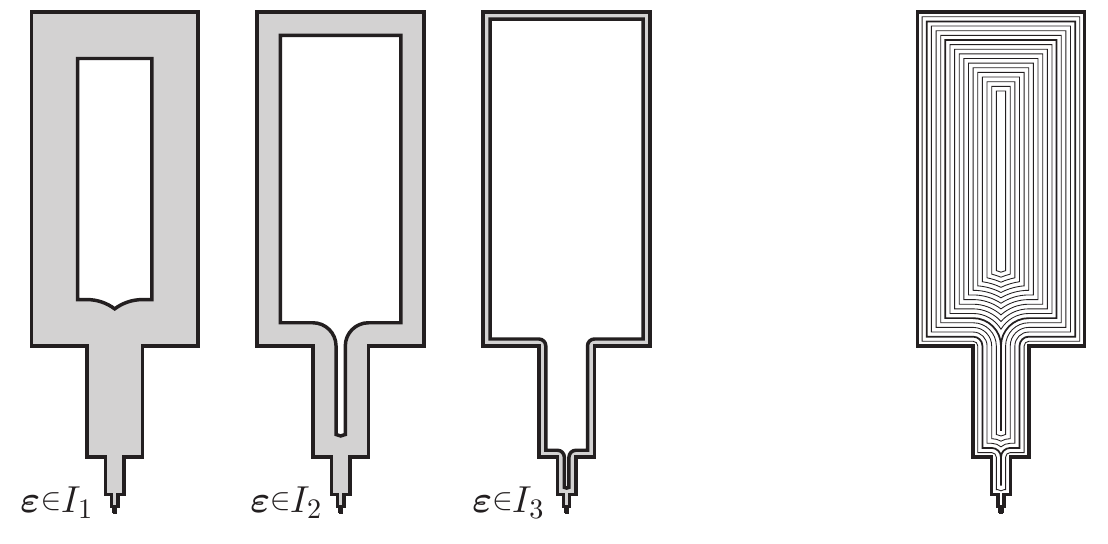}}
  \caption{The generator $\gen = \gen_1$ of $U$ from Figure~\ref{fig:u-shaped}; see Section~\ref{exm:U-shaped}.}
  \label{fig:oscilliphase}
\end{figure}

\begin{figure}
  \centering
  \scalebox{0.7}{\includegraphics{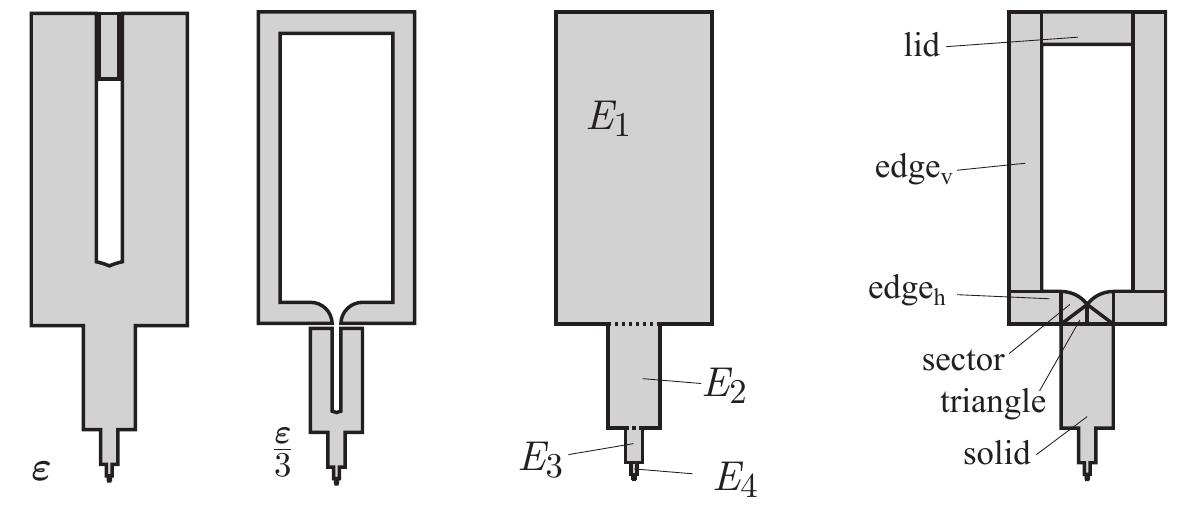}}
  \caption{The relation of $V(\gen,\ge)$ to $V(\gen,\frac{\ge}3)$.}
  \label{fig:oscilliphase-chambers}
\end{figure}
  
  To discuss $\gen = \gen_1$, let us consider the countable partition of $[0, \genir)$ defined by the sequence of intervals $I_m = [\frac{\genir}{3^{m}}, \frac{\genir}{3^{m-1}})$, for $m = 1,2,\dots$. Then the function $m(\ge) := [-\log_3 2\ge]$ gives the index $m=m(\ge)$ for which $\ge \in I_m$.   
  
  In this example, $V(\gen,\ge)$ satisfies a recurrence relation (see the left-hand side of Figure~\ref{fig:oscilliphase-chambers}) given for $\ge \in I_1$ by
  \linenopax
  \begin{align}\label{eqn:oscilliphase-recurrence}
    V(\gen,\ge) 
    = 9 V\left(\gen,\tfrac{\ge}3\right) 
    + \frac{17}9 - \frac{\ge}9 + \left(\gp -\frac{38}9\right)\ge^2 
    \qq \text{ for }\frac{\genir}{3} \leq \ge < \genir,   
  \end{align}
The generator $\gen$ is a union of countably many rectangles whose interiors are disjoint; consider these rectangles as defining a  sequence of ``chambers'' $\{E_m\}_{m=1}^\iy$, as depicted in Figure~\ref{fig:oscilliphase-chambers}. For $\ge \in I_m$, the constant term in $V(\gen,\ge)$ (corresponding to the region labelled ``solid'' in Figure~\ref{fig:oscilliphase-chambers}) is given by the volume of $\bigcup_{\ell=m+1}^\iy E_\ell$, which is
  \linenopax
  \begin{align}\label{eqn:oscilliphase-solid}
    \gl_2\left(\bigcup_{\ell=m+1}^\iy E_{\ell} \right)
    = \gl_2 \left(\frac1{3^m} \gen \right)
    = \frac1{9^m} \gl_2(\gen)
    = \frac1{9^{m+2}} \cdot \frac14.
  \end{align}

\subsection{A binary tree} 
\label{exm:binary-tree}

In this section, we consider the example of a binary tree embedded in $\bR^2$ in a certain way. This example shows how a very slight modification can change a monophase generator to a pluriphase generator, and also how one can compute the tube formula for a set which is not a self-similar fractal (but which does have some self-similarity properties). 

Consider the fractal sprays depicted in Figure~\ref{fig:binary-trees}. Each of these figures is formed by an equilateral triangle whose top vertex is the point $\gx =(1/2, \sqrt3/2)$ and whose base is the unit interval. Beginning at \gx and proceeding down one side of the triangle, one reaches the first branching at the point located $\frac23$ of the way to the bottom in (a) and at the point located $\frac34$ of the way to the bottom in (b). Consequently, the leaves of the first tree are the points of the usual ternary Cantor set, and the leaves of the second tree are the points of the (self-similar) Cantor set which is the attractor of the system $\{\simm_1(x) = \frac x4, \simm_2(x) = \frac x4 + \frac34\}$. It is clear from the ``phase diagram'' to the right of each spray that (a) is monophase and (b) is pluriphase.
\begin{figure}
  \centering
  \scalebox{0.9}{\includegraphics{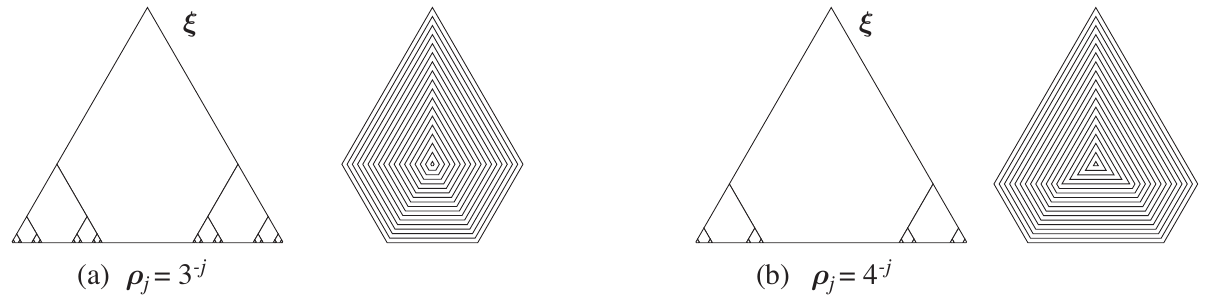}}
  \caption{The two binary trees discussed in Section~\ref{exm:binary-tree}. The left-hand one has a monophase generator and scaling ratios of the form $\scale_j = 3^{-j}$; the right-hand figure has a pluriphase generator and scaling ratios of the form $\scale_j = 4^{-j}$.}
  \label{fig:binary-trees}
\end{figure}


\subsection{Apollonian packings} 
\label{exm:Apollonian-packings}

We consider the fractal spray associated to an Apollonian packing; see Figure~\ref{fig:apollonian-packing}. Recall that the construction of this packing begins with three mutually tangent circles contained in a disk which is mutually tangent to all three. For the next stage of the construction, a new circle is inserted into each lune so as to be tangent to its three neighbours. The Apollonian packing is obtained by iterating infinitely many times. After removing the outermost disk, the rest of the circles in the packing form a fractal spray whose (monophase) generator is a disk, by \cite[Theorem~4.1]{GLMWYgg1}. 
This example of a fractal spray was suggested to us by Hafedh Herichi.
Full details on Apollonian packings and the Apollonian group may be found in \cite{GLMWYnt, GLMWYgg1}; we recommend the lecture notes \cite{Sarnak} for an introduction.

\begin{figure}
  \centering
  \scalebox{0.6}{\includegraphics{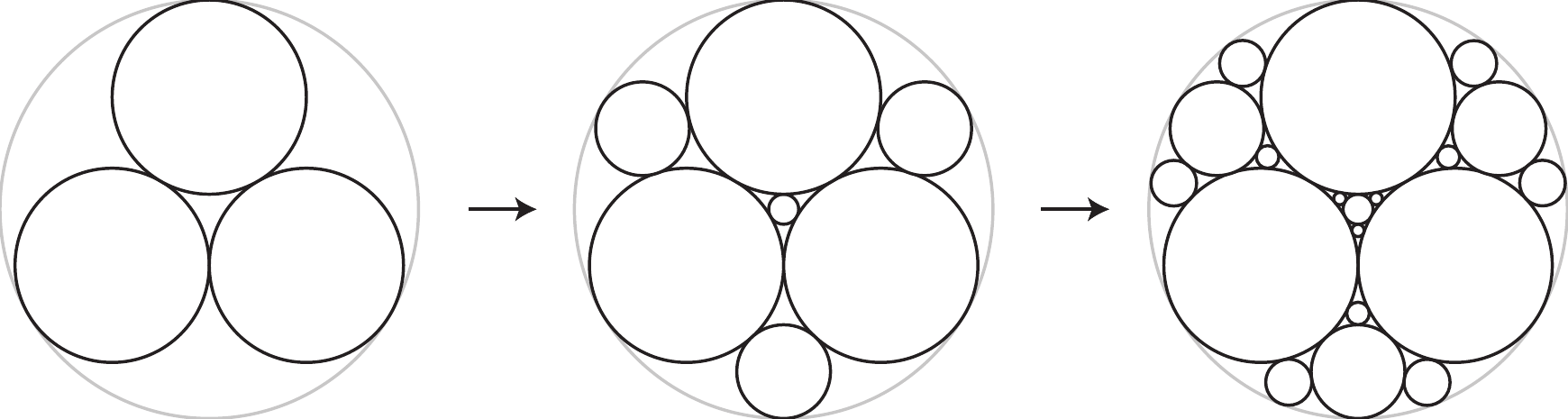}}
  \caption{The first three stages of the construction of the Apollonian packing for three circles with equal radii. The associated fractal spray does not include the outermost circle.}
  \label{fig:apollonian-packing}
\end{figure}

Apollonius' Theorem states that given any three mutually tangent circles $C_1, C_2, C_3$, there are exactly two circles $C_4^+, C_4^-$ that are tangent to the other three (allowing the possibility of a straight line as a circle of infinite radius). Thus, if we have any configuration of four circles, one may be removed and replaced by its counterpart; see Figure~\ref{fig:apollonian-action}. 

\begin{figure}
  \centering
  \scalebox{0.7}{\includegraphics{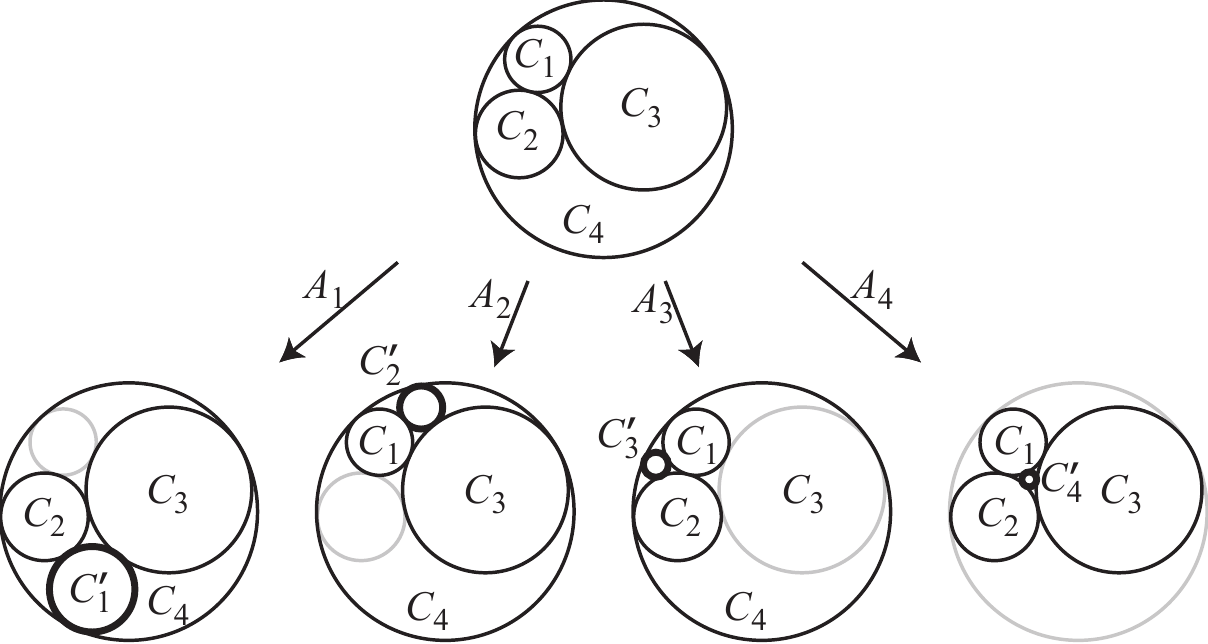}}
  \caption{The action of the Apollonian group on a configuration of 4 circles.}
  \label{fig:apollonian-action}
\end{figure}

Let $a = (a_1, a_2, a_3, a_4)$ be the 4-tuple whose entries are the reciprocal radii (i.e., the curvatures) of the four circles in a mutually tangent configuration. Descartes' Theorem states that these numbers must satisfy $F(a)=0$, where $F$ is the quadratic form 
\linenopax
\begin{align}\label{eqn:Descartes'-form}
  F(a) = 2(a_1^2 + a_2^2 + a_3^2 + a_4^2) - (a_1 + a_2 + a_3 + a_4)^2.
\end{align}
If we start with circles of given radii $a_1^{-1}, a_2^{-1}, a_3^{-1}$, then this allows us to find a fourth via
\linenopax
\begin{align*}
  a_4 = a_1 + a_2 + a_3 \pm 2\sqrt{a_1 a_2 + a_1 a_3 + a_2 a_3}.
\end{align*}
Thus, if we start with three circles of radius $a_1^{-1} = a_2^{-1} = a_3^{-1} = 1$, as in Figure~\ref{fig:apollonian-packing}, then the mutually tangent circle which encloses them will have radius \smash{$a_4^{-1} =(3-2\sqrt3)^{-1}$}.

%
%

For a starting configuration of four mutually tangent circles where one has negative curvature (so it encloses the other four), as in the top of Figure~\ref{fig:apollonian-action}, one can use the \emph{Apollonian group} (a subgroup of $S\negsp[3]L_4(\bZ)$ generated by matrices $A_1, A_2, A_3, A_4$) to geometrically obtain the other circles of the packing.   
Beginning with the configuration $a = (a_1, a_2, a_3, a_4)$, one replaces a circle $C_i$ with $C_i'$ (its reflection with respect to the other three) and the new inradius is obtained from the corresponding matrix multiplication. For example, swapping the first circle $C_1$ with its reflection $C_1'$ yields $a' = (a_1', a_2, a_3, a_4) = (a_1, a_2, a_3, a_4) A_1$, where $1/a_1'$ is the inradius of the new circle $C_1'$. 

Consequently, the scaling zeta function \gzL may be determined by collecting (with the proper multiplicities) the reciprocals of the entries of the 4-tuples
\linenopax
\begin{align*}
  \bigcup_{n=0}^\iy \bigcup_{\pole \in W^n} \{ a A_w \}, 
\end{align*}
where $W^n = \{1,2,3,4\}^n$ is the collection of 4-ary words of length $n$, and $A_w := A_{i_1} A_{i_2} \dots A_{i_n}$ is the matrix product corresponding to $w = i_1 i_2 \dots i_n \in W^n$.
%
If \gen is a disk of radius $r$, then it has inner tube formula 
\linenopax
\begin{align}\label{eqn:disk-inner-tube}
  V_ \gen(\ge) = 2\gp r \ge - \gp \ge^2.
\end{align}
This example has a monophase generator with coefficients $\crv_0 = -\gp$ and $\crv_1 = 2\gp$, when $r = \genir = \gr(\gen) = 1$.

Given the radii of four mutually tangent circles (one of which contains the other three), it is possible (but nontrivial) to determine the radii of the other circles in the packing (and also to obtain asymptotics for their rate of decay; see \cite{GLMWYgg1,GLMWYnt}). Therefore, we omit the full details of the tube formula for Apollonian packings. Note, however, that up to a normalizing factor, the scaling zeta function of the Apollonian packing is given by $\gzL(s) = \sum_{j=1}^\iy \scale_j^s$, where $\sL = \{\scale_j\}_{j=1}^\iy$ denotes the sequence of radii (i.e., the reciprocal curvatures) of the circles comprising the packing (written in nonincreasing order and according to multiplicity).

From \cite[Theorem~4.2]{GLMWYgg1}, the collection of circles in any packing is known to have Hausdorff dimension
\linenopax
\begin{align*}
  1.300197 < \dim_H < 1.314534.
\end{align*}
Note that all packings have the same Hausdorff dimension, since any two packings are equivalent by a Mobius transformation. Note also that the Hausdorff and Minkowski dimensions coincide for Apollonian packings; see \cite{Boyd} (and compare \cite[Theorem~1.10]{FGCD}).
We are hopeful that the methods of this paper will assist in the study of the Minkowski dimension of such objects. Moreover, the determination of the complex dimensions of an Apollonian packing is an interesting and challenging problem.


\section{Proofs of the main results}
\label{sec:the-proof}

\subsection{Proof of the tube formula for fractal sprays, Theorem~\ref{thm:pointwise-tube-formula}}
\label{sec:the-proof-sprays}
\hfill \\
The proof of Theorem~\ref{thm:pointwise-tube-formula} is divided into several steps. We begin with a discussion intended to motivate and explain the approach. 

  \begin{remark}[The philosophy behind the head and the tail]
  \label{rem:obs1}
  A technical part of the proof of Theorem~\ref{thm:pointwise-tube-formula} is inspired by the proof of the pointwise explicit formulas given in \cite[Chapter~5 and Section~ 8.1.1]{FGCD}. The main idea underlying the proof, however, is new and relies on the notions of ``head'' and ``tail''.
  
  In Section~\ref{sec:splitting}, we will split the tubular zeta function into a finite sum, which we call the head, and an infinite sum, which we call the tail (denoted by $\gzT[\head]$ and $\gzT[\tail]$, respectively); see \eqref{eqn:volume-zeta-split2}. Similarly, we will split the tube formula into corresponding finite and infinite sums ($V_{\mathrm{head}}$ and $V_{\mathrm{tail}}$, respectively); see \eqref{eqn:def:V-tail}. This decomposition will allow us to avoid repeating the same type of argument as appears in the proof of Proposition~\ref{prop:poles}, in several different instances. In particular, the decomposition into ``head'' and ``tail'' provides a technical device which allows us to use the Heaviside function as expressed in Lemma~\ref{lem:Heaviside-est}. 
  
  The ``head'' and ``tail'' decomposition is justified by the observation that for every fixed $J\in\bN$, the complex dimensions of a fractal string $\sL = \{\scale_j\}_{j=1}^\iy$ do not depend on the first $J$ scaling ratios $\scale_1,\dots,\scale_J$. This is the idea underpinning Proposition~\ref{prop:poles}: the zeta function
  $\gz_{\sL_J}$ of the string $\sL_J := \sL\setminus\{\scale_1,\scale_2,\dots,
  \scale_J\}=\{\scale_j\}_{j=J+1}^\iy$ has the same poles as the zeta
  function \gzL of the full string \sL. Indeed, the function
  $f$ defined by $f(s)=\gr^s$, where $\gr>0$ is some positive real
  number, is entire, and so is any finite sum of such functions.
  Hence, if \gzL is meromorphic in some connected open neighborhood \gO of $W$, then
  since
  \[
  \gzL(s)=\sum_{j=1}^{J} \scale_j^s + \gz_{\sL_J}(s)
  \]
  for each $s\in\gO$, the truncated zeta function $\gz_{\sL_J}$ is
  meromorphic in \gO and a
  point $\pole \in \gO$ is a pole of \gzL if and only if it is a
  pole of $\gz_{\sL_J}$. Moreover, the residues of \gzL and
  $\gz_{\sL_J}$ at $\pole$ obviously coincide. The same applies if
  the first scaling ratios are changed instead of omitted.
  
  In addition to being technically useful, the ``head and tail'' decomposition helps one understand the conceptual difference between the contributions made to the tube formula by the integer and scaling dimensions, and the origin of the error term. Indeed, the bulk of the proof of Theorem~\ref{thm:pointwise-tube-formula} lies in showing that $V_{\mathrm{head}}$ is given by the residues of $\gzT[\head](\ge,s)$ at the integer dimensions, and $V_{\mathrm{tail}}$ is given by the residues of $\gzT[\tail](\ge,s)$ at the scaling dimensions; cf.~\eqref{eqn:head-term-assertion} and \eqref{eqn:tail-term-assertion}.
\end{remark}

  \subsubsection{Splitting the tubular zeta function.} 
  \label{sec:splitting}
  For $k=0, 1, \dots, d$, we define the function $f_k:(0,\iy)\to \bR$ by
  \begin{equation}\label{eqn:def:f_k}
    f_k(\ge):= \crv_k(\gen,\ge)-\crv_k(\gen).
  \end{equation}
  This function measures the error of replacing the \kth coefficient function $\crv_k(\gen,\mydot)$ by the constant $\crv_k(\gen)$. 
  Note that $f_k(\ge) = 0$ for $\ge \geq \genir$ and that in case $\crv_k(G,\mydot)$ is constant, $f_k \equiv 0$.

  By employing the functions $f_k$ defined above, the tubular zeta function \gzT can be rewritten in a more convenient way (see Definition~\ref{def:volume-zeta-fn}):
  \linenopax
  \begin{align}\label{eqn:volume-zeta-presplit}
    \gzT(\ge,s)
    &= \ge^{d-s} \sum_{j=1}^\iy \scale_j^s
    \sum_{k=0}^{d}\frac{\genir^{s-k} f_k(\scale_j^{-1}\ge)}{s-k} + \ge^{d-s}
    \sum_{j=1}^\iy \scale_j^s \left(\sum_{k=0}^{d} \frac{\genir^{s-k} \crv_k(\gen)}{s-k} -\frac{\genir^{s-d} \lambda_d(G)}{s-d}\right).
  \end{align}
  In the second term, it is now possible to separate the \emph{scaling zeta function} \gzL, since the sum over $k$ does not depend on $j$ any longer. 
  Moreover, according to the fact that $\scale_j\to 0$ as $j\to\iy$, the first sum is taken only over finitely many integers $j$, namely, those indices for which $\scale_j \genir >\ge$. (In fact, it is the finiteness of the first sum which ensures that the expression on the right-hand side in \eqref{eqn:volume-zeta-presplit} converges absolutely exactly when the series defining $\gzT$ does. Hence, \eqref{eqn:volume-zeta-presplit} (as well as \eqref{eqn:volume-zeta-split2} below) holds for all $\ge>0$ and all $s\in\bC$ such that the second sum converges.) 
  Set
  \linenopax
  \begin{align}\label{eqn:def-rho_j}
    \inradius_j := \gr(\gen^j) = \scale_j \genir,
  \end{align}
  and recall that $J=J(\ge)$ is the largest index such that
  $\inradius_J >\ge$, cf.~\eqref{eqn:J(eps)}.
  Obviously, $\inradius_j$ is the inradius of the set $G^j = \gY_j(\gen)$ and $G^j \ci (G^j)_{-\ge}$ iff $j > J$.   
  Hence, for all $j > J$, 
  \linenopax
  \begin{align*}
    \scale_j^{-1}\ge \geq \genir
    \q\q\text{and}\q\q
    f_k( \scale_j^{-1}\ge)=0 \q \text{ for } k=0,1,\dots, d.
  \end{align*}
  Interchanging the order of summation in the first term of \eqref{eqn:volume-zeta-presplit}, and making use of \eqref{eqn:leb-coeff-rel}, we conclude that the tubular zeta function is given by
  \linenopax
  \begin{align} \label{eqn:volume-zeta-split2}
    \gzT(\ge,s)
    &= \underbrace{\ge^{d-s} \sum_{k=0}^{d} \frac{\genir^{s-k}}{s-k} \sum_{j=1}^{J}\scale_j^s f_k(\scale_j^{-1}\ge)}_{=:\gzT[\head](\ge,s)} + \underbrace{\frac{\ge^{d-s}\gzL(s)}{d-s} \left(\sum_{k=0}^{d-1} \frac{\genir^{s-k}}{s-k} (d-k)\crv_k(\gen)\right)}_{=:\gzT[\tail](\ge,s)}, 
  \end{align}
  with $J=J(\ge)$, as in \eqref{eqn:J(eps)}. Note that the \nth[d] term in $\gzT[\tail](\ge,s)$ vanishes because of the presence of the factor $(d-k)$ inside the sum.

In combination with the splitting in \eqref{eqn:volume-zeta-split2}, Proposition~\ref{prop:poles} yields the following result. 

\begin{theorem}
  \label{thm:Meromorphic-continuation-of-gzT} 
  Assume that $W$ is a window for \gzL and that \gW is a connected open neighborhood of $W$ in which \gzL is meromorphic. Fix an arbitrary $\ge>0$. Then:
  \begin{ThmList}
    \item 
    \emph{(Meromorphic continuation and poles of \gzT).} \;
     Both $\gzT(\ge,\mydot)$ and $\gzT[\tail](\ge,\mydot)$ are meromorphic in all of \gW. Furthermore, the set of visible poles of these functions is contained in $\sD_\tiling(W) = \DT \cap W$, as given in Definition~\ref{def:complex-dimn}. 
    %
    \item 
    \emph{(Head and tail decomposition of $\gzT$).} \;
    For every $s \in \gW$ (in particular, for every $s \in W$), the meromorphic continuation of $\gzT(\ge,\mydot)$ to \gW is given by
    \linenopax
  \begin{align}\label{eqn:gzT-decomp}
      \gzT(\ge,s) = \gzT[\head](\ge,s) + \gzT[\tail](\ge,s)\,.
    \end{align}
    Here, $\gzT[\head](\ge,\mydot)$ is given by 
    \linenopax
    \begin{align}\label{eqn:gzThead}
      \gzT[\head](\ge,s) = \ge^{d-s} \sum_{k=0}^d \frac{\genir^{s-k}}{s-k} \sum_{j=1}^{J(\ge)} \scale_j^s f_k(\scale_j^{-1}\ge),
    \end{align}
    which is meromorphic in all of \bC, with poles in $\{0,1,\dots,d\}$, 
    where $f_k$ is as in \eqref{eqn:def:f_k} and $J(\ge) := \max\{j \geq 1 \suth \scale_j^{-1} \ge < \genir\} \vee 0$,%
    \footnote{It may be useful to keep in mind that even though $J = J(\ge)$ is finite for every $\ge>0$, it tends monotonically to \iy as $\ge \to 0^+$, since $\scale_j$ decreases monotonically to $0$.}
    as in \eqref{eqn:J(eps)},
    while the meromorphic continuation of $\gzT[\tail](\ge,\mydot)$ to \gW is given by
    \linenopax
    \begin{align}\label{eqn:gzTtail}
      \gzT[\tail](\ge,s) = \frac{\ge^{d-s}}{d-s} \gzL(s) \sum_{k=0}^{d-1} \frac{\genir^{s-k}}{s-k} (d-k) \crv_k(\gen),
    \end{align}
    where $\gzL(s)$ denotes the meromorphic continuation of $\sum_{j=1}^\iy \scale_j^s$, as usual.
    \item 
    \emph{(Residues of \gzT).} \; 
    For each $\pole \in \sD_\sT(W) = (\DL \cup \{0,1,\dots,d\}) \cap W$, we have 
    \linenopax
  \begin{align}\label{eqn:residue-zeta-split}
    \res{\gzT(\ge,s)}
    = \res{\gzT[\head](\ge,s)} + \res{\gzT[\tail](\ge,s)},
  \end{align}
   which implies 
    \linenopax
    \begin{align}\label{eqn:res(gzt)=res(gzTtail)}
      \res{\gzT(\ge,s)} = \res{\gzT[\tail](\ge,s)}
    \end{align}
   for each $\pole \in \sD_\sL(W) \less \{0,1,\dots,d\}$, since $\gzT[\head](\ge,s)$ is holomorphic on $\bC \less \{0,1,\dots,d\}$. \\
    %
    (i) In the case when $\pole \in \sD_\sL(W) \less \{0,1,\dots,d\}$ is a \emph{simple} pole of $\gzT(\ge,\mydot)$ (and hence also a simple pole of \gzL), 
    \linenopax
    \begin{align}\label{eqn:value-of-res(gzT,w)}
      \res[\pole]{\gzT(\ge,s)} 
      = \frac{\ge^{d-\pole}}{d-\pole} \res[\pole]{\gzL(s)} \sum_{k=0}^{d-1} \frac{\genir^{\pole-k}}{\pole-k} (d-k) \crv_k(\gen). 
    \end{align}
    (ii) In the case when $\pole = k \in \{0,1,\dots,d\}$, we have 
    \linenopax
    \begin{align}\label{eqn:value-of-res(gzThead,k)}
      \res[k]{\gzT[\head](\ge,s)} 
      = \ge^{d-k} \sum _{j=1}^{J(\ge)} \scale_j^k f_k(\scale_j^{-1}\ge)\,.
    \end{align}  
    Furthermore, if $\pole = k \in \{0,1,\dots,d-1\}$ is a \emph{simple} pole of $\gzT(\ge,\mydot)$, we also have
    \linenopax
    \begin{align}\label{eqn:value-of-res(gzTtail,k)}
      \res[k]{\gzT[\tail](\ge,s)}
      = \ge^{d-k} \gzL(k) \crv_k(\gen).
    \end{align}  
    If $\pole = d \in W$ and $\abscissa < d$ (as assumed in Theorem~\ref{thm:pointwise-tube-formula}), then $d$ is not a pole of \gzL. Hence, formulas \eqref{eqn:gzTtail} and \eqref{eqn:leb-coeff-rel} imply that 
    \linenopax
    \begin{align}\label{eqn:res(gzT,d)}
      \res[d]{\gzT[\tail](\ge,s)} = \gzL(d) \left(\crv_d(\gen) - \gl_d(\gen)\right)\,.
    \end{align}
  \end{ThmList}
\end{theorem}

We leave it to the interested reader to perform the necessary (and elementary) computations needed to deal with the case when \pole is a multiple pole of \gzT; see \cite[Section~6.1.1]{FGCD}.

\begin{remark}\label{rem:integer-residue-as-Laurent-constant}
  If we assume that $\pole \in \{0,1,\dots,d-1\} \cap \sD_\sL(W)$ in part 3 of Theorem~\ref{thm:Meromorphic-continuation-of-gzT}, then one can express $\res{\gzT[\tail](\ge,s)}$ in terms of the constant term of the Laurent expansion of $\gzL(s)$ at $s=k$ (even in the case when $k$ is a multiple pole of \gzL), much as was done when $d=1$ in \cite[Corollary~8.10 and Remark~8.11]{FGCD}.
\end{remark}

  \subsubsection{Splitting the parallel volume} 
  For fixed $\ge > 0$, we split the inner parallel volume $V(\tiling,\ge)$ in a similar way as the tubular zeta function above. 
  Taking into account first \eqref{eqn:V(Gj,ge)} and then \eqref{eqn:def:f_k}, we can use \eqref{eqn:def-rho_j} to write
  \linenopax
  \begin{align} \label{eqn:split-T-volume}
  V(\tiling,\ge)
    &= \sum_{j: \inradius_j >  \ge} V(\gen^j,\ge) + \sum_{j: \inradius_j \leq  \ge} \gl_d(\gen^j) \notag \\
    &= \sum_{j: \inradius_j > \ge} \sum_{k=0}^{d} \ge^{d-k} \scale_j^k \crv_k(\gen,\scale_j^{-1}\ge)
     + \sum_{j: \inradius_j \leq  \ge} \gl_d(\gen^j)\\
    &= \sum_{j: \inradius_j >  \ge} \sum_{k=0}^{d} \ge^{d-k} \scale_j^k
      f_k(\scale_j^{-1}\ge)
    + \sum_{j: \inradius_j >  \ge} \sum_{k=0}^{d} \ge^{d-k}
      \scale_j^k \crv_k(\gen) + \sum_{j: \inradius_j \leq  \ge}
      \gl_d(\gen^j)
      .
      \notag
  \end{align}
  Recall that the sum over $j$ in the first two terms is finite for each fixed $\ge > 0$ and that the number of terms is given by $J=J(\ge)$.
  Therefore, in both terms, the sums can be interchanged.
  In the third term, the homogeneity of the $d$-dimensional Lebesgue measure $\gl_d$ implies $\gl_d(\gen^j)=\gl_d(\scale_j \gen)=\scale_j^d \gl_d(\gen)$ for each
  $j$. 
  Thus \eqref{eqn:leb-coeff-rel} yields
  \linenopax
  \begin{align}\label{eqn:def:V-tail}
    V(\tiling,\ge)
    &= \sum_{k=0}^{d} \ge^{d-k} \sum_{j=1}^{J} \scale_j^k f_k(\scale_j^{-1}\ge) + \sum_{k=0}^{d} \crv_k(\gen) \negsp[10] \sum_{j: \inradius_j >  \ge} \ge^{d-k} \scale_j^k
      \; + \; \gl_d(\gen) \negsp[10] \sum_{j: \inradius_j \leq  \ge} \scale_j^d \notag \\
        &= \underbrace{\sum_{k=0}^{d} \ge^{d-k} \sum_{j=1}^{J} \scale_j^k f_k(\scale_j^{-1}\ge)}_{=:V_{\mathrm{head}}(\tiling,\ge)} + \underbrace{\sum_{k=0}^{d}\crv_k(\gen) \left(\sum_{j: \inradius_j >  \ge} \ge^{d-k} \scale_j^k
      + \sum_{j: \inradius_j \leq  \ge} \genir^{d-k} \scale_j^d\right)}_{=:V_{\mathrm{tail}}(\tiling,\ge)}.  
  \end{align}


  \subsubsection{Outline of the remainder of the proof of the tube formula}
%
%
%
  
  In light of \eqref{eqn:def:V-tail}, the tube formula \eqref{eqn:pointwise-tube-formula} of Theorem~\ref{thm:pointwise-tube-formula} will follow upon verification of the following two assertions:
  \linenopax
  \begin{align}
    V_{\mathrm{head}}(\tiling,\ge)
    &= \sum_{\pole\in\sD_{\tiling}(W)} \negsp[8] \res[s=\pole]{\gzT[\head](\ge,s)}+\R_{\he}(\ge)\,, \q\text{and}
    \label{eqn:head-term-assertion} \\
    V_{\mathrm{tail}}(\tiling,\ge)
    &= \negsp[5] \sum_{\pole\in\sD_{\tiling}(W)} \negsp[8] \res[s=\pole]{\gzT[\tail](\ge,s)} +  \gl_d(G)\gzL(d) + \R_{\ta}(\ge),
    \label{eqn:tail-term-assertion}
  \end{align}
  where the error terms $\R_{\he}$ and $\R_{\ta}$ are given by 
  \linenopax
  \begin{align} 
    \R_{\he}(\ge) \label{eqn:pointwise-error-head}
    &= \sum_{k=0}^{[S(0)]} \res[s=k]{\gzT[\head](\ge,s)}\,, \q\text{and}\\ 
    \R_{\ta}(\ge) \label{eqn:pointwise-error-head-tail}
    &= \frac{1}{2\gp\ii} \int_S \gzT[\tail](\ge,s)\,ds\,,
  \end{align}
  respectively. (Here $[x]$ denotes the integer part of $x$.) 
  Indeed, the assumption $S(0)<0$ implies immediately $\R_{\he}\equiv 0$ and, by~\eqref{eqn:pointwise-error},  we have $\R(\ge)= \R_{\ta}(\ge)$. Therefore, if the formulas (\ref{eqn:head-term-assertion}) and
  (\ref{eqn:tail-term-assertion}) hold, then  since $\gzT(\ge,s) = \gzT[\head](\ge,s) + \gzT[\tail](\ge,s)$  by \eqref{eqn:volume-zeta-split2}, it follows from \eqref{eqn:def:V-tail} that 
  \linenopax
  \begin{align*}
    V(\tiling,\ge)
    &= V_{\mathrm{head}}(\tiling,\ge) + V_{\mathrm{tail}}(\tiling,\ge)  \\
    &=\negsp[15]\sum_{\pole\in\sD_{\tiling}(W)} \negsp[10] \res[s=\pole]{\gzT[\head](\ge,s)}+ \R_{\he}(\ge)\\
    &\hstr[10]+ \negsp[15]\sum_{\pole\in\sD_{\tiling}(W)} \negsp[10]\res[s=\pole]{\gzT[\tail](\ge,s)} +  \gl_d(G)\gzL(d)+ \R_{\ta}(\ge)\\
    &= \negsp[10]\sum_{\pole\in\sD_{\tiling}(W)} \negsp[10]\res[s=\pole]{\gzT[\head](\ge,s)+\gzT[\tail](\ge,s)} +  \gl_d(G)\gzL(d) + \R(\ge)\\
    &= \negsp[10]\sum_{\pole\in\sD_{\tiling}(W)} \negsp[10]\res[s=\pole]{\gzT(\ge,s)} +  \gl_d(G)\gzL(d) + \R(\ge)\,,
  \end{align*}
  and we obtain the tube formula \eqref{eqn:pointwise-tube-formula}, with $\R(\ge)$ given as in 
  formula \eqref{eqn:pointwise-error}:
  \linenopax
  \begin{align}\label{eqn:Rhead-expr}
    \R(\ge) 
    = \R_{\ta}(\ge) 
    = \frac1{2\gp\ii} \int_S \gzT[\tail](\ge,s) \,ds
  \end{align}
  Observe that in light of \eqref{eqn:volume-zeta-split2}, the integrand $\gzT[\tail](\ge,s)$ coincides with that of \eqref{eqn:pointwise-error}.
  
  We note that in order to complete the proof of Theorem~\ref{thm:pointwise-tube-formula}, we will still have to establish that the error term $\R_{\ta}(\ge)$ (which coincides with the full error term $\R(\ge)$, since $\R_{\he}(\ge)\equiv 0$, as noted in \eqref{eqn:Rhead-expr}) satisfies the asymptotic estimate $\R_{\ta}(\ge)=O(\ge^{d-\sup S})$ as $\ge\to 0^+$ (in the languid case) and that $\R_{\ta}(\ge)=0$ for all $0<\ge<\min\{g,A^{-1}g\}$ (in the strongly languid case). This will be accomplished, respectively, in Section~\ref{subsec:proof-error-term} and Section~ \ref{subsec:proof-strongly-languid}.
  
  We will establish independently (in Section~\ref{subsec:proof-head} and Section~ \ref{subsec:proof-tail}, respectively) that $V_{\mathrm{head}}(\tiling,\ge)$ can be expressed as a sum of residues of $\gzT[\head](\ge,s)$ (as given by \eqref{eqn:head-term-assertion}) and $V_{\mathrm{tail}}(\tiling,\ge)$ as a sum of residues of $\gzT[\tail](\ge,s)$ (as given by \eqref{eqn:tail-term-assertion}). 
  While $\gzT[\head](\ge,\mydot)$ is meromorphic in $\bC$ and has poles only at the integer values $0,1,\dots, d$, the function $\gzT[\tail](\ge,s)$ recovers the full set of the complex dimensions of \sL (compare Theorem~\ref{thm:Meromorphic-continuation-of-gzT}).
  This fits perfectly with the observation (compare Remark~\ref{rem:obs1}) that the first lengths of a fractal string (or, as here, the first scaled copies of \gen in a fractal spray) do not affect its complex dimensions. 
  The scaling complex dimensions are hidden in the ``tail''.
  The derivation of the first part is elementary and exact in that the error term just collects the residues at the integer dimensions not contained in $W$. 
  The derivation of the second part is more involved. 
  It uses techniques similar to those used in the proof of \cite[Theorem~8.7]{FGCD}, the pointwise tube formula for (1-dimensional) fractal strings. 
  Here, it is necessary for \sL to be languid (for the first part of Theorem~\ref{thm:pointwise-tube-formula}) or strongly languid (for the second part).

  \subsubsection{Proof of \eqref{eqn:head-term-assertion}} \label{subsec:proof-head}
  %
  The residue of $\gzT[\head](\ge,s)$ at $s=k \in \{0,1,\dots,d\}$ is given by formula \eqref{eqn:value-of-res(gzThead,k)} in Theorem~\ref{thm:Meromorphic-continuation-of-gzT}. Observing that the $k$-th term of $V_{\mathrm{head}}(\tiling,\ge)$ in \eqref{eqn:def:V-tail} has exactly the same expression, we conclude that  
  \linenopax
  \begin{align*}
    V_{\mathrm{head}}(\tiling,\ge)
    &= \sum_{k=0}^d\res[s=k]{\gzT[\head](\ge,s)}.
  \end{align*}  
For a fixed screen $S$ (and a corresponding window $W$) such that $S(0)\not\in\{0,\ldots,d\}$, we can split this sum into two parts, according to whether $k$ is contained in the interior of $W$ or in the complement $W^c$. 
Since $\gzT[\head](\ge,\mydot)$ has no poles outside the set $\{0,\ldots,d\}$, we can safely extend the first sum to include the residues at all complex dimensions visible in $W$. Thus
  \linenopax
  \begin{align*}
    V_{\mathrm{head}}(\tiling,\ge)
    &=\negsp[15]\sum_{\pole\in\sD_{\tiling}(W)} \negsp[10] \res[s=\pole]{\gzT[\head](\ge,s)}+ \negsp[15]\sum_{k\in\{0,\ldots,d\}\cap W^c} \negsp[10] \res[s=k]{\gzT[\head](\ge,s)}\,,
  \end{align*} 
 where the second sum is $\R_{\he}(\ge)$ (as defined in \eqref{eqn:pointwise-error-head}). This completes the proof of \eqref{eqn:head-term-assertion}.
We note that in the first part of Theorem~\ref{thm:pointwise-tube-formula}, the assumption $S(0)<0$ for the screen ensures that $\R_{\he}\equiv 0$. In Corollary~\ref{thm:fractal-spray-tube-formula-monophase}, where this assumption is dropped, we also have $\R_{\he}\equiv 0$, but for a different reason; indeed, in the monophase case, $\gzT[\head](\ge,s)$ itself vanishes identically. In the general case (which is not treated in this paper), when the integer dimensions may not all be visible (i.e., $S(0)>0$) and when $G$ is not necessarily monophase, the error term $\R_{\he}$ will have to be taken into account; see also Section~\ref{rem:0-in-W-redux}. 

  \subsubsection{The Heaviside function}

  Before continuing on to the proof of \eqref{eqn:tail-term-assertion}, we need to make some remarks on a certain useful form of the \emph{Heaviside function} $H:\bR\to \{0,1\}$, which is defined (as in \cite[(5.10)]{FGCD}) by
  \linenopax
  \begin{align}\label{eqn:H(x)}
    H(x) = 
    \begin{cases}
      0, &x<0 \\
      \tfrac12,& x=0, \\
      1, &x>0.
    \end{cases}
  \end{align}

  Following the proof of \cite[Theorem~8.7]{FGCD}, 
  we will exploit the following integral representation of the Heaviside function, which comes from number theory \cite[p.105]{Davenport}, and was refined in \cite[Lemma~5.1]{FGCD}.
  
  \begin{lemma}\label{lem:Heaviside-est}
  For $x,y,c>0$, the Heaviside function is given by
  \linenopax
  \begin{equation}\label{eqn:cor:Heavi-est}
    H(x-y) = \frac{1}{2\gp\ii} \int_{c-\ii \iy}^{c+\ii \iy} x^{s} y^{-s} \frac{ds}{s}.
    \end{equation}
  \begin{proof}
	Taking $k=1$ in \cite[Lemma~5.1]{FGCD} for a screen with
$T=T_+=-T_- >0$, the Heaviside function is approximated by
	\linenopax
	\begin{equation}\label{eqn:Heaviside-est}
	  H(x-y) = \frac{1}{2\gp\ii} \int_{c-\ii T}^{c+\ii T} x^{s}
	  y^{-s}\frac{ds}{s} + E,
	\end{equation}
	where the absolute value of the error term $E$ is bounded by
	\linenopax
	\begin{equation}\label{eqn:Heaviside-error-est}
	  \begin{cases}
	    \left(xy^{-1}\right)^c \frac{1}{T} \min\{T,|\log(xy^{-1})|^{-1}\},
	      &\mbox{ if } x \neq y, \\
	    \frac{c}{T},
	      &\mbox{ if } x=y.
	    \end{cases}
	\end{equation}
	It is now easily seen, that, for arbitrary fixed values
	$x,y,c>0$, the error term $E = E(T)$ vanishes as $T \to \iy$.
	Since this is true for all $x,y,c>0$, the result follows.
  \end{proof}
  \end{lemma}

  \subsubsection{Proof of \eqref{eqn:tail-term-assertion}}
  \label{subsec:proof-tail}
  The proof of this part follows roughly the lines of the proof of \cite[Theorem~8.7]{FGCD}. 
  One can rewrite the expression of $V_{\mathrm{tail}}(\tiling,\ge)$ in \eqref{eqn:def:V-tail} as 
  \linenopax
  \begin{equation} \label{eqn:tail-term2}
    V_{\mathrm{tail}}(\tiling,\ge)
    = \sum_{k=0}^{d-1} \genir^{-k} \crv_k(\gen) \left(\sum_{j: \gr_j \geq \ge} \ge^{d-k} \gr_j^k
    + \sum_{j: \gr_j < \ge}  \gr_j^d\right) + \crv_d(G) \gzL(d).
  \end{equation}
  For $k=0, 1, \dots, d-1$, denote the expression within the parentheses of \eqref{eqn:tail-term2} by $v_k(\ge)$.
  
  Using the Heaviside function as defined in \eqref{eqn:H(x)}, we write
  \linenopax
  \begin{align} \label{eqn:v_k-split}
    v_k(\ge)
    &= \sum_{j: \gr_j  \geq \ge}  \ge^{d-k} \gr_j^k  +  \sum_{j: \gr_j  < \ge} \gr_j^d
    = \sum_{j=1}^\iy \left[\ge^{d-k} \gr_j^k H(\gr_j -\ge) +  \gr_j^d H(\ge-\gr_j )\right].
  \end{align}
  %
  Note that, in case $\ge = \gr_j $ for some $j$, the corresponding \jth term in the sum on the right-hand side of \eqref{eqn:v_k-split} is 
  \linenopax
  \begin{align*}
    \ge^{d-k} \gr_j^k H(\gr_j -\ge) +  \gr_j^d H(\ge-\gr_j )
    = \tfrac{1}{2} \gr_j^{d-k} \gr_j^k + \tfrac{1}{2} \gr_j^d 
    =  \gr_j^d, 
  \end{align*}
  which equals the value given by the left-hand side of \eqref{eqn:v_k-split}.

  Now, fix some constant $c$ such that $d-1<c<d$ and $\abscissa<c$.
  (This is possible, since the abscissa of convergence $\abscissa$
  of \gzL was assumed to be strictly less than $d$.) Then
  $c-k$ and $d-c$ are positive numbers and so, by
  Lemma~\ref{lem:Heaviside-est}, the $j$-th term in the above sum
  is given by
  \linenopax
  \[
  \int_{c-k-\ii \iy}^{c-k+\ii \iy} \ge^{d-k-t} \gr_j^{k+t}
  \frac{dt}{2\gp\ii t} +\int_{d-c-\ii \iy}^{d-c+\ii \iy}  \ge^t
  \gr_j^{d-t} \frac{dt}{2\gp\ii t}.
  \]
  Substituting $s=t+k$ in the first integral and $s=d-t$ in the
  second one, we get
  \linenopax
  \[
  \int_{c-\ii \iy}^{c+\ii \iy} \ge^{d-s} \gr_j^{s}
  \frac{ds}{2\gp\ii (s-k)} +\int_{c-\ii \iy}^{c+\ii \iy}  \ge^{d-s}
  \gr_j^{s} \frac{ds}{2\gp\ii (d-s)}.
  \]
  Combining these integrals, one sees that \eqref{eqn:v_k-split} can be rewritten as follows: 
  \linenopax
  \begin{align} \label{eqn:v-k-2}
    v_k(\ge)
    &= \frac{d-k}{2\gp}  \sum_{j=1}^\iy\int_{-\iy}^{\iy}
       \frac{\ge^{d-c-\ii t} \gr_j^{c+\ii t}}{(c-k+\ii t)(d-c-\ii t)}dt.
  \end{align}

  This last integral converges absolutely, which is seen 
  as follows. Since $c$ was chosen strictly between $d-1$ and
  $d$, the numbers $d-c$ and $c-k$ are positive and so we have
  for all real numbers $t$,
  \linenopax
  \[
  |(c-k+\ii t)(d-c-\ii t)|= |(c-k)(d-c)+ t^2+ \ii (d-k)t|\geq
  (c-k)(d-c)+ t^2.
  \]
  Hence
  \linenopax
  \begin{align*}
    \left|\frac{\ge^{d-c-\ii t} \gr_j^{c+\ii t}}{(c-k+\ii t)(d-c-\ii   t)}\right| 
    &= \frac{\ge^{d-c} \gr_j^{c}|\ge^{-\ii t}| |\gr_j^{\ii   t}|}{|(c-k+\ii t)(d-c-\ii t)|}
    \le  \frac{\ge^{d-c} \gr_j^{c}}{(c-k)(d-c)+t^2},
  \end{align*}
  which implies
  \linenopax
  \begin{align*}
    \frac{1}{2\gp } \sum_{j=1}^\iy\int_{-\iy}^{\iy}
      \left|\frac{\ge^{d-c-\ii t} \gr_j^{c+\ii t}}{(c-k+\ii t)(d-c-\ii t)}\right|dt 
    &\leq  \frac{1}{2\gp} \sum_{j=1}^\iy \gr_j^{c}\ge^{d-c} \int_{-\iy}^{\iy} \frac{dt}{(c-k)(d-c)+t^2} \\
    &=  \frac{1}{\gp } \gzL(c) \ge^{d-c} \int_{0}^{\iy}
    \frac{dt}{(c-k)(d-c)+t^2}<\iy.
  \end{align*}
  Note that the last expression is finite because $c$ was chosen such that
  $c>\abscissa$, where \abscissa is the abscissa of convergence of \gzL. 
  It follows that the integrand in (\ref{eqn:v-k-2}) is absolutely integrable
  and it is safe to interchange the order of summation and
  integration. Hence
  \linenopax
  \begin{align*}
   v_k(\ge)
   &= \frac{1}{2\gp \ii} \int_{c-\ii \iy}^{c+\ii \iy}
      \ge^{d-s} \sum_{j=1}^\iy  \gr_j^{s} \frac{(d-k) }{(s-k)(d-s)}ds,
  \end{align*}
  where the sum $\sum_{j=1}^\iy \gr_j^{s}$ can be replaced by
  $\genir^s \gzL(s)$.
  %
  %
  Inserting the derived expressions for $v_k(\ge)$ into
  (\ref{eqn:tail-term2}), we obtain
  \linenopax
  \begin{align*}
    V_{\mathrm{tail}}(\tiling,\ge)
    &= \frac{1}{2\gp \ii} \sum_{k=0}^{d-1} \crv_k(\gen) \int_{c-\ii \iy}^{c+\ii \iy}
       \ge^{d-s} \gzL(s) \frac{\genir^{s-k} (d-k)}{(s-k)(d-s)}ds +  \crv_d(G)\gzL(d)\\
    &= \frac{1}{2\gp \ii} \int_{c-\ii \iy}^{c+\ii \iy}
       \frac{\ge^{d-s} \gzL(s)}{d-s}\sum_{k=0}^{d-1} \frac{\genir^{s-k}}{s-k} (d-k)\crv_k(\gen) ds +  \crv_d(G)\gzL(d) \\
    &= \frac{1}{2\gp \ii} \int_{c-\ii \iy}^{c+\ii \iy}
       \gzT[\tail](\ge,s) \, ds +  \crv_d(G)\gzL(d),
  \end{align*}
  by \eqref{eqn:volume-zeta-split2}.
  Since $d$ is not a pole of $\gzL$, it is a simple pole of $\gzT[\tail](\ge,\mydot)$ 
  and \eqref{eqn:res(gzT,d)} yields
  \linenopax
  \begin{align}\label{eqn:tail-residue-at-d}
    \res[s=d]{\gzT[\tail](\ge,s)}=\gzL(d)(\gk_d(G)-\gl_d(G)).
  \end{align}
  Therefore, we obtain
  \linenopax
  \begin{align*}
    V_{\mathrm{tail}}(\tiling,\ge)
    &= \frac{1}{2\gp \ii} \int_{c-\ii \iy}^{c+\ii \iy} \gzT[\tail](\ge,s) ds +  \res[s=d]{\gzT[\tail](\ge,s)}+\gzL(d)\gl_d(G).
  \end{align*}

  Now the machinery of the Residue Theorem can be applied.
  When pushing the line of integration towards the screen $S$, we
  collect on the way the residues of the poles of
  $\gzT[\tail](\ge,\mydot)$ that lie between the line $\Re s=c$
  and $S$ (see the proof of \cite[Theorem~8.7]{FGCD}). The
  definition of the screen $S$ and the window $W$ imply that
  \gzL is meromorphic in $W$ and, since $\abscissa < c$,
  \gzL has no poles to the right of the vertical line $\Re
  s=c$. Therefore, by Theorem~\ref{thm:Meromorphic-continuation-of-gzT}, any pole
  of $\gzT[\tail](\ge,\mydot)$ in the region between $\Re s=c$
  and $S$ is either contained in $\{0, 1, \dots, d-1\}$ or a pole
  of \gzL in $W$, i.e., an element of $\sD_{\tiling}(W)\setminus\{d\}$.
  Recall that $\gzT[\tail](\ge,\mydot)$ has another pole at
  $d$ but, since $d$ is not passed when pushing the line of
  integration towards the screen, it does not occur again.

  At this point, the languidness of \gzL comes into play.
  Using the sequence $\{T_n\}_{n\in\bZ}$ of Definition~\ref{def:languid}, we write
  $V_{\mathrm{tail}}(\tiling,\ge)$ as a limit of truncated
  integrals:
  \linenopax
  \begin{equation} \label{eqn:V-tail-limit}
    V_{\mathrm{tail}}(\tiling,\ge)
    = \lim_{n\to\iy} \frac{1}{2\gp \ii} \int_{c+\ii T_{-n}}^{c+\ii T_n}
      \gzT[\tail](\ge,s) ds +  \res[s=d]{\gzT[\tail](\ge,s)}+ \gzL(d)\gl_d(G).
  \end{equation}
  If we now replace the vertical line segment $\sC_{|n}:=[c+\ii T_{-n},c+\ii T_n]$
  of integration by the curve given by the union of the two
  horizontal line segments and the truncated screen $S_{|n}$, that is,
  \linenopax
  \begin{align*}
    U_{|n} &:= [c+\ii T_n, S(T_n)+\ii T_n], \\
    \genstr_{|n} &:= [c+\ii T_{-n}, S(T_{-n})+\ii T_{-n}], \q\text{and} \\
    S_{|n} &:= \{S(t)+\ii t: t\in [T_{-n},T_n]\}
  \end{align*}
  with proper orientations, the Residue Theorem implies that
  the \nth integral in \eqref{eqn:V-tail-limit} is equal to
  \linenopax
  \begin{equation*}
    \sum_{\pole\in\sD(W_{|n})} \res[s=\pole]{\gzT[\tail](\ge,s)}+ \sR_n(\ge) +
    U_n^\integral(\ge) + L^\integral_n(\ge),
  \end{equation*}
  where $\sD(W_{|n}):=\sD_{\tiling}(W)\cap W_{|n}$ is the set of
  (possible) poles of $\gzT[\tail](\ge,\mydot)$ that lie inside
  the region $W_{|n}$ bounded by the curves $U_{|n}, S_{|n}, \genstr_{|n}$ and
  $\sC_{|n}$. Hence, $W_{|n}$ is the ``truncated window'' associated to the truncated screen $S_{|n}$. The term $\sR_n(\ge)$ is given by the integral
  \linenopax
  \begin{equation} \label{error1}
    \sR_n(\ge)= \frac{1}{2\gp \ii} \int_{S_{|n}} \gzT[\tail](\ge,s) \, ds
  \end{equation}
  and $U_n^\integral(\ge)$ and $\genstr_n^\integral(\ge)$ are the corresponding integrals
  over the segments $U_{|n}$ and $\genstr_{|n}$, respectively (traversed clockwise around $W_{|n}$).
  More precisely,
  \smash{$U_n^\integral(\ge)$} is given by
  \linenopax
  \begin{align*}
    U_n^\integral(\ge)
    &= \frac{1}{2\gp \ii} \int_{S(T_n)+\ii T_n}^{c+\ii T_n} \gzT[\tail](\ge,s) \,ds\\
    &= \frac{1}{2\gp \ii} \int_{S(T_n)}^{c} \ge^{d-t-\ii T_n} \gzL(t+\ii T_n)
       \sum_{k=0}^{d-1} \frac{\genir^{t+\ii T_n-k} \crv_k(\gen)(d-k) }{(t-k+\ii T_n)(d-t-\ii T_n)} dt
  \end{align*}
  and is absolutely bounded as follows:
  \linenopax
  \begin{align*}
    |U_n^\integral(\ge)|
    &\leq \frac{1}{2\gp} \int_{S(T_n)}^{c} \ge^{d-t} |\gzL(t+\ii T_n)|
      \sum_{k=0}^{d-1} \frac{|\genir^{t-k} \crv_k(\gen)|(d-k) }{|t-k+\ii T_n||d-t-\ii T_n|} dt.
  \end{align*}
  According to the languidness condition \textbf{\ref{eqn:L1}} of Definition~\ref{def:languid} and the hypotheses of the first part of Theorem~\ref{thm:pointwise-tube-formula},
  there exist constants $C>0$ and
  $\languidOrder<1$, such that $|\gzL(t+\ii T_n)|\le C (T_n +1)^{\languidOrder}$.
  Moreover, $|t-k+\ii T_n|\geq T_n$, for all $k=0,\dots,d-1$ and,
  similarly, $|d-t-\ii T_n|\geq T_n$. Hence we get
  \linenopax
  \begin{equation} \label{u-n1}
    |U_n^\integral(\ge)| \le \frac{1}{2\gp} C (T_n +1)^{\languidOrder} \sum_{k=0}^{d-1}
  \frac{|\crv_k(\gen)|(d-k) }{T_n^2} \int_{S(T_n)}^{c} \genir^{t-k} \ge^{d-t} dt.
  \end{equation}
  Since $S(T_n) \geq \inf S$, the integral in this expression is
  bounded by a constant independent of $n$. Thus, there is a
  constant $C_1>0$, independent of $n$, such that
  \linenopax
  \begin{equation}\label{u-n2}
    |U_n^\integral(\ge)| \le C_1 \frac{(T_n +1)^{\languidOrder}}{T_n^2}.
  \end{equation}
  With similar arguments, one can show that the integral
  $\genstr_n^\integral(\ge)$ is absolutely bounded by $C_2 |T_{-n}|^{-2}(|T_{-n}|
  +1)^{\languidOrder}$, for some constant $C_2>0$ independent of $n$. If we
  now take limits as $n \to \iy$, then $T_n \to \iy$ and $T_{-n} \to -\iy$.
  Since $\languidOrder < 1$, this implies \smash{$|U_n^\integral(\ge)|$} and \smash{$|\genstr_n^\integral(\ge)|$} tend to $0$ as $n \to \iy$.

\subsubsection{Estimating the error term } \label{subsec:proof-error-term}
%
  To complete the proof of formula \eqref{eqn:tail-term-assertion}, it remains
  to show that the limit $\R_{\ta}(\ge):=\lim_{n\to\iy} \sR_n(\ge)$
  exists and satisfies the asymptotic estimate $\R_{\ta}(\ge)=O(\ge^{d-\sup S})$ as $\ge\to 0^+$, for which we utilize assumption \textbf{\ref{eqn:L2}} of Definition~\ref{def:languid}. Recall from \eqref{error1} and \eqref{eqn:volume-zeta-split2} that the integral $\R_n(\ge)$ is given by
  \linenopax
  \begin{align*}
    \R_n(\ge) &= \frac{1}{2\gp \ii} \int_{S_{|n}} \ge^{d-s} \gzL(s)
                \sum_{k=0}^{d-1} \frac{\genir^{s-k} \crv_k(\gen)(d-k) }{(s-k)(d-s)} ds \\
    &= \frac{1}{2\gp \ii} \int_{T_{-n}}^{T_{n}} \ge^{d-S(t)-\ii t}
       \gzL(S(t)+\ii t) \sum_{k=0}^{d-1}
       \frac{\genir^{S(t)+\ii t-k} \crv_k(\gen)(d-k) }{(S(t)+\ii t-k)(d-S(t)-\ii t)}(S'(t)+i) dt,
  \end{align*}
  where $S'(t)$ denotes the derivative of $S$ at $t$. Note that,
  since $S$ was assumed in Definition~\ref{def:screen} to be Lipschitz continuous with constant
  $\mathrm{Lip} S$, $S'(t)$ exists for almost all $t\in\bR$ and
  $|S'(t)|\le \mathrm{Lip} S$ at those points. Hence the integral
  above is well defined and absolutely integrable, which is seen as follows:
  \linenopax
  \begin{align} \label{eqn:Rn-abs}
    &\frac{1}{2\gp } \int_{T_{-n}}^{T_{n}} \ge^{d-S(t)} |\gzL(S(t)+\ii t)|
      \sum_{k=0}^{d-1} \frac{|\genir^{S(t)+\ii t-k} \crv_k(\gen)|(d-k) }{|S(t)-k +\ii t||d-S(t)+\ii t|}|S'(t)+i| dt \notag \\
    &\hstr[6]\leq \frac{M(\ge)(1+\mathrm{Lip} S)}{2\gp }
      \sum_{k=0}^{d-1} |\crv_k(\gen)|(d-k)   \int_{T_{-n}}^{T_{n}}
      \frac{|\gzL(S(t)+\ii t)| \, dt}{|S(t)-k +\ii t||d-S(t)+\ii t|},
  \end{align}
  where the number $M(\ge)$, defined by 
  \linenopax
  \begin{align*}
    M(\ge) 
    = \max\{\ge^{d-\sup S}, \ge^{d-\inf S}\} \cdot
      \max\{\genir^{\sup S},\genir^{\inf S}\} \cdot
      \max\{1,\genir^{-d}\}, 
  \end{align*}
  is a uniform upper bound (in $t$) for the term $\ge^{d-S(t)} \genir^{S(t)-k}$, for $k=0,\dots,d$.
  Now we use the languidness assumption \textbf{\ref{eqn:L2}}, which states
  that there exist constants $C>0$ and $\languidOrder<1$ such that
  $|\gz_s(S(t)+\ii t)|\le C |t|^{\languidOrder}$ for all $|t|\geq 1$.
  Observe that, since the screen $S$ avoids the poles of
  \gzL, the expression $|\gz_s(S(t)+\ii t)|$ is bounded on
  any finite interval for $t$. Therefore, \textbf{\ref{eqn:L2}} is equivalent
  to assuming that there are $C_1>0$ and $\languidOrder<1$ such that
  $|\gz_s(S(t)+\ii t)|\le C_1 |t|^{\languidOrder}$ for all $|t|\geq t_0$,
  where $t_0$ is some arbitrary but fixed positive constant.
  (Simply choose $C_1$ sufficiently large.) 
  Next, we describe how to choose $t_0$.
  Since the screen $S$ is assumed to be Lipschitz
  continuous and to avoid the numbers $\{0,\dots,d\}$ when
  passing the real axis, one can find positive constants $t_0$ and $r_0$
  such that $|k-S(t)|\geq r_0$ for all $|t|\le t_0$ and
  $k=0,\dots,d$. (That is, in a tube of width $t_0$ around the
  real axis, the screen $S$ has at least distance $r_0$ to any of
  the lines $\Re s =k$, for $k=0,\dots, d$.)

  Now, for the remaining integrals in the above expression (and
  $n$ sufficiently large), we split the interval of integration
  $(T_{-n},T_n)$ into $(T_{-n}, -t_0)\cup (-t_0, t_0)\cup (t_0,
  T_n)$. In the first and the third intervals, we use (the
  modified) condition \textbf{\ref{eqn:L2}} and, furthermore, that $|d-S(t)+\ii
  t|\geq |t|$ and $|S(t)-k +\ii t|\geq |t|$  to see that, for
  $k=0,1,\dots, d-1$,
  \linenopax
  \begin{align*}
    \int_{T_{-n}}^{-t_0}\frac{|\gzL(S(t)+\ii t)|}{|S(t)-k +\ii t||d-S(t)+\ii t|}  dt
    &\leq C_1 \int_{T_{-n}}^{-t_0} |t|^{\languidOrder-2} dt
    = \frac{C_1}{\languidOrder -1} \left(|T_{-n}|^{\languidOrder-1}-t_0^{\languidOrder-1}\right)
  \end{align*}
  and, similarly, that the \kth integral over the interval
  $(t_0,T_n)$ is bounded by the constant $\frac{C_1}{\languidOrder -1}
  \left(T_{n}^{\languidOrder-1}-t_0^{\languidOrder-1} \right)$.

  In the interval $(-t_0, t_0)$, $|\gzL(S(t)+\ii t)|$ is
  bounded by a constant, say $M$, $|S(t)-k +\ii t|\geq
  |S(t)-k|\geq r_0$ and, similarly,  $|d-S(t)+\ii t|\geq|d-S(t)|\geq r_0$. Therefore, for $k=0,1,\dots,d-1$,
  \linenopax
  \begin{align*}
    \int_{-t_0}^{t_0}\frac{|\gzL(S(t)+\ii t)|}{|S(t)-k +\ii t||d-S(t)+\ii t|} dt
    &\leq \frac{2M t_0}{r_0^{2}}
    =: C_2.
  \end{align*}
  Observe that the derived estimates for the \kth integrals are
  independent of $k$. Thus, putting the pieces back together, we
  have that \eqref{eqn:Rn-abs} is bounded above by
  \linenopax
  \begin{align} \label{eqn:Rn-abs2}
    &C(\ge)  \left(\frac{C_1}{1 - \languidOrder} \left(2 t_0^{\languidOrder-1} - T_n^{\languidOrder-1} - |T_{-n}|^{\languidOrder-1}\right) + C_2\right),
  \end{align}
  where $C(\ge):= \frac{M(\ge)(1+\mathrm{Lip} S)}{2\gp}
  \sum_{k=0}^{d-1} |\crv_k(\gen)|(d-k) $.
  Consequently, $\sR_n(\ge)$ is absolutely integrable for each $n$ and $\ge>0$. Moreover, since \eqref{eqn:Rn-abs2} converges to some finite value as $n \to \iy$ (because $\languidOrder <1$), it follows that also $\sR_{\ta}(\ge)$ is absolutely integrable and thus integrable; i.e., $\sR_{\ta}(\ge)$ is finite for each $\ge>0$. Hence, the error term $\sR_{\ta}(\ge)$ is given as claimed in \eqref{eqn:pointwise-error-head-tail}.
Finally, note that $W_{|n}\to W\cap \{\Re s< c\}$ and
  $\sD_{\tiling}(W)\cap \{\Re s\geq c\}=\{d\}$ imply
  $\sD(W_{|n})\to \sD_{\tiling}(W)\setminus\{d\}$. This completes the proof of
  formula (\ref{eqn:tail-term-assertion}).
  
  Furthermore, from (\ref{eqn:Rn-abs2}) and the definition of $M(\ge)$ (see the discussion following (\ref{eqn:Rn-abs})), it is clear that there is a constant
  $\hat{C}>0$ such that $|\sR_{\ta}(\ge)|\le \hat{C} \ge^{d-\sup S}$ for
  all $0<\ge<\genir$; i.e., $\sR_{\ta}(\ge)$ is of order $O(\ge^{d-\sup S})$
  as $\ge\to 0^+$. 
  Recalling that $\R=\R_{\ta}$, this completes the proof of the languid case in  Theorem~\ref{thm:pointwise-tube-formula}. 

  \subsubsection{The strongly languid case}
  \label{subsec:proof-strongly-languid}
  Now assume that \gzL is strongly languid of order $\languidOrder < 2$, as in Definition~\ref{def:strongly-languid} and the second part of Theorem~\ref{thm:pointwise-tube-formula}.
  Then there exists a sequence $S_m$ of screens and corresponding
  windows $W_m$ with $\sup S_m\to -\iy$ as $m\to\iy$ such that
  \textbf{\ref{eqn:L1}} and \textbf{\ref{eqn:L2'}} are satisfied for each $m$ (with
  constants $C,A>0$ independent of $m$). In addition, we may assume without loss of generality that $\sup S_m < S(0)$ for all $m \geq 1$ (see the discussion preceding Corollary~\ref{thm:fractal-spray-tube-formula-monophase}). 
  For each screen $S_m$
  and for fixed $n\in\bN$, consider the truncated screen $S_{m|n}$
  (truncated at $T_{-n}$ and $T_n$) and the corresponding
  truncated window $W_{m|n}$ bounded from above and below by the
  horizontal lines $\Im s=T_n$ and $\Im s= T_{-n}$ and from the
  right by the line $\Re s=c$. By the Residue Theorem, for each $m$ and
  $n$, the \nth integral in the sequence of truncated integrals
  in (the counterpart of) \eqref{eqn:V-tail-limit} is given by
  \linenopax
  \[
  \sum_{\pole\in\sD(W_{m|n})} \res[s=\pole]{\gzT[\tail](\ge,s)}+
  \sR_{m|n}(\ge) + U^\integral_{m|n}(\ge) + L^\integral_{m|n}(\ge),
  \]
  just as in the languid case, and the integrals \smash{$U^\integral_{m|n}(\ge)$   and $L^\integral_{m|n}(\ge)$} over the horizontal line segments are similar to \smash{$U_n^\integral(\ge)$} and \smash{$L^\integral_n(\ge)$} above, with $S$ replaced by $S_m$. First we keep $n$ fixed and show that $\sR_{m|n}(\ge)$ vanishes as $m\to\iy$.
  Note that $\sR_{m|n}(\ge)$ is given by the same expression as $\sR_n(\ge)$ in (\ref{error1}), except that the integral is now over $S_{m|n}$ instead of $S_{|n}$. Its absolute value is bounded by
  \linenopax
  \begin{align*} \label{eqn:Rmn-abs}
    \frac{1}{2\gp } \int_{T_{-n}}^{T_{n}} \ge^{d-S_m(t)} |\gzL(S_m(t)+\ii t)|
      \sum_{k=0}^{d-1} \frac{|\genir^{S_m(t)-k +\ii t} \crv_k(\gen)|(d-k) }{|S_m(t)-k +\ii t||d-S_m(t)+\ii t|}|S_m'(t)+i|  dt \\
    \leq \frac{B+1}{2\gp }\sum_{k=0}^{d-1} |\crv_k(\gen)|(d-k) 
      \int_{T_{-n}}^{T_{n}} \ge^{d-S_m(t)} \genir^{S_m(t)-k} \frac{|\gzL(S_m(t)+\ii t)|}{|t|^2}  dt,
  \end{align*}
  where we used the inequality $|S_m(t)-k +\ii t||d-S_m(t)+\ii
  t|\geq |t|^2$. Moreover, we utilized that, since the functions
  $S_m$ are assumed to be Lipschitz continuous with a uniform
  Lipschitz bound $B=\sup_m \mathrm{Lip} S_m<\iy$, the inequality
  $|S_m'(t)+i|\le B+1$ holds, whenever $S_m'(t)$ is defined
  (which is the case for almost all $t \in \bR$, independently of $m$).
  Now, by \textbf{\ref{eqn:L2'}} of Definition~\ref{def:strongly-languid}, there are constants $A,C>0$, independent of $n$ and $m$, such that, for
  all $t\in\bR$ and all $m\in\bN$, $|\gzL(S_m(t)+\ii t)|\le C
  A^{|S_m(t)|} (|t|+1)^\languidOrder$. Therefore, there exists a constant
  $C_1$, independent of $n$ and $m$, such that
  \linenopax
  \begin{align*}
    |\sR_{m|n}(\ge)|&\leq C_1
  \int_{T_{-n}}^{T_{n}} \left(\frac\ge\genir\right)^{-S_m(t)}A^{|S_m(t)|}\frac{(|t|+1)^\languidOrder}{|t|^2}
  dt.
  \end{align*}
  For $m$ sufficiently large (indeed, without loss of generality, for all $m \geq 1$), we have $S_m(t)<0$ and so
  $-S_m(t)=|S_m(t)|$. Thus, provided that $\ge < A^{-1} \genir$, we can
  bound the expression $(\ge/\genir)^{-S_m(t)} A^{|S_m(t)|}=(\ge
  A/\genir)^{|S_m(t)|}$ from above by $\left(\ge A/\genir\right)^{|\sup S_m|}$, which is
  independent of $t$ and can thus be taken out of the integral.
  The remaining integral has a finite value for each $n$. Letting
  now $m\to\iy$, $|\sup S_m|\to \iy$ and so $|\sR_{m|n}(\ge)|$
  vanishes.

  When taking the limit as $m\to \iy$, the expression
  \smash{$U^\integral_{m|n}(\ge)$} extends to an integral over the whole half-line
  $(-\iy+\ii T_n, c+\ii T_n]$ and \smash{$L^\integral_{m|n}(\ge)$} to an integral
  over $(-\iy+iT_{-n}, c+\ii T_{-n}]$. More precisely,
  \smash{$U^\integral_{|n}(\ge) := 
  \lim_{m\to\iy} U^\integral_{m|n}(\ge)$} is given by
  \linenopax
  \begin{align*}
    U^\integral_{|n}(\ge)
    &= \frac{1}{2\gp \ii} \int_{-\iy +\ii T_n}^{c+\ii T_n} \gzT[\tail](\ge,s) ds\\
    &= \frac{1}{2\gp \ii} \int_{-\iy}^{c} \ge^{d-t-\ii T_n} \gzL(t+\ii T_n)
      \sum_{k=0}^{d-1} \frac{\genir^{t-k+\ii T_n} \crv_k(\gen)(d-k) }{(t-k+\ii T_n)(d-t-\ii T_n)} dt.
  \end{align*}
  By exploiting the languidness condition \textbf{\ref{eqn:L1}} (which now
  holds for all $t \in \bR$) and the inequalities  $|t-k+\ii T_n| \geq T_n$
  (for $k=0,\dots,d-1$) and $|d-t-\ii T_n| \geq T_n$, it is easily
  seen that there exists some constant $C_2 > 0$, independent of $n$ and $m$, such that $U^\integral_{|n}(\ge)$ is absolutely bounded as follows:
  \linenopax
  \begin{equation} \label{eqn:u-n-L2}
    \left|U^\integral_{|n}(\ge) \right|
    \leq C_2 \frac{(T_n +1)^\languidOrder}{T_n^2} \int_{-\iy}^{c} \left(\frac\ge\genir\right)^{-t} dt.
  \end{equation}
  The remaining integral is finite, provided that $\ge < \genir$. Now,
  as $n \to \iy$, \smash{$|U^\integral_{|n}(\ge)|$} vanishes, for each $\ge < \genir$, and
  with completely analogous arguments, the same can be shown for
  \smash{$|U^\integral_{|n}(\ge)|$}. Hence the tail volume is given in the strongly languid case by
  \linenopax
  \begin{equation}
    V_{\mathrm{tail}}(\tiling,\ge)
    = \sum_{\pole\in\sD_{\tiling}} \res[s=\pole]{\gzT[\tail](\ge,s)};
  \end{equation}
  i.e., equation \eqref{eqn:tail-term-assertion} holds without error term for $\ge<\min\{\genir, A^{-1} \genir\}$. 
  This completes the proof for the strongly languid case, and thus of all of Theorem~\ref{thm:pointwise-tube-formula}. 

\subsubsection{Proof of Corollary~\ref{thm:fractal-spray-tube-formula-monophase}}
\label{sec:proof-of-cor-monophase}
  In the monophase case, we have $f_k(\ge) = 0$ for all $\ge > 0$ and for each $k=0,1,\dots,d$. Therefore, $\gzT[\head]$ (given in \eqref{eqn:gzThead}) vanishes identically, implying $\R_{\he}\equiv 0$ (by \eqref{eqn:pointwise-error-head}) and thus $V_{\mathrm{head}}(\tiling,\ge)=0$ for each $\ge>0$ (by~\eqref{eqn:head-term-assertion}). Consequently, by \eqref{eqn:def:V-tail}, $V(\tiling,\mydot)=V_{\mathrm{tail}}(\tiling,\mydot)$.  Since $\R_{\ta}(\ge)=\R(\ge)$, cf.~\eqref{eqn:pointwise-error} and \eqref{eqn:pointwise-error-head-tail}, the assertion of Corollary~\ref{thm:fractal-spray-tube-formula-monophase} follows by observing that the assumption $S(0)<0$ is not used in the proof of \eqref{eqn:tail-term-assertion} given in Section~ \ref{subsec:proof-tail} and Section~ \ref{subsec:proof-error-term}. It is only used to ensure that the screen $S$ avoids the integer dimensions $0,1,\ldots,d$. Note that $\gzT[\head]\equiv 0$ also implies $\gzT=\gzT[\tail]$. Hence, the error term $\R(\ge)$ is equivalently given by the integral \eqref{eqn:pointwise-error2} in this case, as explained in Remark~\ref{rem:error-term}.    
\hfill $\square$

\subsection{Proof of the fractal tube formula, Corollary~\ref{thm:ptwise-result-self-similar-case-simplified}}
\label{sec:proof-of-fractal-tube-formula}
\hfill \\
Before proceeding, we need to compute some residues. To this end, we introduce the tubular zeta function for the generator \gen. In addition to being a useful technical device, it reveals the structure of the residues of the tubular zeta function \gzT.

\begin{defn}\label{def:gzG}
  Let $\gzG(\ge,s)$ denote the \emph{tubular zeta function of the generator} \gen, where \gen is assumed to have a Steiner-like representation as in \eqref{eqn:def-prelim-Steiner-like-formula}. 
It is defined exactly as $\gzT(\ge,s)$, except that the associated fractal string is given by $\{\genstr_j\}_{j=1}^\iy$ with $\genstr_1 = 1$ and $\genstr_j = 0$ for all $j \geq 2$. In other words, it is the tubular zeta function of the trivial fractal spray with generator \gen.
\end{defn}

Exactly as in \eqref{eqn:volume-zeta-split2}, we write $\gzG = \gzG[\head] + \gzG[\tail]$, so that for $s \in \bC$, we have
  \linenopax
  \begin{align}\label{eqn:gzG[head]}
    \gzG[\head](\ge,s) = 
    \begin{cases}
      \ge^{d-s} \sum_{k=0}^d \frac{\genir^{s-k}}{s-k} f_k(\ge),
      & 0 < \ge \leq \genir, \\
      0, & \ge \geq \genir,
    \end{cases}
  \end{align}
  with $f_k(\ge) = \crv_k(\gen,\ge) - \crv_k(\gen)$ defined as in \eqref{eqn:def:f_k} for $k=0,1,\dots,d$, and
  \linenopax
  \begin{align}\label{eqn:gzG[tail]}
    \gzG[\tail](\ge,s) = \frac{\ge^{d-s}}{d-s} M_s(\gen),
    \qq \ge > 0,
  \end{align}
  where
  \linenopax
  \begin{align}\label{eqn:M_s(G)}
    M_s(\gen) := \sum_{k=0}^{d-1} \frac{\genir^{s-k}}{s-k} (d-k)\crv_k(\gen).
  \end{align}
  To see why the second case of \eqref{eqn:gzG[head]} should be true, consider the definition of \gzG[\head] in the counterpart of \eqref{eqn:volume-zeta-split2} and suppose we define 
  \linenopax
  \begin{align}\label{eqn:J_G(e)}
    J_G(\ge) := \charfn{(0,\genir)}(\ge),
  \end{align}
  which is in parallel to \eqref{eqn:J(eps)}, upon inspection.
   
  Observe that for every $\ge > 0$, $\gzG(\ge,\mydot)$ is meromorphic in all of \bC, with poles contained in $\{0,1,\dots,d\}$. Hence, the set of ``complex dimensions'' of \gen consists of the integer dimensions $\{0,1,\dots,d\}$, and all of these poles are simple.
  
\begin{lemma}[Residues of \gzG]
  \label{thm:gzG-residues}
  For $0 < \ge \leq \genir$, we have the following residues of \gzG:
  \linenopax
  \begin{align}
    &\res[s=k]{\gzG[\head](\ge,s)} = \ge^{d-k} f_k(\ge),
      && k=0,1,\dots,d, 
      \label{eqn:gzG[head]-residue-at-k} \\
    &\res[s=k]{\gzG[\tail](\ge,s)} = \ge^{d-k} \crv_k(\gen),
      && k=0,1,\dots,d-1, 
      \label{eqn:gzG[tail]-residue-at-k} \\
    &\res[s=k]{\gzG[\tail](\ge,s)} = \crv_d(\gen) - \gl_d(\gen),&& k=d.
      \label{eqn:gzG[tail]-residue-at-d}
  \end{align}
  \begin{proof}
    In light of \eqref{eqn:gzG[head]}--\eqref{eqn:M_s(G)}, each of $\gzG[\head](\ge,\mydot)$, $\gzG[\tail](\ge,\mydot)$ and $\gzG(\ge,\mydot)$ is meromorphic in all of \bC, with (simple) poles contained in $\{0,1,\dots,d\}$, for $0 < \ge \leq \genir$. To show \eqref{eqn:gzG[tail]-residue-at-d}, simply use \eqref{eqn:gzG[tail]} and \eqref{eqn:M_s(G)} to compute
    \linenopax
    \begin{align*}
      \res[s=d]{\gzG[\tail](\ge,s)}
      &= \lim_{s \to d} \, (s-d) \gzG[\tail](\ge,s) 
      = - \sum_{k=0}^{d-1} \genir^{d-k} \crv_k(\gen) 
      = \crv_d(\gen) - \gl_d(\gen),
    \end{align*}
    using \eqref{eqn:leb-coeff-rel} to reach the last equality. 
  \end{proof}
\end{lemma}

The following result will not be used in the sequel but may be helpful for the reader; it provides a ``residue formulation'' of the given Steiner-like representation of \gen.

\begin{cor}[Exact tube formula for \gen]\label{thm:Exact-tube-formula-for-G}
  For all $\ge \in (0,\genir]$,
  \linenopax
  \begin{align}\label{eqn:residue-tube-formula-for-G}
    V(\gen,\ge) = \sum_{k=0}^d \res[s=k]{\gzG(\ge,s)} + \gl_d(\gen).
  \end{align}
  \begin{proof}
    First, note that it follows from \eqref{eqn:gzG[head]-residue-at-k} and \eqref{eqn:def:f_k} that for $0 < \ge \leq \genir$,
    \linenopax
    \begin{align}
      \sum_{k=0}^d \res[s=k]{\gzG[\head](\ge,s)}
      &= \sum_{k=0}^d \ge^{d-k} \left(\crv_k(\gen,\ge) - \crv_k(\gen)\right) \notag \\
      &= \sum_{k=0}^d \ge^{d-k} \crv_k(\gen,\ge) - \sum_{k=0}^d \ge^{d-k} \crv_k(\gen) \notag \\
      &= V(\gen,\ge) - \sum_{k=0}^{d-1} \ge^{d-k} \crv_k(\gen) - \crv_d(\gen),
      \label{eqn:residue-tube-formula-for-G-derivation1}
    \end{align}
    where we have used \eqref{eqn:def-prelim-Steiner-like-formula} in the last equality. 
    Furthermore, by \eqref{eqn:gzG[tail]-residue-at-k} and \eqref{eqn:gzG[tail]-residue-at-d}, we have for $\ge \in (0,\genir]$,
    \linenopax
    \begin{align}
      \sum_{k=0}^d \res[s=k]{\gzG[\tail](\ge,s)}
      = \sum_{k=0}^{d-1} \ge^{d-k} \crv_k(\gen) + \left(\crv_d(\gen)- \gl_d(\gen)\right).
      \label{eqn:residue-tube-formula-for-G-derivation2}
    \end{align}
    Since $\gzG = \gzG[\head] + \gzG[\tail]$, the result now follows by adding \eqref{eqn:residue-tube-formula-for-G-derivation1} and \eqref{eqn:residue-tube-formula-for-G-derivation2}.
  \end{proof}
\end{cor}

As an alternative proof of Corollary~\ref{thm:Exact-tube-formula-for-G}, one can  obtain \eqref{eqn:residue-tube-formula-for-G} by applying the second part of Theorem~\ref{thm:pointwise-tube-formula} to the trivial fractal spray on \gen. However, we feel that the proof given above is more edifying and more straightforward.

\subsubsection{The residues of \gzT}
\label{sec:residues-of-gzT}

Let \tiling be a self-similar tiling with a fractal string $\sL = \{\scale_j\}_{j=1}^\iy$ and a single generator \gen for which a Steiner-like representation has been fixed. Let $\gzT = \gzT(\ge,s)$ denote the tubular zeta function of \tiling, and let $\gzT = \gzT[\head] + \gzT[\tail]$ be its head--tail decomposition, as in Section~\ref{sec:splitting}. In light of \eqref{eqn:volume-zeta-split2}, we deduce from \eqref{eqn:gzG[tail]} and \eqref{eqn:M_s(G)} that $\gzT[\tail]$ factors as follows:
  \linenopax
  \begin{align}\label{eqn:gzT[tail]-factored}
    \gzT[\tail](\ge,s) = \gzG[\tail](\ge,s) \gzL(s).
  \end{align}
Furthermore, still by \eqref{eqn:volume-zeta-split2}, 
  \linenopax
  \begin{align}\label{eqn:gzT[head]-notfactored}
    \gzT[\head](\ge,s)
    = \ge^{d-s} \sum_{k=0}^d \frac{\genir^{s-k}}{s-k} 
      \sum_{j=1}^{J(\ge)} \scale_j^s f_k(\scale_j^{-1} \ge),
  \end{align}
  with $f_k$ as in \eqref{eqn:def:f_k} and $J(\ge)$ as in \eqref{eqn:J(eps)}, as usual. 
  Recall that $J(\ge) \to \iy$ monotonically as $\ge \to 0^+$, since $\scale_j$ decreases monotonically to $0$ as $j \to \iy$.
  
\begin{remark}\label{rem:gzT[head]-notfactored}
  Observe that \eqref{eqn:gzT[head]-notfactored} is not at all the counterpart of the factorization given in \eqref{eqn:gzT[tail]-factored}. Indeed, it clearly does not enable us to write $\gzT[\head]$ as the product of $\gzG[\head]$ and \gzL (which would be false). This is the source of some difficulty if we wish to estimate the residues of $\gzT[\head](\ge,s)$ as $\ge \to 0^+$. 
\end{remark}
  

\begin{lemma}[Residues of \gzT]
  \label{thm:residues-of-gzT}
  Fix $\ge \in (0,\gen]$. Then:
  \begin{enumerate}
    \item[\emph{(i)}] When $\pole \in \sD_\sL \less \{0,1,\dots,d\}$ is a simple pole of \gzL, the residue $\res{\gzT(\ge,s)}$ is given by
      \linenopax
      \begin{align}\label{eqn:residues-of-gzT(i)}
        \gzG[\tail](\ge,\pole) \res{\gzL(s)}
        = \frac{\ge^{d-\pole}}{d-\pole} \res{\gzL(s)} M_\pole(\gen),
      \end{align}
      with $M_\pole(\gen) = \sum_{k=0}^{d-1} \frac{\genir^{\pole-k}}{\pole-k}(d-k)\crv_k(\gen)$ as in \eqref{eqn:M_s(G)}.
    \item[\emph{(ii)}] For $\pole = k \in \{0,1,\dots,d\} \less \sD_\sL$,%
    \footnote{\label{foot}Note that if \tiling is a self-similar tiling, then the only real pole of \gzL is $\abscissa = \abscissa_\sL < d$. Hence, the only way that \pole could belong to both $\sD_\sL$ and $\{0,1,\dots,d\}$ would be if $\pole = \abscissa = k$, for some $k \in \{0,1,\dots,d-1\}$.} 
    the residue $\res[s=k]{\gzT[\tail](\ge,s)}$ is given by 
      \linenopax
      \begin{align}\label{eqn:residues-of-gzT(ii)}
        \res[s=k]{\gzT[\tail](\ge,s)} 
        = \begin{cases}
          \ge^{d-k} \crv_k(\gen) \gzL(k), & k \neq d, \\
          \left(\crv_d(\gen) - \gl_d(\gen) \right) \gzL(d), &k=d.
        \end{cases}
      \end{align}
    \item[\emph{(iii)}] For $\pole = k \in \{0,1,\dots,d\}$, the residue $\res[s=k]{\gzT[\head](\ge,s)}$ is given by
      \linenopax
      \begin{align}\label{eqn:residues-of-gzT(iii)}
        \res[s=k]{\gzT[\head](\ge,s)} 
        = \ge^{d-k} \sum_{j=1}^{J(\ge)} \scale_j^k f_k(\scale_j^{-1} \ge),
      \end{align}
      with $J(\ge)$ as in \eqref{eqn:J(eps)} and $f_k$ as in \eqref{eqn:def:f_k}.
  \end{enumerate}
  \begin{proof}
    In light of the factorization formula \eqref{eqn:gzT[tail]-factored}, 
    \eqref{eqn:residues-of-gzT(i)} follows from \eqref{eqn:gzG[tail]} and the fact that, under the assumption of (i),
    \linenopax
    \begin{align*}
      \res{\gzT(\ge,s)} &= \res{\gzT[\tail](\ge,s)},    
      \qq \text{for } \pole \in \DL \less \{0,1,\dots,d\},
    \end{align*}
    while \eqref{eqn:residues-of-gzT(ii)} follows from \eqref{eqn:gzG[tail]-residue-at-k}--\eqref{eqn:gzG[tail]-residue-at-d} of Lemma~\ref{thm:gzG-residues}. Note that \pole is a simple pole of \gzL in part (i), and hence it is at most a simple pole of $\gzT[\tail](\ge,s)$; whence
    \linenopax
    \begin{align}\label{eqn:gzTtail-res-at-w}
      \res{\gzT[\tail](\ge,s)} 
      = \lim_{s \to \pole} (s-\pole) \gzT[\tail](\ge,s)
      = \gzG[\tail](\ge,\pole) \res{\gzL(s)},
    \end{align}
    from which \eqref{eqn:residues-of-gzT(i)} follows in light of \eqref{eqn:gzG[tail]}.
    Since $\pole = k \in \{0,1,\dots,d\}$ is a simple pole of $\gzT[\head]$, 
    \linenopax
    \begin{align*}
      \res[s=k]{\gzT[\head](\ge,s)} = \lim_{s \to k} \, (s-k) \gzT[\head](\ge,s),
    \end{align*}
    and hence \eqref{eqn:residues-of-gzT(iii)} follows immediately from \eqref{eqn:gzT[head]-notfactored}.
    Finally, as was already observed, the poles of $\gzT[\head]$ and $\gzT[\tail]$ belong to $\{0,1,\dots,d\}$ and $\sD_\tiling$, respectively.
  \end{proof}
\end{lemma}

\begin{remark} \label{rem:residues-of-gzT} 
Note that Lemma~\ref{thm:residues-of-gzT} is valid for an arbitrary  fractal spray satisfying the hypotheses of the first part of Theorem~\ref{thm:pointwise-tube-formula}, but without the assumption that $S(0)<0$ (which is not necessary for Lemma~\ref{thm:residues-of-gzT} to hold).
\end{remark}

\subsubsection{The proof of Corollary~\ref{thm:ptwise-result-self-similar-case-simplified}}
  \label{sec:proof-of-some-corollary}
  Let \tiling be a self-similar tiling satisfying the hypotheses of Corollary~\ref{thm:ptwise-result-self-similar-case-simplified}. 
  Note that since the poles of \gzT are assumed to be simple, it follows that $\sD_\sL$ and $\{0,1,\dots,d\}$ are disjoint; that is, all poles of \gzL are simple and $D \notin \{1,\dots,d-1\}$. 
Recall that since \tiling is a self-similar tiling, we have $0 < \abscissa < d$ and \abscissa is the only pole of \gzL on the real axis. See footnote 6. 
The following proof makes use of the decomposition $\gzT = \gzT[\head] + \gzT[\tail]$ from \eqref{eqn:volume-zeta-split2}.
    Since $\sD_\tiling = \sD_\sL \cup \{0,1,\dots,d\}$ is a disjoint union, the present hypotheses and Theorem~\ref{thm:ptwise-result-self-similar-case} yield
    (for $\ge\in(0,g)$ and with $e_k(\ge)$ defined as in \eqref{eqn:self-similar-pointwise-tube-formula-coefficients-ek}, for $k=0,1,\dots,d$)
    \linenopax
    \begin{align}      
    V(\tiling,\ge)
      &= \sum_{\pole \in \sD_\sL} \res[\pole]{\gzT[\tail](\ge,s)}
       + \sum_{k \in \{0,1,\dots,d\}} \res[k]{\gzT[\tail](\ge,s)} 
       \notag \\
       &\hstr[15]+ \sum_{k \in \{0,1,\dots,d\}} \res[k]{\gzT[\head](\ge,s)} + \gl_d(\gen) \gzL(d) 
       \label{eqn:fractal-tube-formula-derivation1} \\
      &= \sum_{\pole \in \sD_\sL} \frac{\ge^{d-\pole}}{d-\pole} M_\pole(\gen) \res[\pole]{\gzL(s)}
       + \sum_{k=0}^{d-1} \ge^{d-k} \crv_k(\gen) \gzL(k) 
       + \left(\crv_d(\gen) - \gl_d(\gen)\right)\gzL(d) \notag \\
       &\hstr[15]+ \sum_{k=0}^d \ge^{d-k} e_k(\ge) + \gl_d(\gen)\gzL(d),
       \label{eqn:fractal-tube-formula-derivation2}
    \end{align}
    from which \eqref{eqn:self-similar-pointwise-tube-formula-simplified} follows. 
    In \eqref{eqn:fractal-tube-formula-derivation2}, we have set 
    \linenopax
    \begin{align*}
      M_\pole(\gen) = \sum_{k=0}^{d-1} \frac{\genir^{\pole-k}}{\pole-k}(d-k) \crv_k(\gen)
    \end{align*}
    as in \eqref{eqn:M_s(G)}, so that 
    $c_\pole = M_\pole(\gen) \res{\gzL(s)}/(d-\pole)$ for $\pole \in \sD_\sL$, and applied Lemma~\ref{thm:residues-of-gzT} to obtain the precise values of the residues of $\gzT[\head]$ and $\gzT[\tail]$; 
    see \eqref{eqn:residues-of-gzT(i)}--\eqref{eqn:residues-of-gzT(iii)}. 
    In particular, this verifies \eqref{eqn:self-similar-pointwise-tube-formula-coefficients-cw}.
    Note also that the residue of $\gzT[\tail](\ge,s)$ at $s=d$ and the term $\gl_d(\gen) \gzL(d)$ have been combined to yield $\crv_d(\gen) \gzL(d) = c_d$. This verifies \eqref{eqn:self-similar-pointwise-tube-formula-coefficients-ck}.
    Note that the expression of $c_k$ given in \eqref{eqn:self-similar-pointwise-tube-formula-coefficients-ck} for $k \in \{0,1,\dots,d-1\}$ follows from the second sum in \eqref{eqn:fractal-tube-formula-derivation2}.

\section{Concluding remarks and future directions}
\label{sec:conclusion}

\subsection{Relation to previous results.}
\label{sec:consistency} 
We will now discuss in more detail the consistency of our tube formula with the tube formulas for fractal sprays and strings previously obtained; see also Remark~\ref{rem:exact-formula-generalized}. 

\subsubsection{Comparison of the present pointwise results with the distributional results of \cite{TFCD}}  
Recall that in \cite{TFCD}, the tube formula obtained is only shown to hold distributionally, and only for fractals sprays with monophase generators. (For a discussion of how Theorem~\ref{thm:pointwise-tube-formula} extends results of \cite{TFCD} to generators which may not be monophase (or even pluriphase), see Remark~\ref{rem:monophase-and-pluriphase}.) 

For monophase generators $G$, the tubular zeta function $\gzT$ in Definition~\ref{def:volume-zeta-fn} simplifies to the zeta function appearing in \cite[Definition~7.1]{TFCD}, and consequently Corollary~\ref{thm:fractal-spray-tube-formula-monophase}, the monophase case of Theorem~\ref{thm:pointwise-tube-formula}, is precisely the pointwise analogue of \cite[Theorem~7.4]{TFCD}. We leave this as an exercise to the reader, with the following hints: 
\begin{enumerate}
  \item Note that the constant $\gk_d(G)$ has a different meaning in \cite[Eq.~(5.9)]{TFCD}, namely $\gk_d(G) = -\gl_d(G)$. In this paper, we have $\gk_d(G) = 0$ in the monophase case (cf. Remark~\ref{rem:monophase-and-pluriphase}) and $\gl_d(G)$ is kept as $\gl_d(G)$ in the formulas. 
  \item When one computes the residue of $\gzT(\ge,s)$ at $s=d$ in the version of \cite{TFCD}, a term appears which cancels the term $\gl_d(\gen)\gzL(d)$ in \eqref{eqn:pointwise-tube-formula}. 
\end{enumerate}
Note that for the earlier distributional results, the assumptions on the underlying fractal string $\sL$ are slightly weaker: the order of languidity is arbitrary, and fractal strings with $D=d$ are permitted for fractal sprays in $\bR^d$. The additional assumption $\abscissa < d$ on the abscissa of convergence of \gzL in Theorem~\ref{thm:pointwise-tube-formula} is necessary for the proof to hold and is similar to the assumption $\abscissa < 1$ in \cite[Theorem~8.7]{FGCD}. 
  Note that one always has $\abscissa \leq d$ for a fractal spray with finite total volume, as the latter is given by $\gzL(d)\gl_d(G)$.  
  Although it is easy to construct a fractal spray with $\abscissa = d$, it follows from Proposition~\ref{cor:OSC-dimension-d-implies-trivial} that a self-similar tiling cannot satisfy $\abscissa = d$; indeed, this would violate the nontriviality condition.

\subsubsection{Comparison with the 1-dimensional case}
To see that the pointwise tube formula for fractal strings in \cite[Theorem~8.7]{FGCD} is a special case of our tube formula for fractal sprays in Theorem~\ref{thm:pointwise-tube-formula}, let $\sT$ be a fractal spray in $\bR$, i.e., a geometric fractal string. Then the generator \gen is always a bounded open interval of length $2g$ ($g$ being the inradius of \gen) and with a (monophase) Steiner-like representation $V(G_{-\ge})=2\ge$, for $0<\ge\le g$, implying $\gk_0(G)=2$ and $\gk_1(G)=0$. The fractal string $\sL=\{\ell_1,\ell_2,\ldots\}$ of the scaling ratios generating $\sT$ corresponds to the fractal string $\tilde\sL=\{\tilde\ell_1,\tilde\ell_2,\ldots\}$ of the \emph{lengths} $\tilde\ell_j:=2g\ell_j$ of the intervals used in \cite{FGCD}, whence $\gz_{\tilde\sL}(s)=(2g)^s\gzL(s)$. 
Since we are in the monophase case, the tubular zeta function \gzT 
simplifies to 
\linenopax
\begin{align}
\gzT(\ge,s)
    &= 
      \ge^{1-s} \gzL(s) \left(
      \frac{2\genir^{s}}{s}- \frac{2\genir^{s}}{s-1}\right)
     = \gz_{\tilde\sL}(s) \frac{(2\ge)^{1-s}}{s(1-s)}\,,\notag 
\end{align}
which is precisely the function appearing in 
\cite[Thm 8.7]{FGCD}. Moreover, the complex dimensions at which the residues are taken also coincide, except for the two integer dimensions $0$ and $1$. However, the residue at $1$ cancels for the same reasons as in hint (ii) above, and one can show that the residue at $0$ 
appears in the tube formula \eqref{eqn:pointwise-tube-formula} if and only if $0\in W\setminus\sD_{\sL}(W)=W\setminus\sD_{\tilde\sL}(W)$, just as in \cite[Thm~8.7]{FGCD}.
Finally, we remark that, in the setting of geometric fractal strings, the hypotheses of Corollary~\ref{thm:fractal-spray-tube-formula-monophase}  are exactly the same as in \cite[Thm~8.7]{FGCD}.

\subsection{Origin of the terms in the tube formula} 
\label{rem:proof-of-fractal-tube-formula}
  The proof of Corollary~\ref{thm:ptwise-result-self-similar-case-simplified} (given in Section~\ref{sec:proof-of-some-corollary}) explains the origin of each term in the exact tube formula \eqref{eqn:self-similar-pointwise-tube-formula-simplified}. Indeed, in \eqref{eqn:fractal-tube-formula-derivation1}--\eqref{eqn:fractal-tube-formula-derivation2}, the first and second sum express the contribution of the tail zeta function $\gzT[\tail](\ge,\mydot)$ at the scaling and integer dimensions of \tiling, respectively, while the third sum 
  expresses the contribution of the residues of the head zeta function $\gzT[\head](\ge,\mydot)$ at the integer dimensions.

\subsection{The monophase case} \label{rem:monophase-gzG}
  Note that if \gen is monophase, its coefficient functions $\crv_k(\gen,\ge)$ are constant (and equal to $\crv_k(\gen)$). Consequently, the functions $f_k$ in \eqref{eqn:def:f_k} vanish identically, and hence so does $\gzT[\head](\ge,s)$ in \eqref{eqn:gzT[head]-notfactored}. As a result, one has $\gzT = \gzT[\tail]$, which is the case treated in \cite{TFCD}. This is so, in particular, when $d = 1$ and \gen is a bounded interval (i.e., in the case of a fractal string). 
  As a result, the contributions of the residues of the head tubular zeta function $\gzT[\head]$ vanish identically and thus do not have to be taken into account. 
  Note that in the monophase case, one must also have $\lim_{\ge\to 0^+} \crv_d(\gen,\ge) = 0$, and hence $\crv_d(\gen,\ge) = 0$ for all $0 < \ge \leq \genir$; see also the discussion of the monophase and pluriphase case in Remark~\ref{rem:monophase-and-pluriphase}. Consequently, this explains why the \nth[d] term drops out of the fractal tube formula appearing in \cite{TFCD}.

\subsection{The general case} \label{rem:0-in-W-redux}
  In Remark~\ref{rem:0-in-W}, it was observed that it is extremely useful to be able to drop the assumption that $S(0)<0$, especially for investigating delicate questions concerning the Minkowski measurability of fractal sprays and self-similar tilings. Indeed, by analogy with \cite[Section~8.4]{FGCD}, the proof of such a result requires screens lying arbitrarily close to the line $\Re s = \abscissa$.
  If in the languid case of Theorem~\ref{thm:pointwise-tube-formula} one drops the requirement that $S(0)<0$, then the tube formula in \eqref{eqn:pointwise-tube-formula} still holds, provided the error term $\R(\ge)$ given in \eqref{eqn:pointwise-error} (or equivalently, given by $\R_{\ta}(\ge)$ in \eqref{eqn:pointwise-error-head-tail}) is replaced by
  \linenopax
  \begin{align}\label{eqn:error-correction-for-S(0)>0}
    \R(\ge) =\R_{\ta}(\ge) + \R_{\he}(\ge)\,,
  \end{align}
  with $\R_{\he}(\ge)$ as in \eqref{eqn:pointwise-error-head}.
However, while the estimate $O(\ge^{d-\sup S})$ as $\ge \to 0^+$ remains true for $\R_{\ta}(\ge)$, it will not be satisfied in general for $\R_{\he}(\ge)$, the sum of the residues of $\gzT[\head]$ over the hidden integer dimensions. Hence, such a tube formula would be rather useless, as its error term may be of the same order as (or even dominate) its `main term'. 
As a result, this generalization of Theorem~\ref{thm:pointwise-tube-formula} would not be suitable for investigating the Minkowski measurability of fractal sprays or even of self-similar tilings.

In fact, the assumption $S(0)<0$ should be seen as the price one has to pay for the generality of the allowed Steiner-like representations.
Stronger hypotheses on the generator \gen (or on the coefficients in the Steiner-like representation) will lead to better estimates of the error term and thus allow one to drop this assumption, as in the monophase case, and to extend the results on Minkowski measurability mentioned in Remark~\ref{rem:0-in-W} (and discussed in detail in \cite{MMappendix}) beyond the monophase setting. We plan to address this issue in \cite{MMFS}.

\subsection{Piecewise analytic Steiner-like representations}
  \label{rem:Piecewise-analytic-SLreps}
   In Example~\ref{exm:Cantor Carpet Tiling}, there is a partition of the interval $(0,\genir]$ into finitely many pieces, namely
   $  (0,\genir] = \left(0,\genir/\sqrt2\right]
     \cup \left(\genir/\sqrt2,\genir\right]$,
   such that each coefficient function $\crv_k(\gen,\ge)$ is analytic on the interior of each subinterval. That is, each $\crv_k(\gen,\ge)$ is continuous and given by an absolutely convergent power series in \ge (in the first subinterval) or $\frac1\ge$ (in the remaining subintervals).
   In such a case, we say that \gen has a \emph{piecewise analytic Steiner-like representation}. 
   
   This condition appears to be satisfied by many natural examples of fractal sprays (and self-similar tilings in particular). Indeed, it may be the key assumption needed to be able to apply our tube formulas efficiently to a wide variety of examples. In some future work, we plan to investigate this property further, especially with regard to associated Minkowski measurability results; see \cite{MMFS} and Section~\ref{rem:0-in-W-redux} just above.

\subsection{Fractal curvatures} \label{rem:search-for-fractal-curvatures}
  As was mentioned at the end of Section~\ref{sec:motivation}, a key motivation for the present work is the search for a good notion of fractal curvature. (See \cite[Section~8.2 and Section~12.7]{FGCD} for a discussion in the 1-dimensional case.) In our context, this would entail obtaining a local tube formula (with or without error term) corresponding to Theorem~\ref{thm:pointwise-tube-formula} (and its corollaries). This would lead naturally to an interpretation of the coefficients of such a local tube formula in terms of ``curvature measures'' (or rather, distributions) associated with each complex dimension (that is, with each scaling and integer dimension). We hope to explore such a possibility in future work 
  and to establish in the process some useful connections with \cite{Steffen:thesis} and some of the references (on geometric measure theory and differential geometry) discussed in Section~\ref{sec:motivation}; see \cite{Federer, Schneider, HuLaWe} and \cite{Weyl, BergerGostiaux,Gray}, in particular. Furthermore, we expect that eventually the present work and its ramifications will be helpful in obtaining global and local tube formulas (and an appropriate notion of fractal curvature), for more general fractal objects than fractal sprays and self-similar tilings.

\pgap

\bibliographystyle{math}
\bibliography{strings}


\end{document}